\documentclass[11pt,oneside,english]{amsart}
\usepackage[latin9]{inputenc}
\usepackage[a4paper]{geometry}
\usepackage{babel}
\usepackage{amstext}
\usepackage{amsthm}
\usepackage{amssymb}
\usepackage{setspace}
\usepackage{esint}
\usepackage[unicode=true,pdfusetitle,
 bookmarks=true,bookmarksnumbered=false,bookmarksopen=false,
 breaklinks=false,pdfborder={0 0 1},backref=false,colorlinks=false]
 {hyperref}

\makeatletter
\numberwithin{equation}{section}
\numberwithin{figure}{section}
\theoremstyle{plain}
\newtheorem{thm}{\protect\theoremname}[section]
\theoremstyle{remark}
\newtheorem{rem}[thm]{\protect\remarkname}
\theoremstyle{plain}
\newtheorem{cor}[thm]{\protect\corollaryname}
\theoremstyle{definition}
\newtheorem{problem}[thm]{\protect\problemname}
\theoremstyle{remark}
\newtheorem*{acknowledgement*}{\protect\acknowledgementname}
\theoremstyle{plain}
\newtheorem{lem}[thm]{\protect\lemmaname}
\theoremstyle{definition}
\newtheorem{defn}[thm]{\protect\definitionname}
\theoremstyle{plain}
\newtheorem{prop}[thm]{\protect\propositionname}

\DeclareMathOperator{\Hom}{Hom}
\DeclareMathOperator{\Mat}{Mat}

\makeatother

\providecommand{\acknowledgementname}{Acknowledgement}
\providecommand{\corollaryname}{Corollary}
\providecommand{\definitionname}{Definition}
\providecommand{\lemmaname}{Lemma}
\providecommand{\problemname}{Problem}
\providecommand{\propositionname}{Proposition}
\providecommand{\remarkname}{Remark}
\providecommand{\theoremname}{Theorem}

\begin{document}
\title{Fourier and small ball estimates for word maps on unitary groups}
\author{Nir Avni}
\email{avni.nir@gmail.com}
\address{Department of Mathematics\\
 Northwestern University\\
 2033 Sheridan Road, Evanston, IL 60208\\
 U.S.A.}
\author{Itay Glazer}
\email{itayglazer@gmail.com}
\address{Department of Mathematics\\
 University of Oxford \\
 Andrew Wiles Building, Radcliffe Observatory Quarter (550), Woodstock
Road, Oxford, OX2 6GG\\
 UK}
\author{Michael Larsen}
\email{mjlarsen@indiana.edu}
\address{Department of Mathematics\\
 Indiana University \\
 831 East Third Street, Bloomington, IN 47405\\
 U.S.A.}
\begin{abstract}
To a non-trivial word $w(x_{1},...,x_{r})$ in a free group $F_{r}$
on $r$ elements and a group $G$, one can associate the \emph{word
map }$w_{G}:G^{r}\rightarrow G$ that takes an $r$-tuple $(g_{1},...,g_{r})$
in $G^{r}$ to $w(g_{1},...,g_{r})$. If $G$ is compact, we further
associate the \emph{word measure} $\tau_{w,G}$, defined as the distribution
of $w_{G}(\mathsf{X}_{1},...,\mathsf{X}_{r})$, where $\mathsf{X}_{1},...,\mathsf{X}_{r}$
are independent and Haar-random elements in $G$. In this paper we
study word maps and word measures on the family of special unitary
groups $\left\{ \mathrm{SU}_{n}\right\} _{n\geq2}$.

Our first result is a small ball estimate for $w_{\mathrm{SU}_{n}}$.
We show that for every $w\in F_{r}\smallsetminus\left\{ 1\right\} $
there are $\epsilon(w),\delta(w)>0$ such that if $B\subseteq\mathrm{SU}_{n}$
is a ball of radius at most $\delta(w)\mathrm{diam}(\mathrm{SU}_{n})$
in the Hilbert\textendash Schmidt metric, then $\tau_{w,\mathrm{SU}_{n}}(B)\leq(\mu_{\mathrm{SU}_{n}}(B))^{\epsilon(w)}$,
where $\mu_{\mathrm{SU}_{n}}$ is the Haar probability measure.

Our second main result is about the random walks generated by $\tau_{w,\mathrm{SU}_{n}}$.
We provide exponential upper bounds on the large Fourier coefficients
of $\tau_{w,\mathrm{SU}_{n}}$, and as a consequence we show there
exists $t(w)\in\mathbb{N}$, such that $\tau_{w,\mathrm{SU}_{n}}^{*t}$
has bounded density for every $t\geq t(w)$ and every $n\geq2$, answering
a conjecture by the first two authors. As a key step in the proof,
we establish, for every large irreducible character $\rho$ of $\mathrm{SU}_{n}$,
an exponential upper bound of the form $\left|\rho(g)\right|<\rho(1)^{1-\epsilon}$,
for elements $g$ in $\mathrm{SU}_{n}$ whose eigenvalues are sufficiently
spread out on the unit circle in $\mathbb{C^{\times}}$.
\end{abstract}

\maketitle
\global\long\def\N{\mathbb{N}}%
\global\long\def\R{\mathbb{\mathbb{R}}}%
\global\long\def\Z{\mathbb{Z}}%
\global\long\def\val{\mathbb{\mathrm{val}}}%
\global\long\def\Qp{\mathbb{Q}_{p}}%
\global\long\def\Zp{\mathbb{\mathbb{Z}}_{p}}%
\global\long\def\ac{\mathbb{\mathrm{ac}}}%
\global\long\def\C{\mathbb{\mathbb{C}}}%
\global\long\def\Q{\mathbb{\mathbb{Q}}}%
\global\long\def\supp{\mathbb{\mathrm{supp}}}%
\global\long\def\VF{\mathbb{\mathrm{VF}}}%
\global\long\def\RF{\mathbb{\mathrm{RF}}}%
\global\long\def\VG{\mathbb{\mathrm{VG}}}%
\global\long\def\spec{\operatorname{Spec}}%
\global\long\def\Ldp{\mathbb{\mathcal{L}_{\mathrm{DP}}}}%
\global\long\def\sgn{\mathrm{sgn}}%
\global\long\def\id{\mathrm{Id}}%
\global\long\def\Sym{\mathrm{Sym}}%
\global\long\def\Vol{\mathrm{Vol}}%
\global\long\def\cyc{\mathrm{cyc}}%
\global\long\def\U{\mathrm{U}}%
\global\long\def\SU{\mathrm{SU}}%
\global\long\def\Wg{\mathrm{Wg}}%
\global\long\def\E{\mathbb{E}}%
\global\long\def\Irr{\mathrm{Irr}}%
\global\long\def\P{\mathbb{P}}%
\global\long\def\bh{\mathbf{h}}%
\global\long\def\Span{\operatorname{Span}}%
\global\long\def\pr{\operatorname{pr}}%
\global\long\def\sgn{\operatorname{sgn}}%
\global\long\def\sG{\mathsf{G}}%
\global\long\def\sW{\mathsf{W}}%
\global\long\def\sX{\mathsf{X}}%
\global\long\def\sY{\mathsf{Y}}%
\global\long\def\sZ{\mathsf{Z}}%
\global\long\def\sH{\mathsf{H}}%
\global\long\def\sV{\mathsf{V}}%
\global\long\def\sT{\mathsf{T}}%
\global\long\def\v{\mathsf{v}}%
\global\long\def\d{\mathsf{d}}%
\global\long\def\tr{\operatorname{tr}}%
\global\long\def\lct{\operatorname{lct}}%
\global\long\def\fraku{\mathfrak{u}}%

\raggedbottom

\section{Introduction}

Let $w$ be a non-trivial word in the variables $x_{1},\ldots,x_{r}$,
i.e., a non-identity element of the free group $F_{r}$ generated
by the set $\left\{ x_{i}\right\} $. We assume that $w$ is reduced
and denote its length by $\ell(w)$. Given a group $G$ and elements
$g_{1},\ldots,g_{r}\in G$, denote by $w_{G}(g_{1},\ldots,g_{r})$
the element of $G$ obtained by substituting $x_{i}\leftarrow g_{i}$
in $w$. In this paper, we study random $w$-images in $\SU_{n}$,
i.e., the $\SU_{n}$-valued random variable $w_{\SU_{n}}(\mathsf{G}_{1},\ldots,\mathsf{G}_{r})$,
where $\mathsf{G}_{1},\ldots,\mathsf{G}_{r}$ are independent, Haar-distributed
random elements in $\SU_{n}$.

Denote the distribution of $w_{\SU_{n}}(\mathsf{G}_{1},\ldots,\mathsf{G}_{r})$
by $\tau_{w,\SU_{n}}$. By \cite{Bor83}, the word map $w_{\SU_{n}}:\SU_{n}^{r}\rightarrow\SU_{n}$
is a submersion outside a proper subvariety. It follows that $\tau_{w,\SU_{n}}$
is absolutely continuous with respect to the Haar probability measure
$\mu_{\SU_{n}}$ of $\SU_{n}$. In general, the density\textemdash that
is, the Radon\textendash Nikodym derivative $\frac{\tau_{w,\SU_{n}}}{\mu_{\SU_{n}}}$\textemdash is
unbounded.

Our first result is an upper bound on the probability that $w_{\SU_{n}}(\mathsf{G}_{1},\ldots,\mathsf{G}_{r})$
belongs to a small ball. We equip $\SU_{n}$ with the bi-invariant
Riemannian metric $\gamma$ coming from the inner product $\langle A,B\rangle_{\widetilde{HS}}:=\frac{n}{4\pi^{2}\lfloor n/2\rfloor\lceil n/2\rceil}\tr(AB^{*})$
on the Lie algebra $\mathfrak{su}_{n}$. The normalization factor
is chosen so that the diameter of $\SU_{n}$ is 1\footnote{To see that $\gamma$ has diameter 1, see \cite[Table 4.2, line 1]{Yang};
note that the normalization chosen in this paper is $1/\sqrt{2}$
of the Killing metric, as in \cite[Chapter VII, Section 11, pp. 339--340]{Hel01}.}. We denote the $\gamma$-ball of radius $r$ around $g\in\SU_{n}$
by $B(g,r)$. 
\begin{thm}
\label{thm:small ball estimate}For every non-trivial word $w\in F_{r}$
there exist $\delta(w),\epsilon(w)>0$ such that for every $n$, every
$0<\delta<\delta(w)$, and every $g\in\SU_{n}$, 
\begin{equation}
\mathbb{P}\Big(w_{\SU_{n}}(\mathsf{G}_{1},\ldots,\mathsf{G}_{r})\in B(g,\delta)\Big)\leq\left(\mu_{\SU_{n}}(B(g,\delta))\right)^{\epsilon(w)}.\label{eq:small ball estimate}
\end{equation}
If $n>8\ell(w)$, one can take $\epsilon(w)=\frac{1}{256(\ell(w)+1)}$.
Otherwise, $\epsilon(w)=\frac{1}{2\cdot10^{4}\ell(w)^{3}}$ suffices. 
\end{thm}

\begin{rem}
Theorem \ref{thm:small ball estimate} also holds if $B(g,\delta)$
is replaced with $\widetilde{B}(g,\delta):=\left\{ h\in\SU_{n}:\left\Vert g-h\right\Vert _{\mathrm{\widetilde{HS}}}\leq\delta\right\} $.
This follows from the inequality $B(g,\delta)\subseteq\widetilde{B}(g,\delta)\subseteq B\left(g,\frac{\pi}{2}\delta\right)$
(see Lemma \ref{lem:compare to hilbert Schmidt norm}) and Corollary
\ref{cor:Estimates on volume of metric balls}.
\end{rem}

In the terminology of local dimension of measure \cite[(10.1)]{Fal97},
Theorem \ref{thm:small ball estimate} immediately implies: 
\begin{cor}
For every $1\neq w\in F_{r}$, the lower local dimension of $\tau_{w,\SU_{n}}$
is at least $\epsilon(w)(n^{2}-1)$. 
\end{cor}

Our next result concerns the random walk generated by $\tau_{w,\SU_{n}}$.
In \cite[Proposition 7.2]{AG}, two of us showed that there is $t(w,n)\in\N$
such that, for $t>t(w,n)$, the $t$-th self convolution $\tau_{w,\SU_{n}}^{*t}$
is in $L^{\infty}(\SU_{n})$, and asked \cite[Question 1.7]{AG} whether
$t(w,n)$ can be taken to be \textbf{independent of $n$} and whether
it has a polynomial dependence on the length $\ell(w)$. In this paper,
we show that the answer is positive:
\begin{thm}
\label{thm: A, bounded density}For every non-trivial word $w\in F_{r}$
and every $n$, 
\begin{enumerate}
\item $\tau_{w,\SU_{n}}^{*t}\in L^{\infty}(\SU_{n})$, for all $t>t_{1}(w):=5120\cdot\ell(w)^{3}$. 
\item If $n>8\ell(w)$, then $\tau_{w,\SU_{n}}^{*t}\in L^{\infty}(\SU_{n})$,
for all $t>t_{2}(w):=3\cdot2^{10}\cdot\ell(w)$. 
\end{enumerate}
\end{thm}

\begin{rem}
~\label{rem:remark on main theorem} 
\begin{enumerate}
\item The claim $\tau_{w,\SU_{n}}^{*t}\in L^{\infty}(\SU_{n})$ can be seen
as a weaker form of uniform $L^{\infty}$-mixing for the family of
random walks on the groups $\SU_{n}$ induced by the word measures
$\tau_{w,\SU_{n}}$. 
\item By analyzing the power word $w=x^{\ell}$ in \S\ref{subsec:Optimality-of-the},
we show that $t_{2}(w)$ grows at least linearly with $\ell(w)$.
In particular, in Theorem \ref{thm: A, bounded density}(2) we obtain
optimal dependence on $\ell(w)$, up to an absolute multiplicative
constant. 
\item The proof of Theorem \ref{thm: A, bounded density} also implies that
the density of $\tau_{w,\SU_{n}}^{*t}$ is bounded by $2^{Cn^{2}\log n}$
for some constant $C>0$ depending only on $w$. 
\end{enumerate}
\end{rem}

The proofs of Theorems \ref{thm:small ball estimate} and \ref{thm: A, bounded density}
use a non-concentration result for the spectrum of $w_{\SU_{n}}(\mathsf{G}_{1},\ldots,\mathsf{G}_{r})$
which we now state. For $\epsilon>0$, we say that an element $g\in\SU_{n}$
is \emph{$\left(\frac{1}{2},\epsilon\right)$-spread} if no arc of
diameter $2\epsilon$ in $\U_{1}$ contains more than $\frac{1}{2}n$
of the eigenvalues of $g$. The following is a special case of Theorem
\ref{thm:probabilistic result}: 
\begin{thm}
Let $w\in F_{r}$ be a non-trivial word, let $\epsilon>0$, and let
$n>\max\left\{ 8\ell(w),16\right\} $. Then, 
\[
\P\left(w_{\U_{n}}(\mathsf{G}_{1},\ldots,\mathsf{G}_{r})\text{ is not \ensuremath{\left(\frac{1}{2},\epsilon\right)}-spread }\right)<2^{3n^{2}}\epsilon^{\frac{n^{2}}{16(\ell(w)+1)}}.
\]
\end{thm}

Theorem \ref{thm: A, bounded density} is proved by bounding the Fourier
coefficients of the measures $\tau_{w,\SU_{n}}$. We denote the set
of irreducible characters of a group $G$ by $\Irr(G)$. In \cite[Conjecture 1.5]{AG},
the first two authors conjectured that for every non-trivial word
$w\in F_{r}$ there exists a constant $\epsilon=\epsilon(w)>0$ such
that for every $n$ and every $\rho\in\Irr(\mathrm{U}_{n})$, 
\begin{equation}
\left|\E(\rho(w_{\U_{n}}(\mathsf{G}_{1},\ldots,\mathsf{G}_{r})))\right|\le\rho(1)^{1-\epsilon}.\label{goal}
\end{equation}

This conjecture implies a stronger version of Theorem \ref{thm: A, bounded density},
namely, uniform bounds on the $L^{\infty}$-norms for the measures
$\tau_{w,\SU_{n}}^{*t(w)}$ as $n$ varies. In \cite[Theorem 1.1]{AG},
the conjecture was proved for the characters of the fundamental representations
of $\U_{n}$. Moreover, an analogous estimate is known to hold for
all sufficiently large finite simple groups, see \cite[Theorem 4]{LST19}.

The following result, which is the heart of this paper, is the major
step in the proof of Theorem \ref{thm: A, bounded density} and gives
evidence to the conjecture that the bounds (\ref{goal}) hold uniformly. 
\begin{thm}[Fourier bounds for high-degree characters]
\label{Thm B, large Fourier coefficients}For every non-trivial word
$w\in F_{r}$ there exist constants $C>0$ and $\epsilon>0$ such
that if $\rho\in\Irr(\U_{n})$ and $\rho(1)\geq2^{Cn^{2}\log n}$,
then (\ref{goal}) holds. 
\end{thm}

Additional evidence is given by the following theorems: 
\begin{thm}[Fourier bounds for low-degree characters]
\label{thm C:small Fourier coefficients}For every non-trivial $w\in F_{r}$
there exist constants $C,\alpha,\epsilon>0$ such that if $\rho\in\Irr(\U_{n})$
and $\rho(1)\leq2^{Cn^{\alpha}}$, then (\ref{goal}) holds. 
\end{thm}

\begin{thm}[Fourier bounds for power words]
\label{thm D:Fourier coefficients of power words}For every non-trivial
word $w=x^{\ell}$ in one variable there exists $\epsilon>0$ such
that for every $\rho\in\Irr(\U_{n})$, (\ref{goal}) holds. 
\end{thm}

\begin{rem}
It follows from the proof of Theorem \ref{thm D:Fourier coefficients of power words},
that $\epsilon$ can be taken to be $\frac{1}{\ell-1}-\delta$, for
$\delta$ arbitrarily small, if we further assume that $n\gg_{\delta,\ell}1$. 
\end{rem}

\begin{thm}[Fourier bounds for symmetric powers]
\label{thm:main thm symmetric powers}For every non-trivial $w\in F_{r}$
there exists $\epsilon>0$ such that for every $n\geq2$ and $m\geq1$,
\[
\mathbb{E}\left(\Big|\tr\left(\Sym^{m}\left(w(\mathsf{G}_{1},\ldots,\mathsf{G}_{r})\right)\right)\Big|^{2}\right)\leq\binom{n+m-1}{m}^{2(1-\epsilon)},
\]
where $\Sym^{m}(g)$ denotes the symmetric power of $g$. 
\end{thm}

\subsection{Methods and ideas of proof}

Let $\rho$ be an irreducible character of $\U_{n}$ whose degree
is sufficiently large in terms of $n$, and let $g$ be an element
of $\U_{n}$ for which $\log_{\rho(1)}\left|\rho(g)\right|$ is sufficiently
close to $1$. Theorem~\ref{thm:Exponential bound for separated}
asserts that in this setting, a large fraction of the eigenvalues
of $g$ must be very close to one another.

The idea behind the proof is that if the eigenvalues of $g$ are not
so distributed, then we can partition them into two large parts, of
sizes $k$ and $n-k$, in such a way that no pair of elements which
belong to distinct parts are very close together. Regarding $g$ as
an element of $\U_{k}\times\U_{n-k}$, this means no eigenvalue of
$g$ acting on $V_{n,k}:=\fraku_{n}/(\fraku_{k}\times\fraku_{n-k})$
is too close to $1$, which gives a lower bound on $\left|\xi(g)\right|$,
where $\xi$ is the character of the virtual $\U_{k}\times\U_{n-k}$
representation 
\[
\sum_{i=0}^{k(n-k)}(-1)^{i}\wedge^{i}V_{n,k}.
\]
We can express $\rho|_{\U_{k}\times\U_{n-k}}\cdot\xi$ as the difference
between two effective characters of $\U_{k}\times\U_{n-k}$ whose
total degree is small compared to $\rho(1)$, thanks to the Weyl dimension
formula, and thus get an upper bound on $\left|\rho(g)\xi(g)\right|$.
Combining this bound with the lower bound on $\left|\xi(g)\right|$
proves Theorem~\ref{thm:Exponential bound for separated}.

Now, an element $g$ for which a large fraction of the eigenvalues
of $g$ are very close to one another has many approximate eigenvectors,
meaning unit vectors $v$ for which the projection $\pr_{v^{\perp}}g(v)$
is very short. For $m/n$ bounded away from $0$, we consider the
probability that for $m$ independent random vectors $\{v_{1},\ldots,v_{m}\}$
on the unit sphere of $\C^{n}$, the $\pr_{v_{i}^{\perp}}(g(v_{i}))$
are all very short. If we can show that the probability of this is
sufficiently small when $g=w(\mathsf{G}_{1},\ldots,\mathsf{G}_{r})$
is a random $w$-image, it will follow that the bound of Theorem~\ref{thm:Exponential bound for separated}
almost always applies, and 
\[
\frac{\log\E(\left|\rho(w(\mathsf{G}_{1},\ldots,\mathsf{G}_{r}))\right|)}{\log\rho(1)}
\]
is bounded away from $1$, which, in turn, will imply (\ref{goal}).
This is how our main high energy result, Theorem~\ref{Thm B, large Fourier coefficients}
is proved.

The proof that $w(\mathsf{G}_{1},\ldots,\mathsf{G}_{r})$ is unlikely
to have many approximate eigenvectors is given in \S\ref{sec:Probability-bounds}.
Writing $w=x_{i_{\ell}}^{\epsilon_{\ell}}\cdots x_{i_{1}}^{\epsilon_{1}}$
as a product of letters and their inverses, the \emph{trajectory}
of a unit vector $v$ is the sequence 
\[
v\,,\,\mathsf{G}_{i_{1}}^{\epsilon_{1}}(v)\,,\,\mathsf{G}_{i_{2}}^{\epsilon_{i_{2}}}\mathsf{G}_{i_{1}}^{\epsilon_{1}}(v)\,,\,\ldots\,,\,\mathsf{G}_{i_{\ell}}^{\epsilon_{\ell}}\cdots\mathsf{G}_{i_{1}}^{\epsilon_{1}}(v)=w(\mathsf{G}_{1},\ldots,\mathsf{G}_{r})(v).
\]
We randomly choose $m$ unit vectors $v_{1},\ldots,v_{m}$ and, for
each $i$, consider the trajectory $v_{i}=h_{i,0},h_{i,1},\ldots,h_{i,\ell}=w(\mathsf{G}_{1},\ldots,\mathsf{G}_{r})(v_{i})$
of $v_{i}$. In this way we get a random array $h_{i,j}$ of unit
vectors. We show that, conditional on the values $h_{i,j}$ for $(i,j)\preceq(a,b)$
in the lexicographic order $\preceq$, the random vector $h_{a,b+1}$
is uniformly distributed in a sphere whose radius is length of the
projection of $h_{a,b}$ on the orthogonal complement of a certain
subset of the lexicographically previous vectors. By such analysis,
we bound the probability that, for all $i$, $h_{i,\ell}=w(\mathsf{G}_{1},\ldots,\mathsf{G}_{r})(v_{i})$
is close to a scalar multiple of $v_{i}$.

The argument presented above requires $n$ to be at least linear with
$\ell(w)$. In \S\ref{sec:Geometry-and-singularities}, we use algebro-geometric
methods developed by Glazer and Handel \cite{GH19,GH21,GHb} to deal
with the low rank case $n=O(\ell(w))$. Combining these results with
the large Fourier coefficient estimates of Theorem~\ref{Thm B, large Fourier coefficients}
for $n>C\ell(w)$ allows us to prove Theorems~\ref{thm: A, bounded density}
and \ref{Thm B, large Fourier coefficients} (See Sections \ref{subsec:Proof-of-Theorem B}
and \ref{subsec:Proof of Thm A}).

In \S\ref{sec:small ball}, we use our upper bound on the probability
that random $w$-images have many approximate eigenvalues (Theorem
\ref{thm:probabilistic result}), to obtain small ball estimates for
word maps on $\SU_{n}$ (Theorem \ref{thm:small ball estimate}).
In order to make Theorem \ref{thm:probabilistic result} fully compatible
with Theorem \ref{thm:small ball estimate}, we apply some estimates
on the volume of metric balls in the unitary group, using the Bishop\textendash Gromov
Volume Comparison Theorem (Lemmas \ref{lem:Bishop=002013Gromov} and
Corollary \ref{cor:Estimates on volume of metric balls}).

In \S\ref{sec:Fourier-coefficients-of power word}, we prove (\ref{goal})
for all characters when $w=x^{\ell}$ is a power word. We first deduce
from works of Rains \cite{Rai97,Rai03} that $\E(\rho(\mathsf{G}^{\ell}))=\dim\rho^{H_{n,\ell}}$,
where $H_{n,\ell}:=\U_{\left\lfloor n/\ell\right\rfloor +1}^{j}\times\U_{\left\lfloor n/\ell\right\rfloor }^{\ell-j}$
and $j\equiv n\text{ (mod \ensuremath{\ell})}$. Then, in order to
bound $\dim\rho^{H_{n,\ell}}$, we use an inductive argument, by first
restricting $\rho$ to $\U_{k}\times\U_{n-k}$ for $k\sim\frac{n}{\ell}\left\lfloor \frac{\ell}{2}\right\rfloor $.
A key property used in the induction step is the sphericity of $\mathrm{GL}_{n-k}(\C)\times\mathrm{GL}_{k}(\C)$
inside $\mathrm{GL}_{n}(\C)$, from which we deduce bounds on branching
multiplicities from $\U_{n}$ to $\U_{k}\times\U_{n-k}$ (see Lemma
\ref{lem:multiplicities in spherical pairs}). This analysis also
shows that, up to a multiplicative constant, the dependence of $\epsilon(w)$
on $w$ given in Theorem~\ref{Thm B, large Fourier coefficients}
is sharp.

In sections 8\textendash 9, we provide further evidence for (\ref{goal})
in low energy, using a variety of techniques (in particular, in \S\ref{sec:Fourier-estimates-for small reps},
the Weingarten calculus).

\subsection{\label{subsec:Related-work-and}Related work and further discussion}

\subsubsection{\label{subsec:Small-ball-estimates-and Fourier decay}Small ball
estimates and decay of Fourier coefficients of pushforward measures}

Word measures are a special case of pushforwards of smooth, compactly
supported measures by polynomial maps. For simplicity of presentation,
we consider the case of polynomial maps between affine spaces. 
\begin{problem}
\label{prob:General problem}Let $F$ be a local field, let $f:F^{n}\rightarrow F^{m}$
be a dominant polynomial map, and let $\mu$ be a smooth\footnote{If $F$ is non-Archimedean, smooth means locally constant.},
compactly supported measure on $F^{n}$. We ask the following 
\begin{enumerate}
\item \textbf{Decay of Fourier coefficients}: Can we find $C,\delta>0$,
such that $\left|\mathcal{F}(f_{*}\mu)(y)\right|<C\left|y\right|_{F}^{-\delta}$,
for every $y\in F^{m}$, where $\mathcal{F}$ denotes the Fourier
transform on $F^{m}$? 
\item \textbf{Small ball estimates}: Let $B(y,r)\subseteq F^{m}$ denotes
a ball of radius $r$ around $y\in F^{m}$. Can we find $C,\delta>0$,
such that: 
\begin{equation}
f_{*}\mu\left(B(y,r)\right)<C\left(\mathrm{Vol}B(y,r)\right)^{\delta}\text{ for all }r>0\text{ and }y\in F^{m}\,?\label{eq:small ball condition}
\end{equation}
\end{enumerate}
Bounds of the form (\ref{goal}) are non-commutative variants of Item
(1), replacing the Fourier transform in $F^{m}$ with the non-commutative
Fourier transform in $\SU_{n}$. 
\end{problem}

Problem \ref{prob:General problem} has been extensively studied in
many different contexts. It is closely related to the singularities
of the polynomial map $f$. When $m=1$, Problem \ref{prob:General problem}
has a positive answer (see e.g.~\cite{Igu78}, \cite[Section 1.4]{Den91a}
and \cite[Theorem 1.5]{CMN19} for the non-Archimedean case, and \cite[Chapter 7]{AGV88}
for the Archimedean case), and $\delta$ can be chosen, in both Items
(1) and (2), to be any number smaller than $\underset{y\in F}{\min}\lct(f-y)$,
where $\lct$ is a certain singularity invariant of polynomial maps
called the \emph{log-canonical threshold} (Definition \ref{def:log canonical threshold}).
When $m>1$, the exponent $\delta$ in Item (2) of Problem \ref{prob:General problem}
is still controlled by $\min_{y\in F^{m}}\lct(f-y)$ (see e.g.~\cite[Corollary 2.9]{VZG08}
and \cite[Theorem 4.12]{CGH23} for the non-Archimedean setting, and
\cite[Corollary 1]{LCVZ13} for the Archimedean setting). The Fourier
bound in Item (1), however, is more mysterious. We use the connection
between Problem \ref{prob:General problem} and the singularities
of $f$ in \S\ref{sec:Geometry-and-singularities} to prove Theorems
\ref{thm: A, bounded density} and \ref{thm:small ball estimate}
in bounded rank.

A natural variant of Problem \ref{prob:General problem} is to consider
families $\left\{ (f_{i},\mu_{i})\right\} _{i\in\N}$ of polynomial
maps $f_{i}:F^{n_{i}}\rightarrow F^{m_{i}}$ and measures $\mu_{i}$
on $F^{n_{i}}$, with $n_{i}\underset{i\rightarrow\infty}{\rightarrow}\infty$,
and ask whether one can take $C,\delta$ to be \textbf{dimension independent}
under suitable normalization on $f_{i}$ and $\mu_{i}$. Questions
such as this arise naturally in high-dimensional asymptotic geometry
\cite{Bou91,CCW99,NSV02,carbery2001distributional}, e.g.~in analyzing
volumes of polynomial sections of convex bodies. As an example of
one such result, if $\mu_{n}$ denotes the standard Gaussian measure
in $\R^{n}$, then \cite[Corollary 4]{GM22} implies that there is
a constant $C$ such that for every $d,n\in\N$ and every degree $d$-polynomial
map $f:\R^{n}\rightarrow\R$, where $f(x)=\sum_{I}a_{I}x^{I}$ and
$\sum_{\left|I\right|=d}a_{I}^{2}=1$, 
\begin{equation}
f_{*}\mu_{n}\left(J\right)\leq C\cdot d\cdot\left|J\right|^{\frac{1}{d}},\label{eq:dimension independent small ball}
\end{equation}
for every interval $J\subseteq\mathbb{R}$.

Given a word $w\in F_{r}$, the collection $\left\{ (w_{\SU_{n}},\mu_{\SU_{n}})\right\} _{n\in\N}$
fits into this high dimensional variant of Problem \ref{prob:General problem}
with the extra complication that the dimension of the target $m_{i}$
grows to infinity as well. Theorem \ref{thm:small ball estimate}
gives a dimension independent small ball estimate for balls of all
magnitudes (up to constant radius). We expect this phenomenon to generalize
to the family of all compact, connected, semisimple Lie groups.

Analogous questions can be asked for word maps on finite simple groups.
If $G$ is a finite group, an element $g\in G$ can be seen as a ``small
ball'' of volume $\left|G\right|^{-1}$, and the analog of (\ref{eq:small ball condition})
is $\left|w_{G}^{-1}(g)\right|<\left|G\right|^{r-\epsilon}$. Such
small ball estimates were obtained in \cite[Theorem 1.1]{LaS12},
and dimension independent Fourier estimates as in (\ref{goal}) were
obtained in \cite[Theorem 4]{LST19}. In the setting of compact $p$-adic
groups, small ball estimates were obtained for the family $\left\{ \mathrm{SL}_{n}(\Zp)\right\} _{n,p}$
in \cite[Theorem I]{GHb} with dimension independent exponent $\delta$
and dimension dependent constant $C$.

\subsubsection{\label{subsec:Fourier-coefficients-in}Fourier coefficients in the
stable range}

The study of the Fourier coefficients of $\tau_{w,\U_{n}}$ gets a
different flavor when one restricts to characters of small degrees,
i.e.~to characters $\rho$ for which $\rho(1)\leq n^{D}$ (for $D$
fixed). Up to a twist, such characters are determined by an ordered
pair of partitions with total sum at most $D$. In the case of the
pair consisting of the partition $\lambda=(1)$ and the empty partition,
the corresponding character $\rho_{n}$ of $\U_{n}$ is the trace
function. In \cite{MP19}, Magee and Puder have shown that, fixing
a pair of partitions, the function $n\mapsto\mathbb{E}\left(\rho_{n}\left(w_{\U_{n}}(\mathsf{G}_{1},\ldots,\mathsf{G}_{r})\right)\right)$
coincides with a rational function of $n$ for $n$ sufficiently large,
and bounded its degree in terms of the commutator length of $w$.
The coefficients of these rational functions produce interesting invariants
of words (see e.g.~\cite[Corollaries 1.8 and 1.11]{MP19}).

The above setting includes the low moments of the random variables
$\tr\left(w_{\U_{n}}(\mathsf{G}_{1},\ldots,\mathsf{G}_{r})\right)$
whose limiting distribution is studied in the context of Voiculescu's
free probability (see e.g.~\cite{DNV92,MS17}). Using the moment
method, it was shown in \cite{Voi91,Rud06,MSS07} that, for a fixed
$w\in F_{r}$, the sequence of random variables $\tr\left(w_{\U_{n}}(\mathsf{G}_{1},\ldots,\mathsf{G}_{r})\right)$
converges in distribution to a complex normal random variable. An
important tool used in those works is the \emph{Weingarten Calculus},
developed in \cite{Wein78,Col03,CS06}, which expresses integrals
on unitary groups as sums of \emph{Weingarten functions} over symmetric
groups. We use the Weingarten Calculus in the proof of Theorem \ref{thm C:small Fourier coefficients}
in \S\ref{sec:Fourier-estimates-for small reps}.
\begin{acknowledgement*}
We thank Emmanuel Breuillard, Noam Lifshitz, Michael Magee, Dan Mikulincer,
Doron Puder and Ofer Zeitouni for useful conversations. NA was supported
by NSF grant DMS\textendash 1902041 and BSF grant 2018201, IG was
supported by AMS\textendash Simons travel grant and partially supported
by BSF grant 2018201. ML was supported by NSF grant DMS\textendash 2001349.
\end{acknowledgement*}

\section{\label{sec:Preliminaries-in-the}Preliminaries in the representation
theory of the symmetric and unitary groups }

A \emph{partition} $\lambda$ of $m\in\N$, denoted $\lambda\vdash m$,
is a sequence $\lambda=(\lambda_{1},..,\lambda_{k})$ of non-negative
integers that sum to $m$, and $\lambda_{1}\geq\ldots\geq\lambda_{k}\geq0$.
Two partitions are equivalent if they differ only by a string of $0$'s
at the end. A partition $\lambda=(\lambda_{1},..,\lambda_{k})$, with
$\lambda_{k}>0$, is graphically encoded by a \emph{Young diagram},
which is a finite collection of boxes (or cells) arranged in $k$
left-justified rows, where the $j$-th row has $\lambda_{j}$ boxes.
It is convenient to write $\lambda=(1^{a_{1}}\cdots m^{a_{m}})$ for
the partition
\[
(\underset{a_{m}\text{ times}}{\underbrace{m,\ldots,m}},\ldots,\underset{a_{1}\text{ times}}{\underbrace{1,\ldots,1}})\vdash m.
\]
The \emph{length} $\ell(\lambda)$ of a partition $\lambda\vdash m$
is the number of non-zero parts $\lambda_{i}$, or equivalently the
number of rows in the corresponding Young diagram.

To each partition $\lambda\vdash m$, one can attach an irreducible
character $\chi_{\lambda}$ of $S_{m}$, and the map $\lambda\mapsto\chi_{\lambda}$
is a bijection from the set of partitions of $m$ to $\Irr(S_{m})$.
For each cell $(i,j)$ in the Young diagram of $\lambda$, the \emph{hook
length} $\mathrm{hook}_{\lambda}(i,j)$ is the number of cells $(a,b)$
in the Young diagram of $\lambda$ such that either $a=i$ and $b\geq j$,
or $a\geq i$ and $b=j$. The hook-length formula states that 
\begin{equation}
\chi_{\lambda}(1)=\frac{m!}{\prod_{(i,j)\in\lambda}\mathrm{hook}_{\lambda}(i,j)}.\label{eq:Hook length}
\end{equation}

If $\ell(\lambda)\leq n$, then one can further attach to $\lambda$
an irreducible character $\rho_{\lambda,n}$ of $\U_{n}$. The image
of the map $\lambda\mapsto\rho_{\lambda,n}$ in $\Irr(\U_{n})$ is
the set of \emph{polynomial characters of }$\U_{n}$. More generally,
$\Irr(\U_{n})$ is in bijection with the set $\Lambda_{n}$ of dominant
weights, 
\[
\Lambda_{n}:=\left\{ (\lambda_{1},\ldots,\lambda_{n}):\lambda_{1}\geq\ldots\geq\lambda_{n},\,\,\lambda_{i}\in\Z\right\} .
\]
For each $\lambda\in\Lambda_{n}$ we write $\rho_{\lambda,n}$ for
the corresponding irreducible character. Note that if $\lambda_{n}\geq0$
then $\lambda\in\Lambda_{n}$ can be identified with a partition of
$\sum_{i=1}^{n}\lambda_{i}$, so the notation for $\rho_{\lambda,n}$
is consistent. The irreducible representation $\bigwedge\nolimits ^{m}\C^{n}$
of $\U_{n}$ whose character is $\rho_{1^{m},n}$, is called the\emph{
$m$-th fundamental representation. }The highest weight $\varpi_{m}$
of $\rho_{1^{m},n}$ is called the \emph{$m$-th fundamental weight}.
Any $\lambda\in\Lambda_{n}$ can be written as a linear combination
$\lambda=\sum_{i=1}^{n}a_{i}\varpi_{i}$ with $a_{i}\in\N$ for $i\leq n-1$
and $a_{n}\in\Z$.

The following lemma relates the degrees of $\chi_{\lambda}$ and $\rho_{\lambda,n}$. 
\begin{lem}[{\cite[I.6, Exc. 6.4]{FH91}}]
\label{lemma:formula for dimension of rep}For each $\lambda=(\lambda_{1},\ldots,\lambda_{n})\vdash m$,
we have: 
\begin{equation}
\rho_{\lambda,n}(1)=\frac{\chi_{\lambda}(1)\cdot\prod_{(i,j)\in\lambda}(n+j-i)}{m!}=\frac{\prod_{(i,j)\in\lambda}(n+j-i)}{\prod_{(i,j)\in\lambda}\mathrm{hook}_{\lambda}(i,j)},\label{eq:formula for dimension}
\end{equation}
where $(i,j)$ are the (row,column)-coordinates of the cells in the
Young diagram with shape $\lambda$. 
\end{lem}

\begin{lem}[{Weyl Dimension Formula, \cite[Theorem 6.3]{FH91}}]
\label{lem:Weyl Dimension Formula}Let $\lambda=(\lambda_{1},\ldots,\lambda_{n})\vdash m$.
Then: 
\begin{equation}
\rho_{\lambda,n}(1)=\prod_{1\leq i<j\leq n}\frac{\lambda_{i}-\lambda_{j}+j-i}{j-i}.\label{eq:Weyl dimension formula}
\end{equation}
\end{lem}

We next describe the branching rules from $\U_{n}$ to $\U_{m}\times\U_{n-m}$.
We need the following definition: 
\begin{defn}
\label{def:=00005Cmu extansion of a Young diagram}~
\begin{enumerate}
\item Fix a Young Diagram $\lambda$ and let $n\in\mathbb{N}$. An \emph{$n$-expansion}
of $\lambda$ is any Young diagram obtained by adding $n$ boxes to
$\lambda$ in such a way that no two boxes are added in the same column. 
\item Given a partition $\lambda=(\lambda_{1},\ldots,\lambda_{l_{1}})\vdash k$
and a partition $\mu=(\mu_{1},\ldots,\mu_{l_{2}})\vdash l$, a $\mu$-expansion
of $\,\lambda$ is defined to be a $\mu_{l_{2}}$-expansion of a $\mu_{l_{2}-1}$-expansion
of a $\cdots$ of a $\mu_{1}$-expansion of the Young diagram of $\lambda$.
For a $\mu$-expansion of $\lambda$, we label the boxes added in
the $\mu_{l_{j}}$-expansion by the number $j$ and order the boxes
lexicographically by their position, first from top to bottom and
then from right to left. We say that a $\mu$-expansion of $\lambda$
is \emph{strict} if, for every $p\in\{1,\ldots,l_{2}-1\}$ and every
box $t$, the number of boxes coming before $t$ that are labeled
$p$ is greater than or equal to the number of boxes coming before
$t$ that are labeled $(p+1)$. 
\end{enumerate}
\end{defn}

\begin{prop}[{{Littlewood\textendash Richardson rule, see e.g.~\cite[I.6, Eq. (6.7)]{FH91}}}]
\label{prop:Littlewood-Richardson-unitary}Let $\mu,\nu$ be partitions,
with $\ell(\mu),\ell(\nu)\leq n$. Then: 
\[
\rho_{\mu,n}\otimes\rho_{\nu,n}=\bigoplus_{\lambda:\ell(\lambda)\leq n}N_{\mu,\nu}^{\lambda}\rho_{\lambda,n},
\]
where $\{N_{\mu,\nu}^{\lambda}\}$ are the Littlewood\textendash Richardson
coefficients, namely, $N_{\lambda\mu\nu}$ is the number of ways the
Young diagram of the partition $\nu$ can be realized as a strict
$\mu$-expansions of $\lambda$. 
\end{prop}

\begin{prop}[{\cite[Theorem 9.2.3]{GW09}}]
\label{prop:branching to Un-k times Uk}Let $\rho_{\lambda,n}\in\mathrm{Irr}(\U_{n})$.
Then for all $k<n$: 
\[
\rho_{\lambda,n}|_{\U_{k}\times\U_{n-k}}=\sum_{\mu,\nu:\ell(\mu)\leq k,\ell(\nu)\leq n-k}N_{\mu,\nu}^{\lambda}\rho_{\mu,k}\boxtimes\rho_{\nu,n-k}.
\]
\end{prop}

The following Lemma \ref{lem:dimension of small representations}
and Definition \ref{def:first and last fundamental weights} will
be used in Sections \ref{subsec:Fourier-coefficients-of power word}
and \ref{sec:Fourier-estimates-for small reps}. 
\begin{lem}
\label{lem:dimension of small representations}If $m,n\in\N$, and
$\mu\vdash m$ is a partition, then: 
\begin{enumerate}
\item $\rho_{\mu,n}(1)\geq\left(\frac{n}{m}\right)^{m}$, for all $m\leq n$. 
\item $\rho_{\mu,n}(1)\geq\min\left\{ 2^{\frac{m}{4}},2^{\frac{n}{2}}\right\} $,
whenever $\ell(\mu)\leq\frac{n}{2}$. 
\end{enumerate}
\end{lem}

\begin{proof}
By Lemma \ref{lemma:formula for dimension of rep}, and since $\mathrm{hook}_{\mu}(i,j)\leq m+1-i$,
we have: 
\[
\rho_{\mu,n}(1)=\frac{\prod_{(i,j)\in\mu}(n+j-i)}{\prod_{(i,j)\in\mu}\mathrm{hook}_{\mu}(i,j)}\geq\prod_{(i,j)\in\mu}\frac{n+j-i}{m+1-i}\geq\prod_{(i,j)\in\mu}\frac{n+1-i}{m+1-i}\geq\prod_{i=1}^{n}\left(\frac{n}{m}\right)^{\mu_{i}}=\left(\frac{n}{m}\right)^{m}.
\]
This proves Item (1). For Item (2), note that by (\ref{eq:Weyl dimension formula}),
since $\mu_{i}\geq\mu_{j}$ for $i<j$, and $\mu_{j}=0$ for $j>\frac{n}{2}$,
\begin{equation}
\rho_{\mu,n}(1)=\prod_{1\leq i<j\leq n}\frac{\mu_{i}-\mu_{j}+j-i}{j-i}\geq\prod_{1\leq i\leq\frac{n}{2}}\prod_{\frac{n}{2}<j\leq n}\frac{\mu_{i}+j-i}{j-i}\geq\prod_{1\leq i\leq\frac{n}{2}}\left(1+\frac{\mu_{i}}{n}\right)^{\frac{n}{2}}.\label{eq:lower bound for small partition}
\end{equation}
If $\mu_{1}\geq n$, (\ref{eq:lower bound for small partition}) is
bounded from below by $\left(1+\frac{\mu_{1}}{n}\right)^{\frac{n}{2}}\geq2^{\frac{n}{2}}$.
If $\mu_{1}<n$, we use the inequality $1+x>e^{x/2}$ (valid for $0\leq x<1$)
to get: 
\[
\rho_{\mu,n}(1)\geq\prod_{1\leq i\leq\frac{n}{2}}\left(1+\frac{\mu_{i}}{n}\right)^{\frac{n}{2}}\geq\prod_{1\leq i\leq\frac{n}{2}}e^{\frac{\mu_{i}}{4}}=e^{\frac{m}{4}}>2^{\frac{m}{4}}.\qedhere
\]
\end{proof}
\begin{defn}
~\label{def:first and last fundamental weights} 
\begin{enumerate}
\item Given $\lambda=\sum_{i=1}^{n}a_{i}\varpi_{i}\in\Lambda_{n}$, we set
$\lambda_{-}:=\sum_{i=1}^{\left\lfloor \frac{n}{2}\right\rfloor }a_{i}\varpi_{i}$
and $\lambda_{+}:=\sum_{i=\left\lfloor \frac{n}{2}\right\rfloor +1}^{n}a_{i}\varpi_{i}$,
so that $\lambda=\lambda_{+}+\lambda_{-}$. In particular, $\rho_{\lambda,n}\in\mathrm{Irr}(\U_{n})$
is embedded inside $\rho_{\lambda_{-},n}\otimes\rho_{\lambda_{+},n}$. 
\item Given a partition $\lambda\vdash m$ with $\ell(\lambda)\leq n$,
denote by $\lambda^{\vee,n}$ the partition $(-\lambda_{n},-\lambda_{n-1},\ldots,-\lambda_{1})+(\lambda_{1},\ldots,\lambda_{1})$.
In particular, $\rho_{\lambda,n}^{\vee}\simeq\rho_{\lambda^{\vee,n},n}\otimes\mathrm{det}^{-\lambda_{1}}$. 
\end{enumerate}
\end{defn}

\begin{rem}
\label{rem:useful fact}If $\lambda=\sum_{i=1}^{n}a_{i}\varpi_{i}\in\Lambda_{n}$,
then by Lemma \ref{lem:Weyl Dimension Formula}, $\rho_{\lambda,n}(1)\geq\rho_{\sum_{i=1}^{n}a'_{i}\varpi_{i},n}(1)$
if $a_{i}\geq a'_{i}$ for all $i\in[n]$. In particular, $\rho_{\lambda,n}(1)\geq\max\left\{ \rho_{\lambda_{-},n}(1),\rho_{\lambda_{+},n}(1)\right\} $. 
\end{rem}

\section{\label{sec:Character-bounds}Character bounds for well spread elements}

Let $\rho\in\mathrm{Irr}(\U_{n})$ with $\rho(1)>1$. Then $\left|\rho(g)\right|$
achieves its maximum $\rho(1)$ exactly for $g$ a scalar matrix.
In other words, $\left|\rho(g)\right|$ can be bounded away from $\rho(1)$
only if the eigenvalues of $g$ are not all contained in a small disk
in $\C$. This motivates the following definition: 
\begin{defn}
\label{def:separated and spread}Given $0<\beta<1$, $0<\gamma<\frac{1}{2}$
and $\epsilon>0$, we say that $g\in\U_{n}$ is: 
\begin{enumerate}
\item \emph{$(\gamma,\epsilon)$-separated} if the eigenvalues of $g$ can
be separated into two disjoint sets $Q,R$ such that $\left|Q\right|,\left|R\right|\geq\gamma n$
and every element of $Q$ is at distance at least $\epsilon$ from
every element of $R$. 
\item \emph{$(\beta,\epsilon)$-spread} if at most $(1-\beta)n$ of the
eigenvalues of $g$ lie in any arc of diameter $2\epsilon$ of $\U_{1}$. 
\end{enumerate}
\end{defn}

Our goal in this section is to show that if $\rho(1)$ is sufficiently
large, we can use a $(\gamma,\epsilon)$-separation condition to deduce
an exponential-type bound, $\left|\rho(g)\right|\le\rho(1)^{1-\delta}$.
The precise statement is Theorem~\ref{thm:Exponential bound for separated}
below. We further deduce character estimates for $(\beta,\epsilon)$-spread
elements, by showing that any $(\beta,\epsilon)$-spread element is
also $(\frac{\beta}{2},\frac{\epsilon}{n})$-separated (Lemma \ref{Lem: reduction from spread to separated}
and Corollary \ref{cor:character estimates for spread elements}).

We begin with an elementary inequality about dividing a finite metric
space (or, more generally, a pseudometric space, where distance zero
is allowed between distinct points) into two parts. If $Y$ and $Z$
are finite subsets of a pseudometric space $(X,d)$, we define 
\[
d(Y,Z):=\sum_{y\in Y}\sum_{z\in Z}d(y,z).
\]

\begin{lem}
\label{Metric}Let $X$ be a finite pseudometric space and $X=Y\coprod Z$
a partition of $X$ into two sets. Then 
\[
\frac{d(Y,Y)+d(Z,Z)}{d(X,X)}\le\frac{\left|Y\right|^{2}+\left|Z\right|^{2}}{\left|Y\right|^{2}+\left|Y\right|\left|Z\right|+\left|Z\right|^{2}}.
\]
\end{lem}

\begin{proof}
For all $(y_{1},y_{2},z)\in Y\times Y\times Z$, we have 
\[
d(y_{1},y_{2})\le d(y_{1},z)+d(y_{2},z).
\]
Summing over all such triples, we have 
\[
\left|Z\right|d(Y,Y)\le\left|Y\right|d(Y,Z)+\left|Y\right|d(Y,Z).
\]
Thus, 
\[
d(Y,Y)\le\frac{2d(Y,Z)\left|Y\right|}{\left|Z\right|},
\]
and likewise, 
\[
d(Z,Z)\le\frac{2d(Y,Z)\left|Z\right|}{\left|Y\right|}.
\]
These inequalities give 
\[
\frac{d(Y,Y)+d(Z,Z)}{d(X,X)-d(Y,Y)-d(Z,Z)}=\frac{d(Y,Y)+d(Z,Z)}{2d(Y,Z)}\le\frac{\left|Y\right|^{2}+\left|Z\right|^{2}}{\left|Y\right|\left|Z\right|},
\]
so 
\[
\frac{d(X,X)}{d(Y,Y)+d(Z,Z)}=1+\frac{d(X,X)-d(Y,Y)-d(Z,Z)}{d(Y,Y)+d(Z,Z)}\ge\frac{\left|Y\right|^{2}+\left|Y\right|\left|Z\right|+\left|Z\right|^{2}}{\left|Y\right|^{2}+\left|Z\right|^{2}},
\]
which implies the lemma. 
\end{proof}
\begin{lem}
\label{Products}Let $X$ be a finite set of integers and $X=Y\coprod Z$
a partition. Then: 
\[
\prod_{\substack{y_{1},y_{2}\in Y\\
y_{1}<y_{2}
}
}(y_{2}-y_{1})\prod_{\substack{z_{1},z_{2}\in Z\\
z_{1}<z_{2}
}
}(z_{2}-z_{1})<2^{\left|X\right|}\Bigl(\!\!\prod_{\substack{x_{1},x_{2}\in X\\
x_{1}<x_{2}
}
}(x_{2}-x_{1})\Bigr)^{\frac{\left|Y\right|^{2}+\left|Z\right|^{2}}{\left|Y\right|^{2}+\left|Y\right|\left|Z\right|+\left|Z\right|^{2}}}.
\]
\end{lem}

\begin{proof}
We define a metric on $X$ by setting 
\[
d(x_{1},x_{2}):=\ln\left(\max\left(2,\left|x_{2}-x_{1}\right|\right)\right),
\]
for all $x_{1}\neq x_{2}$. If $c(X)$ denotes the cardinality of
$\{x\in X\mid x+1\in X\}$, then 
\[
\exp(\frac{1}{2}d(X,X))=2^{c(X)}\prod_{\substack{x_{1},x_{2}\in X\\
x_{1}<x_{2}
}
}\left(x_{2}-x_{1}\right).
\]
By Lemma \ref{Metric}, if $\beta:=\frac{\left|Y\right|^{2}+\left|Z\right|^{2}}{\left|Y\right|^{2}+\left|Y\right|\left|Z\right|+\left|Z\right|^{2}}$,
we have: 
\begin{align*}
 & 2^{c(Y)+c(Z)}\prod_{\substack{y_{1},y_{2}\in Y\\
y_{1}<y_{2}
}
}(y_{2}-y_{1})\prod_{\substack{z_{1},z_{2}\in Z\\
z_{1}<z_{2}
}
}(z_{2}-z_{1})=\exp(\frac{1}{2}d(Y,Y)+\frac{1}{2}d(Z,Z))\\
\leq & \exp\left(\frac{\beta}{2}d(X,X)\right)=\left(2^{c(X)}\prod_{\substack{x_{1},x_{2}\in X\\
x_{1}<x_{2}
}
}\left(x_{2}-x_{1}\right)\right)^{\beta}.
\end{align*}
Since $\beta<1$ and $c(X)<\left|X\right|$, the lemma follows. 
\end{proof}
For integers $1\leq m<n$, let $G:=\U_{n}$, let $H:=\U_{m}\times\U_{n-m}\subseteq G$,
and let $T:=\U_{1}^{n}\subseteq H\subseteq G$ be a maximal torus
in both $H$ and $G$. We identify weights of $T$ with ordered $n$-tuples
of integers. Recall that a weight $(\lambda_{1},\ldots,\lambda_{n})\in\Z^{n}$
is dominant for $G$ if $\lambda_{1}\ge\lambda_{2}\ge\cdots\ge\lambda_{n}$,
and dominant for $H$ if $\lambda_{1}\ge\cdots\ge\lambda_{m}$ and
$\lambda_{m+1}\ge\cdots\ge\lambda_{n}$. Let $\delta_{G}:=(n-1,n-2,\ldots,0)$
and $\delta_{H}:=(m-1,\ldots,0,n-m-1,\ldots,0)$ denote the half sums
of positive roots of $G$ and $H$ respectively. Let $W_{G}=S_{n}$
and $W_{H}=S_{m}\times S_{n-m}$ denote the Weyl groups of $G$ and
$H$ respectively. For $\lambda$ a dominant weight of $G$ (resp.
$H$), we denote by $\rho_{\lambda,G}=\rho_{\lambda,n}$ (resp. $\rho_{\lambda,H}$)
the character of the irreducible representation with highest weight
$\lambda$. Finally, for every weight $\lambda=(\lambda_{1},\ldots,\lambda_{n})\in\Z^{n}$
we denote by $\xi_{\lambda}=\rho_{\lambda,T}$ the corresponding character
of $T$, so that $\xi_{\lambda}(t_{1},\ldots,t_{n})=t_{1}^{\lambda_{1}}\cdots t_{n}^{\lambda_{n}}$. 
\begin{prop}
\label{Dim ineq}Let $\lambda$ denote a dominant weight of $G$ and
$\mu$ a dominant weight of $H$, such that $\lambda+\delta_{G}$
and $\mu+\delta_{H}$ lie in the same $W_{G}$-orbit. Then: 
\[
\prod_{i=1}^{m-1}i{}^{m-i}\prod_{j=1}^{n-m-1}j{}^{n-m-j}\cdot\rho_{\mu,H}(1)\le2^{n}\Bigl(\rho_{\lambda,G}(1)\cdot\prod_{k=1}^{n-1}k^{n-k}\Bigr)^{\frac{n^{2}-2mn+2m^{2}}{n^{2}-mn+m^{2}}}.
\]
\end{prop}

\begin{proof}
If $\lambda=(\lambda_{1},\ldots,\lambda_{n})$, then the coordinates
of $\lambda+\delta_{G}=(\lambda_{1}+n-1,\lambda_{2}+n-2,\ldots,\lambda_{n})$
are all distinct; we define $X$ to be the set of coordinates of $\lambda+\delta_{G}$.
We define $Y$ to be the set consisting of the first $m$ coordinates
of $\mu+\delta_{H}$ and $Z$ the set consisting of the last $n-m$.
Since $\lambda+\delta_{G}$ and $\mu+\delta_{H}$ are in the same
$W$-orbit, $X=Y\coprod Z$. By the Weyl dimension formula (Lemma
\ref{lem:Weyl Dimension Formula}), 
\[
\rho_{\mu,H}(1)=\frac{\prod\limits _{\substack{y_{1},y_{2}\in Y\\
y_{1}<y_{2}
}
}(y_{2}-y_{1})\prod\limits _{\substack{z_{1},z_{2}\in Z\\
z_{1}<z_{2}
}
}(z_{2}-z_{1})}{\prod_{i=1}^{m-1}i{}^{m-i}\prod_{j=1}^{n-m-1}j{}^{n-m-j}},
\]
and 
\[
\rho_{\lambda,G}(1)=\frac{\prod\limits _{\substack{x_{1},x_{2}\in X\\
x_{1}<x_{2}
}
}(x_{2}-x_{1})}{\prod_{k=1}^{n-1}k^{n-k}},
\]
respectively. The proposition follows immediately from Lemma~\ref{Products}. 
\end{proof}
\begin{prop}
\label{Chi of t bound}Let $t\in T$ be an element whose characteristic
polynomial satisfies 
\[
P_{t}(x)=\prod_{i=1}^{m}(x-\alpha_{i})\prod_{j=1}^{n-m}(x-\beta_{j}),
\]
where $\left|\alpha_{i}-\beta_{j}\right|\ge\epsilon$ for all $i,j$.
Then for every dominant weight $\lambda$ of $G$,
\[
\left|\rho_{\lambda,G}(t)\right|\le2^{2n^{2}\log n}\epsilon^{-m(n-m)}\Bigl(\rho_{\lambda,G}(1)\Bigr)^{\frac{n^{2}-2mn+2m^{2}}{n^{2}-mn+m^{2}}}.
\]
\end{prop}

\begin{proof}
Let $\Phi_{G}$ (resp. $\Phi_{H}$) denote the set of roots of $G$
(resp. $H$) with respect to $T$. Write $\Phi_{G}^{+}$ (resp. $\Phi_{H}^{+}$)
for the set of positive roots. By the Weyl character formula, we have
\begin{equation}
\rho_{\lambda,G}(t)=\frac{\sum_{w\in W_{G}=S_{n}}\sgn(w)\cdot\xi_{w(\lambda+\delta_{G})}(t)}{\xi_{\delta_{G}}(t)^{-1}\prod_{\alpha\in\Phi_{G}^{+}}(\alpha(t)-1)}.\label{WG sum}
\end{equation}
By the hypothesis, for all $\alpha\in\Phi_{G}\setminus\Phi_{H}$,
$\left|\alpha(t)-1\right|\ge\epsilon$.

Every right $W_{H}$-coset in $W_{G}$ has a representative $\sigma$
such that $\sigma(\lambda+\delta_{G})$ is an $H$-dominant weight.
Let $\Sigma\subset W_{G}$ be the set of these coset representatives.
Then $W_{G}=W_{H}\Sigma$ and the right hand side of (\ref{WG sum})
is equal to 
\[
\sum_{\sigma\in\Sigma}\sgn(\sigma)\frac{\sum_{w\in W_{H}}\sgn(w)\cdot\xi_{w(\sigma(\lambda+\delta_{G}))}(t)}{\xi_{\delta_{G}}(t)^{-1}\prod_{\alpha\in\Phi_{G}^{+}}(\alpha(t)-1)}=\sum_{\mu}\pm\frac{\sum_{w\in W_{H}}\sgn(w)\cdot\xi_{w(\mu+\delta_{H})}(t)}{\xi_{\delta_{G}}(t)^{-1}\prod_{\alpha\in\Phi_{G}^{+}}(\alpha(t)-1)},
\]
where the outer sum on the right hand side is taken over all $H$-dominant
weights $\mu$ such that $\mu+\delta_{H}$ lies in the $W_{G}$-orbit
of $\lambda+\delta_{G}$. Taking absolute values, we have 
\begin{equation}
\begin{split}\left|\rho_{\lambda,G}(t)\right| & \le\sum_{\mu}\left|\frac{\sum_{w\in W_{H}}\sgn(w)\cdot\xi_{w(\mu+\delta_{H})}(t)}{\xi_{\delta_{H}}(t)^{-1}\prod_{\alpha\in\Phi_{G}^{+}}(\alpha(t)-1)}\right|\\
 & \le\epsilon^{-\left|\Phi_{G}^{+}\setminus\Phi_{H}^{+}\right|}\sum_{\mu}\left|\frac{\sum_{w\in W_{H}}\sgn(w)\cdot\xi_{w(\mu+\delta_{H})}(t)}{\xi_{\delta_{H}}(t)^{-1}\prod_{\alpha\in\Phi_{H}^{+}}(\alpha(t)-1)}\right|\\
 & =\epsilon^{-\left|\Phi_{G}^{+}\setminus\Phi_{H}^{+}\right|}\sum_{\mu}\left|\rho_{H,\mu}(t)\right|\le\epsilon^{-m(n-m)}\binom{n}{m}\max_{\mu}\rho_{H,\mu}(1).
\end{split}
\label{chi G of t}
\end{equation}
By Proposition~\ref{Dim ineq}, for all $\mu$, 
\[
\rho_{H,\mu}(1)\le2^{n}\Bigl(\prod_{k=1}^{n-1}k{}^{n-k}\rho_{\lambda,G}(1)\Bigr)^{\frac{n^{2}-2mn+2m^{2}}{n^{2}-mn+m^{2}}}\leq2^{n}\cdot2^{\frac{1}{2}n^{2}\log n}\Bigl(\rho_{\lambda,G}(1)\Bigr)^{\frac{n^{2}-2mn+2m^{2}}{n^{2}-mn+m^{2}}},
\]
which, together with the inequality $\binom{n}{m}\leq2^{n}\leq2^{\frac{1}{2}n^{2}\log n}$,
implies the proposition. 
\end{proof}
We can now state the first main theorem of the section: 
\begin{thm}
\label{thm:Exponential bound for separated}Let $n\geq2$ be an integer,
$g\in\U_{n}$, $\rho\in\Irr(\U_{n})$, and $0<\gamma<\frac{1}{2}$.
Denote $\epsilon=\rho(1)^{-\gamma(1-\gamma)/n^{2}}$ and assume that
$\rho(1)>2^{\frac{16}{\gamma}n^{2}\log n}$ and that $g$ is $(\gamma,\epsilon)$-separated.
Then 
\begin{equation}
\left|\rho(g)\right|\le\rho(1)^{1-\frac{1}{2}\gamma(1-\gamma)}.\label{delta bound}
\end{equation}
\end{thm}

\begin{proof}
Without loss of generality, we can assume that $g$ is diagonal with
eigenvalues $\lambda_{1},\ldots,\lambda_{n}$ and there exists $m\in[\gamma n,\frac{1}{2}n]$
such that $\left|\lambda_{i}-\lambda_{j}\right|\geq\epsilon$ if $i\leq m<j$.
Let 
\[
f(x):=\frac{x(1-x)}{x^{2}+1+(1-x)^{2}}.
\]
Then $f(x)$ is increasing on $[0,\frac{1}{2}]$ and $f(x)\geq\frac{1}{2}x(1-x)$.
Since $\frac{m}{n}\in[\gamma,\frac{1}{2}]$, 
\[
1-\gamma(1-\gamma)\ge1-2f(\gamma)\ge1-2f\bigl(\frac{m}{n}\bigr)=\frac{n^{2}-2mn+2m^{2}}{n^{2}-mn+m^{2}}.
\]
By Proposition~\ref{Chi of t bound}, 
\begin{equation}
\left|\rho(g)\right|\le2^{2n^{2}\log n}\epsilon^{-n^{2}/4}\rho(1)^{1-\gamma(1-\gamma)}.\label{ineq 1}
\end{equation}
Since $\rho(1)>2^{\frac{16}{\gamma}n^{2}\log n}$ we have: 
\begin{equation}
2^{2n^{2}\log n}\leq\rho(1)^{\frac{1}{8}\gamma}\leq\rho(1)^{\frac{1}{4}\gamma(1-\gamma)},\label{ineq 2}
\end{equation}
By definition of $\epsilon$, 
\begin{equation}
\epsilon^{-n^{2}/4}=\rho(1)^{\frac{1}{4}\gamma(1-\gamma)}.\label{ineq 3}
\end{equation}
The theorem follows immediately from inequalities (\ref{ineq 1}),
(\ref{ineq 2}), and (\ref{ineq 3}). 
\end{proof}
We now turn to proving character estimates for $(\beta,\epsilon)$-spread
elements. 
\begin{lem}
\label{Lem: reduction from spread to separated}Let $n\ge2$, and
suppose $0<\epsilon<1$ and $0<\beta<\frac{1}{2}$. If $g\in\U_{n}$
is $(2\beta,\epsilon)$-spread, then it is $(\beta,\frac{\epsilon}{n})$-separated. 
\end{lem}

\begin{proof}
Let $g\in\U_{n}$ be a $(2\beta,\epsilon)$-spread element. By the
pigeonhole principle, there exists an open arc on the unit circle
of angle $\frac{2\pi}{n}>\frac{2\epsilon}{n}$ which contains no eigenvalue
of $g$. The complement of this arc is the image of some interval
of the form $[a,a+2\pi(n-1)/n]$ under the map $\theta\mapsto e^{i\theta}$,
so we can identify the eigenvalues of $g$ with $e^{i\theta_{1}},\ldots e^{i\theta n}$,
where 
\[
a\le\theta_{1}\le\theta_{2}\le\cdots\le\theta_{n}\le a+\frac{2\pi(n-1)}{n}.
\]
Let $r:=\lceil\beta n\rceil-1$. Since $g$ is $(2\beta,\epsilon)$-spread,
and since $n-2r>(1-2\beta)n$, it is impossible that $\theta_{r+1},\theta_{r+2},\ldots,\theta_{n-r}$
all lie in a subinterval of length $2\epsilon$. Hence, there must
be some gap of size $>2\frac{\epsilon}{n}$ between some pair $\theta_{s}$
and $\theta_{s+1}$, where $r+1\le s\le n-r-1$. The arc between any
element of $\{\theta_{1},\ldots,\theta_{s}\}$ and any element of
$\{\theta_{s+1},\ldots,\theta_{n}\}$ is at least $2\frac{\epsilon}{n}$,
so the distance is greater than $\frac{\epsilon}{n}$, and hence $g$
is $(\beta,\frac{\epsilon}{n})$-separated. 
\end{proof}
\begin{cor}
\label{cor:character estimates for spread elements}Let $0<\beta<1$.
Then for every $n\ge2$, every $\rho\in\Irr(\U_{n})$ with $\rho(1)>2^{\frac{8}{\beta}n^{2}\log n}$
and every element $g$ of $\U_{n}$ which is $(\beta,\epsilon)$-spread,
for $\epsilon=\rho(1)^{-\frac{1}{8}\beta(2-\beta)/n^{2}}$, we have:
\begin{equation}
\left|\rho(g)\right|\le\rho(1)^{1-\frac{\beta(2-\beta)}{8}}.\label{delta bound for spread elements}
\end{equation}
\end{cor}

\begin{proof}
By our assumption on $\rho(1)$, we have $\frac{1}{n}>\rho(1)^{-\frac{1}{8}\beta(2-\beta)/n^{2}},$
and hence 
\[
\frac{\epsilon}{n}=\frac{1}{n}\cdot\rho(1)^{-\frac{1}{8}\beta(2-\beta)/n^{2}}>\rho(1)^{-\frac{1}{4}\beta(2-\beta)/n^{2}}.
\]
By Lemma \ref{Lem: reduction from spread to separated}, if $g$ is
$(\beta,\epsilon)$-spread, then it is $(\frac{\beta}{2},\frac{\epsilon}{n})$-separated,
and hence $(\frac{\beta}{2},\rho(1)^{-\frac{1}{4}\beta(2-\beta)/n^{2}})$-separated.
Theorem \ref{thm:Exponential bound for separated} now implies the
corollary. 
\end{proof}

\subsection{Improved character estimates for symmetric and exterior power representations}

We conclude this section by observing that for symmetric and exterior
powers of the natural representation of $\U_{n}$, we can prove upper
bounds of exponential type for character degrees much smaller than
$2^{O(n^{2}\log n)}$. The following lemma is a direct consequence
of Stirling's approximation. 
\begin{lem}
\label{lem:Striling's inequality}For all $k\geq1$ and $n>k$, we
have 
\begin{equation}
\left(\frac{n}{k}\right)^{k}\leq\prod_{j=0}^{k-1}\left(\frac{n-j}{k-j}\right)=\binom{n}{k}\leq\frac{n^{k}}{k!}\leq\left(\frac{n}{k}\right)^{k}e^{k}.\label{eq:Striling on binomial coefficients}
\end{equation}
\end{lem}

Let us denote by $h_{m,n}:=\rho_{m^{1},n}\in\mathrm{Irr}(\U_{n})$
the character of the $m$-th symmetric power of the standard representation
of $\U_{n}$. 
\begin{prop}
\label{prop:character bound for symmetric}Fix $0<\beta<1$. There
exist $A,c>0$ such that if $n\ge2$, $m\ge An$, then for every $(\beta,h_{m,n}(1)^{-\frac{c}{n}})$-spread
$g\in\U_{n}$, 
\[
\left|h_{m,n}(g)\right|<h_{m,n}(1)^{1-\frac{\beta}{4}}.
\]
\end{prop}

\begin{proof}
For any fixed $n$, taking $c$ small enough, we have $\epsilon:=h_{m,n}(1)^{-\frac{c}{n}}>h_{m,n}(1)^{-\frac{\beta(2-\beta)}{8n^{2}}}$.
By taking $A$ large enough, we can guarantee that $h_{m,n}(1)>2^{\frac{8}{\beta}n^{2}\log n}$,
for Corollary~\ref{cor:character estimates for spread elements}
to apply. We may therefore take $n$ as large as we wish.

Let $\zeta_{1},\ldots,\zeta_{n}$ be the eigenvalues of $g$. Then:
\[
\sum_{m=0}^{\infty}h_{m,n}(g)t^{m}=\prod_{i=1}^{n}(1-\zeta_{i}t)^{-1}.
\]
Therefore, by Cauchy's differentiation formula, 
\begin{equation}
h_{m,n}(g)=\frac{1}{2\pi i}\oint\frac{\prod_{i=1}^{n}(1-\zeta_{i}t)^{-1}}{t^{m+1}}dt,\label{contour}
\end{equation}
where we take the integral counterclockwise around the circle of radius
$\frac{m}{m+n}$ centered at the origin.

For every $t$ on the contour of integration, we have $\left|1-\zeta_{i}t\right|\ge\frac{n}{m+n}$.
Furthermore, since $g$ is $(\beta,\epsilon)$-spread, we have $\left|1-\zeta_{i}t\right|\ge\epsilon/2$
for at least $\beta n$ values of $i$. It follows that the integrand
in (\ref{contour}) is bounded above by 
\begin{equation}
\bigl(1+\frac{m}{n}\bigr)^{(1-\beta)n}\bigl(\frac{2}{\epsilon}\bigr)^{\beta n}\bigl(1+\frac{n}{m}\bigr)^{m+1}.\label{eq:bound on integrand}
\end{equation}
By possibly enlarging $A$, we may assume $A>e^{24/\beta}$ so by
(\ref{eq:Striling on binomial coefficients}), 
\[
h_{m,n}(1)^{\frac{\beta}{4}}=\binom{m+n-1}{n-1}^{\frac{\beta}{4}}\ge\bigl(\frac{m}{n-1}+1\bigr)^{\frac{(n-1)\beta}{4}}>A^{\frac{(n-1)\beta}{4}}>e^{6(n-1)}\ge e^{3n}.
\]
Taking $c<\frac{1}{4}$, we may further assume $\epsilon>h_{m,n}(1)^{-\frac{1}{4n}}$.
By (\ref{eq:bound on integrand}) and Lemma \ref{lem:Striling's inequality},
and since $(1+\frac{n}{m})^{m}<e^{n}$, we have: 
\begin{align*}
\left|h_{m,n}(g)\right| & \leq\bigl(\frac{m+n-1}{n-1}\bigr)^{n(1-\beta)}\bigl(\frac{2}{\epsilon}\bigr)^{\beta n}2e^{n}\leq h_{m,n}(1)^{(1+\frac{1}{n-1})(1-\beta)}\epsilon^{-\beta n}e^{3n}\\
 & \leq h_{m,n}(1)^{1-\beta+\frac{1}{n-1}}\epsilon^{-\beta n}e^{3n}\leq h_{m,n}(1)^{1-\beta+\frac{1}{n-1}}h_{m,n}(1)^{\frac{\beta}{2}}\leq h_{m,n}(1)^{1-\frac{\beta}{4}},
\end{align*}
for $n\gg1$, as required.
\end{proof}
\begin{prop}
\label{prop:character bounds for exterior}Suppose $0<\alpha<\frac{1}{2}$
and $0<\epsilon,\beta<1$. Then there exists $\delta>0$ such that
if $n\ge2$ and $\alpha n\le m\le(1-\alpha)n$ then for every $(\beta,\epsilon)$-spread
$g\in G$, we have 
\[
\left|\rho_{1^{m},n}(g)\right|<\rho_{1^{m},n}(1)^{1-\delta}=\binom{n}{m}^{1-\delta}.
\]
\end{prop}

\begin{proof}
Without loss of generality, we assume $m\le n/2$. Let $\zeta_{1},\ldots,\zeta_{n}$
be the eigenvalues of $g$. Then 
\[
\sum_{k=0}^{n}\rho_{1^{k},n}(g)t^{k}=\prod_{i=1}^{n}(1+\zeta_{i}t).
\]
Let $\omega$ denote a primitive $n+1$st root of unity. Then 
\[
\rho_{1^{m},n}(g)=\frac{(\frac{n}{m}-1)^{m}}{n+1}\sum_{k=0}^{n}\omega^{-mk}\prod_{i=1}^{n}\bigl(1+\frac{m\omega^{k}\zeta_{i}}{n-m}\bigr).
\]
Thus, 
\[
\left|\rho_{1^{m},n}(g)\right|\le(\frac{n}{m}-1)^{m}\max_{0\le k\le n}\prod_{i=1}^{n}\left|1+\frac{m\omega^{k}\zeta_{i}}{n-m}\right|.
\]
At least $\beta n$ elements $i\in\{1,\ldots,n\}$ satisfy $\left|-1+\omega^{k}\zeta_{i}\right|\ge\epsilon/2$.
For these values, we have 
\begin{align*}
\left|1+\frac{m\omega^{k}\zeta_{i}}{n-m}\right|^{2} & \le\bigl(1+\frac{m}{n-m}\bigr)^{2}-\frac{\epsilon^{2}m}{4(n-m)}\\
 & \le\bigl(1-\frac{\epsilon^{2}m}{16(n-m)}\bigr)\bigl(1+\frac{m}{n-m}\bigr)^{2}\le\bigl(1-\frac{\epsilon^{2}\alpha}{16(1-\alpha)}\bigr)\bigl(\frac{n}{n-m}\bigr)^{2}.
\end{align*}
It follows that 
\[
\left|\rho_{1^{m},n}(g)\right|\le\bigl(\frac{n-m}{m}\bigr)^{m}\bigl(\frac{n}{n-m}\bigr)^{n}\bigl(1-\frac{\epsilon^{2}\alpha}{16(1-\alpha)}\bigr)^{\frac{\beta n}{2}}=\frac{n^{n}}{m^{m}(n-m)^{n-m}}\bigl(1-\frac{\epsilon^{2}\alpha}{16(1-\alpha)}\bigr)^{\frac{\beta n}{2}}.
\]
Using Stirling's approximation in the form 
\[
\sqrt{2\pi n}e^{-n}n^{n}<n!<e^{\frac{1}{12n}}\sqrt{2\pi n}e^{-n}n^{n}\le e^{\frac{1}{12}}\sqrt{2\pi n}e^{-n}n^{n}
\]
for $n\ge1$, we have 
\[
\frac{n^{n}}{m^{m}(n-m)^{n-m}}<e^{1/6}\sqrt{\pi n/2}\binom{n}{m},
\]
and since $\binom{n}{m}<2^{n}$, if $\delta>0$ is small enough, 
\[
\left|\rho_{1^{m},n}(g)\right|\le\binom{n}{m}^{1-\delta}=\rho_{1^{m},n}(1)^{1-\delta}
\]
for all $n$ sufficiently large. For any given $n\ge2$, there are
only finitely many possibilities for $m$, so the proposition holds
for all $n$. 
\end{proof}

\section{\label{sec:Probability-bounds}Probability bounds for poorly spread
elements}

In this section, we bound the probability that a random value of a
word $w$ in $\U_{n}$ fails to be $(\beta,\epsilon)$-spread, as
defined in Definition~\ref{def:separated and spread}. 
\begin{thm}
\label{thm:probabilistic result} Let $w\in F_{r}$ be a word of length
$\ell\geq1$, let $0<\beta<1$, and let $\mathsf{G}_{1},\ldots,\mathsf{G}_{r}$
be independent, Haar distributed random variables in $\U_{n}$. For
every $\epsilon>0$ and every $n>\max\left\{ \frac{4\ell}{1-\beta},\frac{8}{1-\beta},16\right\} $,
\[
\mathbb{P}\left(w(\mathsf{G}_{1},\ldots,\mathsf{G}_{r}))\text{ fails to be }(\beta,\epsilon)\text{-spread}\right)<2^{3n^{2}}\epsilon^{\frac{n^{2}(1-\beta)^{2}}{4(\ell+1)}}.
\]
\end{thm}

The proof of Theorem \ref{thm:probabilistic result} follows the paper
\cite{LaS12} by the third author and Shalev. If $g\in\U_{n}$, we
say that a unit vector $v$ is an \emph{$\epsilon$-approximate eigenvector}
of $g$ if the orthogonal projection of $gv$ to the hyperplane $v^{\perp}$
has norm at most $\epsilon$. For $g\in\U_{n}$ and an integer $m<\frac{n}{\ell+1}$,
we randomize $m$ independent and uniformly distributed unit vectors
in $\mathbb{C}^{n}$ and consider the event $\Omega_{m,\epsilon}$
that all of those vectors are $\epsilon$-approximate eigenvectors
of $g$. We will show that: 
\begin{enumerate}
\item If $g$ is chosen as a random $w$-value, the probability of $\Omega_{m,\epsilon}$
is bounded from above by $(c_{1}\epsilon)^{m(2n-2m(\ell+1)-1)}$. 
\item Conditional on $g$ not being $(\beta,\epsilon)$-spread, the probability
of $\Omega_{m,\epsilon}$ is bounded from below by $(c_{2}\epsilon)^{2\beta mn}$. 
\end{enumerate}
It follows that the probability that a random $w$-value fails to
be $(\beta,\epsilon)$-spread is bounded by $\frac{(c_{1}\epsilon)^{m(2n-2m(\ell+1)-1)}}{(c_{2}\epsilon)^{2\beta mn}}$.
Optimizing over $m$, we get Theorem \ref{thm:probabilistic result}.

Let $G$ be a compact Lie group and let $H_{2}\subseteq H_{1}\subseteq G$
be closed subgroups of $G$. We denote the $G$-invariant probability
measures on $G/H_{1}$ and $G/H_{2}$ by $\mu_{G/H_{1}}$ and $\mu_{G/H_{2}}$,
respectively. The quotient map $\pi\colon G/H_{1}\to G/H_{2}$ is
$G$-equivariant so $\pi_{*}\mu_{G/H_{1}}=\mu_{G/H_{2}}$. In addition,
for every $g$, the conditional measure on the fiber $\pi^{-1}(gH_{2})$
is $gH_{2}g^{-1}$-invariant. Applying this to $G=\U_{n}$, $H_{1}=\U_{m+1}$,
and $H_{2}=\U_{m}$, we get that, if $V,W$ are $m$-dimensional subspaces
of $\mathbb{C}^{n}$, $v\in V^{\perp}$, and $w\in W^{\perp}$, then,
conditioning on $g(V)=W$ and on the restriction $g|_{V}$, the random
vector $g(v)$ is uniformly distributed on the unit sphere in $W^{\perp}$
and $g^{-1}(w)$ is uniformly distributed on the unit sphere of $V^{\perp}$.
Hence, 
\begin{lem}
\label{lem:cond.independence}Let $V,W\subseteq\mathbb{C}^{n}$ be
two linear subspaces and let $v\in V^{\perp},w\in W^{\perp}$ be unit
vectors. Suppose that $g\in\U_{n}$ is a Haar-random element. Then,
conditioning on $g(V)=W$, 
\begin{enumerate}
\item The random vector $g(v)$ is uniformly distributed on the unit sphere
in $W^{\perp}$ and $g^{-1}(w)$ is uniformly distributed on the unit
sphere of $V^{\perp}$. 
\item The random linear map $g|_{V}\in\Hom(V,W)$ and the random unit vector
$g(v)\in W^{\perp}$ are independent. 
\item The random linear map $g|_{V}\in\Hom(V,W)$ and the random unit vector
$g^{-1}(w)\in V^{\perp}$ are independent. 
\end{enumerate}
\end{lem}

Let $r\in\mathbb{N}$, $w\in F_{r}$, and $\ell$ be as in Theorem
\ref{thm:probabilistic result}. We denote the generators of $F_{r}$
by $x_{1},\ldots,x_{r}$ and write $w=w_{\ell}\cdots w_{1}$ where
each $w_{i}$ is either $x_{j}$ or $x_{j}^{-1}$ and no two consecutive
terms are inverses of one another. Denote $w^{0}=1$ and $w^{j}=w_{j}\cdots w_{2}w_{1}$,
if $1\leq j\leq\ell$.

For a linear subspace $V\subseteq\mathbb{C}^{n}$, we denote the unit
sphere in $V$ by $S(V)$. Given a positive integer $m$, define maps
$\phi_{i,j}\colon\U_{n}^{r}\times S(\C^{n})^{m}\to S(\C^{n})$ by
\[
\phi_{i,j}(g_{1},\ldots,g_{r},v_{1},\ldots,v_{m})=w^{j}(g_{1},\ldots,g_{r})(v_{i}),
\]
and define $\phi\colon\U_{n}^{r}\times S(\C^{n})^{m}\to S(\C^{n})^{m(\ell+1)}$
by 
\[
\phi(g_{1},\ldots,g_{r},v_{1},\ldots,v_{m})=\Big(\phi_{i,j}(g_{1},\ldots,g_{r},v_{1},\ldots,v_{m})\Big)_{\substack{1\le i\le m\\
0\le j\le\ell
}
}.
\]
We denote the lexicographic order on $\left\{ 1,\ldots,m\right\} \times\left\{ 0,\ldots,\ell\right\} $
by $\prec$. Given an array 
\[
\bh:=\Big(h_{ij}\Big)_{\substack{1\le i\le m\\
0\le j\le\ell
}
}\in S(\mathbb{C}^{n})^{m(\ell+1)},
\]
let 
\[
Z_{i,j}(\bh):=\Span\left\{ h_{p,q}\mid(p,q)\prec(i,j)\right\} .
\]
We say that $\bh$ has an \emph{$\epsilon$-fault} at $i$ if 
\[
\left|\pr_{Z_{i,\ell}(\bh)^{\perp}}h_{i,\ell}\right|\leq\epsilon,
\]
where $\pr_{V}$ denotes the orthogonal projection on a subspace $V\subseteq\mathbb{C}^{n}$.

If $1\leq i\leq m$, then $Z_{i,\ell}(\phi(g_{1},\ldots,g_{r},v_{1},\ldots,v_{m}))^{\perp}\subseteq v_{i}^{\perp}$.
Thus, 
\begin{lem}
\label{lem:few approx eigenvalues for non faulty }If $v_{i}$ is
an $\epsilon$-approximate eigenvector of $w(g_{1},\ldots,g_{r})$
then $\phi(g_{1},\ldots,g_{r},v_{1},\ldots,v_{m})$ has an $\epsilon$-fault
at $i$. 
\end{lem}

We next prove a lemma about the norm of the projection of a random
unit vector to a subspace. 
\begin{defn}
\label{defn:Pdn}For natural numbers $1\leq d\leq n$, define a probability
distribution $P_{d,n}$ as follows. Let $W=\Span\left\{ e_{1},\ldots,e_{d}\right\} \subseteq\mathbb{C}^{n}$
and let $\sX$ be a uniformly distributed random unit vector in $\mathbb{C}^{n}$.
The distribution $P_{d,n}$ is the distribution of the random variable
$\left|\pr_{W}(\sX)\right|$. 
\end{defn}

\begin{lem}
\label{sphere.lemma}Let $1\leq d\leq n$ and let $\mathsf{d}$ be
a random variable with distribution $P_{d,n}$. Then: 
\end{lem}

\begin{enumerate}
\item For every $\epsilon>0$, 
\[
4^{-n}\epsilon^{2d}\leq\mathbb{P}\left(\mathsf{d}\leq\epsilon\right)\leq2^{n}\epsilon^{2d}.
\]
\item For every $0\leq s\leq2d-1$, 
\[
\mathbb{E}\mathsf{d}^{-s}\leq2^{n+s+1}.
\]
\end{enumerate}
\begin{proof}
Let $W,\mathsf{X}$ be as in Definition \ref{defn:Pdn}. For $r\le1$,
\[
\P\left(\left|\pr_{W}(r\sX)\right|\le\epsilon\right)\ge\P\left(\left|\pr_{W}(\sX)\right|\le\epsilon\right),
\]
and for $r\ge1$, we have the opposite inequality. Therefore, if $\sY$
is a random variable uniformly distributed on the unit ball $B(\C^{n})$
in $\C^{n}$ , then 
\[
\P\left(\left|\pr_{W}(\sY)\right|\le\epsilon\right)\ge\P\left(\left|\pr_{W}(\sX)\right|\le\epsilon\right),
\]
and if $\mathsf{Z}$ is a uniformly distributed random variable on
$2B(\C^{n})\smallsetminus B(\C^{n})$, then 
\[
\P\left(\left|\pr_{W}(\mathsf{Z})\right|\le\epsilon\right)\le\P\left(\left|\pr_{W}(\sX)\right|\le\epsilon\right),
\]
Recall that the volume of the unit ball in $\mathbb{C}^{n}$ is $\frac{\pi^{n}}{n!}$.
The set of points in $B(\C^{n})$ whose projection to $W$ has length
$\le\epsilon$ is contained in $\epsilon B(W)\times B(W^{\perp})$,
which has volume 
\[
\frac{\pi^{d}\epsilon^{2d}}{d!}\frac{\pi^{n-d}}{(n-d)!}=\epsilon^{2d}\binom{n}{d}\frac{\pi^{n}}{n!}.
\]
Therefore, 
\[
\P\left(\left|\pr_{W}(\sX)\right|\le\epsilon\right)<\epsilon^{2d}\binom{n}{d}<2^{n}\epsilon^{2d}.
\]
Similarly, the set of points in $2B(\C^{n})\smallsetminus B(\C^{n})$
whose projection to $W$ has length $\le\epsilon$ contains
\[
\epsilon B(W)\times\bigl(\sqrt{4-\epsilon^{2}}B(W^{\perp})\smallsetminus B(W^{\perp})\bigr),
\]
which implies 
\[
\P\left(\left|\pr_{W}(\sX)\right|\le\epsilon\right)>\frac{\epsilon^{2d}((4-\epsilon^{2})^{n-d}-1)\binom{n}{d}}{4^{n}-1}>4^{-n}\epsilon^{2d}.
\]
Finally, if $s\leq2d-1$, then 
\[
\E\left(\mathsf{d}^{-s}\right)\leq\sum_{k=0}^{\infty}\mathbb{P}(2^{-k-1}\leq\mathsf{d}\leq2^{-k})2^{(k+1)s}\leq2^{n+s}\sum_{k=0}^{\infty}2^{(s-2d)k}\leq2^{n+s+1}.\qedhere
\]
\end{proof}
We now bound the probability that, for a random $w$-value, $m$ random
unit vectors are all $\epsilon$-approximate eigenvectors. 
\begin{prop}
\label{Prop: Random faults}Let $r,w,\ell$ be as in Theorem \ref{thm:probabilistic result}
and let $m,n$ be natural numbers such that $n>m(\ell+1)$. Suppose
that $\sG_{1},\ldots,\sG_{r}$ are Haar-distributed random elements
in $\U_{n}$ and $\sX_{1},\ldots,\sX_{m}$ are uniformly distributed
random vectors in $S(\C^{n})$ such that, jointly, $\sG_{i}$ and
$\sX_{j}$ are independent. For every $\epsilon>0$, 
\begin{equation}
\mathbb{P}\Big(\text{all \ensuremath{\mathsf{X}_{i}} are \ensuremath{\epsilon}-approximate eigenvectors of \ensuremath{w(\mathsf{G}_{1},\ldots,\mathsf{G}_{r})}}\Big)<2^{3nm(\ell+1)}\epsilon^{m(2n-2m(\ell+1)-1)}.
\end{equation}
\end{prop}

\begin{proof}
Consider the random array $\bh=\phi(\sG_{1},\ldots,\sG_{r},\sX_{1},\ldots,\sX_{m})$.
By Lemma \ref{lem:few approx eigenvalues for non faulty }, it is
enough to bound the probability that $\bh$ has an $\epsilon$-fault
at all $i$.

For each index $(i,j)$, the vectors $\left\{ h_{p,q}\right\} _{(p,q)\prec(i,j)}$
determine the restrictions of $\mathsf{G}_{i}$ to some (random) subspaces.
We collect this information as follows. For $(i,j)\in\left\{ 1,\ldots,m\right\} \times\left\{ 0,\ldots,\ell\right\} $
and $x\in\left\{ x_{1},\ldots,x_{r},x_{1}^{-1},\ldots,x_{r}^{-1}\right\} $,
define 
\begin{enumerate}
\item A random vector $\mathsf{v}_{i,j}:=h_{i,j}=\phi_{i,j}(\mathsf{G}_{1},\ldots,\mathsf{G}_{r},\mathsf{X}_{1},\ldots,\mathsf{X}_{m})$. 
\item A random subspace 
\[
\mathsf{V}_{i,j}^{x}:=\Span\Big(\left\{ \mathsf{v}_{p,q-1}\mid(p,q)\prec(i,j)\text{ and }w_{q}=x\right\} \cup\left\{ \mathsf{v}_{p,q}\mid(p,q)\prec(i,j)\text{ and }w_{q}=x^{-1}\right\} \Big),
\]
\item A random linear map $\mathsf{T}_{i,j}^{x}:\mathsf{V}_{i,j}^{x}\rightarrow\mathsf{V}_{i,j}^{x^{-1}}$
as the restriction of $x(\mathsf{G}_{1},\ldots,\mathsf{G}_{r})$ to
$\mathsf{V}_{i,j}^{x}$. 
\end{enumerate}
We denote the $\sigma$-algebra generated by $\left\{ \mathsf{v}_{p,q}\mid(p,q)\prec(i,j)\right\} $
by $\mathcal{F}_{i,j}$. Equivalently, $\mathcal{F}_{i,j}$ is the
$\sigma$-algebra generated by $\mathsf{T}_{i,j}^{x_{1}},\ldots,\mathsf{T}_{i,j}^{x_{r}^{-1}}$.
Finally, let 
\[
\mathsf{Z}_{i,j}:=Z_{i,j}(\bh)=\Span\left\{ \mathsf{v}_{p,q}\right\} _{(p,q)\prec(i,j)}=\Span\Big(\bigcup_{x\in\left\{ x_{1},\dots,x_{r}^{-1}\right\} }\left(\mathsf{V}_{i,j}^{x}\right)\Big).
\]
Define a sequence of random variables $\mathsf{d}_{i,j}$ by $\mathsf{d}_{i,0}:=\left|\pr_{\mathsf{Z}_{i,0}^{\perp}}\mathsf{v}_{i,0}\right|$
and, for $1\leq j\leq\ell$, 
\[
\mathsf{d}_{i,j}:=\frac{\left|\pr_{\mathsf{Z}_{i,j}^{\perp}}\mathsf{v}_{i,j}\right|}{\left|\pr_{\mathsf{Z}_{i,j-1}^{\perp}}\mathsf{v}_{i,j-1}\right|}.
\]
The array $\bh$ has $\epsilon$-fault in $i$ if and only if $\prod_{j=0}^{\ell}\mathsf{d}_{i,j}\leq\epsilon$.

Fix $(i,j)$ with $j\geq1$. Assume without loss of generality that
$w_{j}=x_{k}$. Decomposing $\mathsf{v}_{i,j-1}$ into its components
in $\mathsf{V}_{i,j}^{w_{j}}$ and $(\mathsf{V}_{i,j}^{w_{j}})^{\perp}$,
we have 
\[
\mathsf{v}_{i,j}=\mathsf{G}_{k}(\mathsf{v}_{i,j-1})=\mathsf{G}_{k}\left(\pr_{\mathsf{V}_{i,j}^{w_{j}}}\mathsf{v}_{i,j-1}\right)+\mathsf{G}_{k}\left(\pr_{(\mathsf{V}_{i,j}^{w_{j}})^{\perp}}\mathsf{v}_{i,j-1}\right).
\]
$\mathsf{G}_{k}$ maps $\pr_{\mathsf{V}_{i,j}^{w_{j}}}\mathsf{v}_{i,j-1}$
into $\mathsf{Z}_{i,j}$. Therefore, 
\[
\mathsf{d}_{i,j}=\frac{\left|\pr_{\mathsf{Z}_{i,j}^{\perp}}\left(\mathsf{G}_{k}\left(\pr_{(\mathsf{V}_{i,j}^{w_{j}})^{\perp}}\mathsf{v}_{i,j-1}\right)\right)\right|}{\left|\pr_{\mathsf{Z}_{i,j-1}^{\perp}}\mathsf{v}_{i,j-1}\right|}\geq\frac{\left|\pr_{\mathsf{Z}_{i,j}^{\perp}}\left(\mathsf{G}_{k}\left(\pr_{(\mathsf{V}_{i,j}^{w_{j}})^{\perp}}\mathsf{v}_{i,j-1}\right)\right)\right|}{\left|\pr_{(\mathsf{V}_{i,j}^{w_{j}})^{\perp}}\mathsf{v}_{i,j-1}\right|},
\]
where we used the fact that $\sZ_{i,j-1}^{\perp}\subset(\mathsf{V}_{i,j}^{w_{j}})^{\perp}$.
Let $\widetilde{\mathsf{d}}_{i,j}$ be the random variable
\[
\widetilde{\mathsf{d}}_{i,j}:=\frac{\left|\pr_{\mathsf{Z}_{i,j}^{\perp}}\left(\mathsf{G}_{k}\left(\pr_{(\mathsf{V}_{i,j}^{w_{j}})^{\perp}}\mathsf{v}_{i,j-1}\right)\right)\right|}{\left|\pr_{(\mathsf{V}_{i,j}^{w_{j}})^{\perp}}\mathsf{v}_{i,j-1}\right|}.
\]
It is clear that $\widetilde{\mathsf{d}}_{i,j}$ is measurable with
respect to $\mathcal{F}_{i,j}$. Crucially, by Lemma \ref{lem:cond.independence},
$\widetilde{\mathsf{d}}_{i,j}$ is independent of $\mathcal{F}_{i,j-1}$,
so the $\widetilde{\mathsf{d}}_{i,j}$ are all independent. By Lemma
\ref{lem:cond.independence} again, $\mathsf{G}_{k}\left(\pr_{(\mathsf{V}_{i,j}^{w_{j}})^{\perp}}\mathsf{v}_{i,j-1}\right)$
is a vector uniformly distributed on $\left|\pr_{(\mathsf{V}_{i,j}^{w_{j}})^{\perp}}\v_{i,j-1}\right|S\left((\mathsf{V}_{i,j}^{w_{j}^{-1}})^{\perp}\right)$.
Since $\dim\mathsf{Z}_{i,j}$ and $\dim\mathsf{V}_{i,j}^{x}$ are
constant almost surely and $\dim\mathsf{Z}_{i,j},\dim\mathsf{V}_{i,j}^{x}\leq m(\ell+1)$,
Lemma \ref{lem:cond.independence} also implies that $\widetilde{\mathsf{d}}_{i,j}$
has distribution $P_{a_{i,j},b_{i,j}}$, for $n-m(\ell+1)\leq a_{i,j}<b_{i,j}\leq n$.
By Lemma \ref{sphere.lemma} with $s=2(n-m(\ell+1))-1$, 
\[
\mathbb{P}\left(\widetilde{\mathsf{d}}_{i,0}\cdot\widetilde{\mathsf{d}}_{i,1}\cdots\widetilde{\mathsf{d}}_{i,\ell}\leq\epsilon\right)\leq\frac{\mathbb{E}\left(\widetilde{\mathsf{d}}_{i,0}\cdot\widetilde{\mathsf{d}}_{i,1}\cdots\widetilde{\mathsf{d}}_{i,\ell}\right)^{-s}}{\epsilon^{-s}}\leq2^{(\ell+1)(3n-2m(\ell+1))}\epsilon^{s}\leq2^{3n(\ell+1)}\epsilon^{2(n-m(\ell+1))-1}.
\]
Thus, 
\begin{align*}
 & \mathbb{P}\Big((\forall i)\quad\text{\ensuremath{\bh} has \ensuremath{\epsilon}-fault in \ensuremath{i}}\Big)=\mathbb{P}\Big((\forall i)\quad\mathsf{d}_{i,0}\cdot\mathsf{d}_{i,1}\cdots\mathsf{d}_{i,\ell}\leq\epsilon\Big)\\
\leq & \mathbb{P}\Big((\forall i)\quad\widetilde{\mathsf{d}}_{i,0}\cdot\widetilde{\mathsf{d}}_{i,1}\cdots\widetilde{\mathsf{d}}_{i,\ell}\leq\epsilon\Big)\leq\left(2^{3n(\ell+1)}\epsilon^{2n-2m(\ell+1)-1}\right)^{m},
\end{align*}
and the proposition follows.
\end{proof}
Next, we bound the probability that, for a badly spread matrix, $m$
random unit vectors will each be an $\epsilon$-approximate eigenvector. 
\begin{lem}
\label{Badly separated case}Let $0<\beta,\epsilon<1$, let $g\in\U_{n}$
be a matrix which is not $(\beta,\epsilon)$-spread, and let $\mathsf{X}_{1},\ldots,\mathsf{X}_{k}$
be i.i.d random vectors in $\mathbb{C}^{n}$, each uniformly distributed
on the unit sphere. Then 
\[
\mathbb{P}\left(\text{all \ensuremath{\mathsf{X}_{i}} are \ensuremath{2\epsilon}-approximate eigenvectors of \ensuremath{g}}\right)>4^{-2nk}\epsilon^{2\beta nk}.
\]
\end{lem}

\begin{proof}
By assumption, there is $\zeta\in S(\mathbb{C})$ and a $g$-invariant
subspace $W$ of dimension at least $(1-\beta)n$ such that the eigenvalues
of $g$ on $W$ lie on an arc $I\subseteq\U_{1}$ of diameter $2\epsilon$,
centered at $\zeta$. If $w\in W$, then $\left|gw-\zeta w\right|\leq\epsilon\left|w\right|$.
Thus, for every unit vector $v$, 
\[
\left|\pr_{v^{\perp}}gv\right|^{2}\leq\left|gv-\zeta v\right|^{2}=\left|g(\pr_{W}v)-\zeta\pr_{W}v\right|^{2}+\left|g(\pr_{W^{\perp}}v)-\zeta\pr_{W^{\perp}}v\right|^{2}\leq\epsilon^{2}+4\left|\pr_{W^{\perp}}v\right|^{2}.
\]
From the inequality, if $\left|\pr_{W^{\perp}}v\right|\leq\frac{\epsilon}{2}$,
then $v$ is a $2\epsilon$-approximate eigenvector of $g$. By Lemma
\ref{sphere.lemma} $\mathbb{P}\Big(\left|\pr_{W^{\perp}}\mathsf{X}_{i}\right|\leq\frac{\epsilon}{2}\Big)>4^{-2n}\epsilon^{2\beta n}$
and the result follows. 
\end{proof}
\begin{proof}[Proof of Theorem \ref{thm:probabilistic result}]
Let $A$ be the event that $\mathsf{X}_{1},\ldots,\mathsf{X}_{m}$
are all $2\epsilon$-approximate eigenvectors of $w(\mathsf{G}_{1},\ldots,\mathsf{G}_{r})$,
and let $B$ be the event that $w(\mathsf{G}_{1},\ldots,\mathsf{G}_{r})$
is not $(\beta,\epsilon)$-spread. By Proposition \ref{Prop: Random faults},
\[
\mathbb{P}(A)\leq2^{nm(3\ell+5)}\epsilon^{2m(n-m(\ell+1))-m}.
\]
By Lemma \ref{Badly separated case}, 
\[
\frac{\mathbb{P}(A)}{\mathbb{P}(B)}\geq\frac{\mathbb{P}(A\cap B)}{\mathbb{P}(B)}=\mathbb{P}(A|B)\geq4^{-2nm}\epsilon^{2\beta nm}.
\]
Thus, 
\[
\mathbb{P}(B)\leq2^{3nm(\ell+3)}\epsilon^{2m(n(1-\beta)-m(\ell+1))-m}.
\]
We set $M:=\frac{n(1-\beta)}{2(\ell+1)}$ and $m:=\max\left\{ \left\lfloor M\right\rfloor ,2\right\} $.
Since $n(1-\beta)>\max\left\{ 4\ell,8\right\} $ we have $M\geq\frac{\max\left\{ 4\ell+1,9\right\} }{2(\ell+1)}\geq\max\left\{ \frac{3}{2},\frac{9}{2(\ell+1)}\right\} $.
If $M>2$, then $\frac{2}{3}M<m\leq M$ and hence: 
\[
\mathbb{P}(B)\leq2^{3n^{2}\frac{\ell+3}{2(\ell+1)}}\epsilon^{m\left(2(2M-m)(\ell+1)-1\right)}\leq2^{3n^{2}}\epsilon^{\frac{2}{3}M\left(2M(\ell+1)-1\right)}<2^{3n^{2}}\epsilon^{M^{2}(\ell+1)\left(\frac{4}{3}-\frac{2}{3M(\ell+1)}\right)}<2^{3n^{2}}\epsilon^{\frac{n^{2}(1-\beta)^{2}}{4(\ell+1)}}.
\]
If $M\leq2$ then $m=2$ and $M\leq m\leq\frac{4}{3}M$. Since $n>\max\left\{ 16,4\ell\right\} $,
we have $n>2(\ell+3)$. Hence, 
\[
\mathbb{P}(B)\leq2^{6n(\ell+3)}\epsilon^{m\left(2(2M-m)(\ell+1)-1\right)}<2^{3n^{2}}\epsilon^{M^{2}(\ell+1)\left(\frac{4}{3}-\frac{1}{M(\ell+1)}\right)}<2^{3n^{2}}\epsilon^{\frac{n^{2}(1-\beta)^{2}}{4(\ell+1)}}.\qedhere
\]
\end{proof}

\subsection{\label{subsec:Proof-of-Theorem B}Proof of Theorem \ref{Thm B, large Fourier coefficients}}

Theorem \ref{Thm B, large Fourier coefficients} follows from: 
\begin{thm}
\label{thm:Detalied Theorem B}There exist constants $a,b>0$ such
that, for every word $1\neq w\in F_{r}$, every $n\geq2$, and every
$\rho\in\Irr(\U_{n})$ satisfying $\rho(1)\ge2^{a(\ell(w)+1)n^{2}\log n}$,
\[
\left|\mathbb{E}\left(\rho\left(w(\mathsf{G}_{1},\ldots,\mathsf{G}_{r})\right)\right)\right|\leq\rho(1)^{1-\frac{b}{\ell(w)}}.
\]
In fact, if $n>8\ell(w)$, we can take $a=2^{9},b=2^{-10}$. 
\end{thm}

\begin{proof}
By \cite[Proposition 7.2(1)]{AG}, it is enough to prove the theorem
assuming $n>\max\left\{ 8\ell(w),16\right\} $. We assume this in
the following.

Given $w$ and $\rho$ such that $\rho(1)>2^{512(\ell(w)+1)n^{2}\log n}$,
let $\beta=\frac{1}{2}$ and $\epsilon=\rho(1)^{-\frac{3}{32n^{2}}}$.
By Corollary \ref{cor:character estimates for spread elements}, if
$g\in\U_{n}$ is $(\beta,\epsilon)$-spread, then $\left|\rho(g)\right|<\rho(1)^{\frac{29}{32}}$.
By Theorem \ref{thm:probabilistic result}, 
\[
\mathbb{P}\left(\text{\ensuremath{w(\mathsf{G}_{1},\ldots,\mathsf{G}_{r})} is not \ensuremath{(\beta,\epsilon)}-spread}\right)<2^{3n^{2}}\epsilon^{\frac{n^{2}}{16(\ell(w)+1)}}=2^{3n^{2}}\rho(1)^{-\frac{3}{512(\ell(w)+1)}}.
\]
Since $\rho(1)>2^{512(\ell(w)+1)n^{2}\log n}$ and $n>16$, we get
that $\rho(1)^{\frac{2}{512(\ell(w)+1)}}>2^{3n^{2}+1}$. Thus, 
\[
\mathbb{P}\left(\text{\ensuremath{w(\mathsf{G}_{1},\ldots,\mathsf{G}_{r})} is not \ensuremath{(\beta,\epsilon)}-spread}\right)<\frac{1}{2}\rho(1)^{-\frac{1}{512(\ell(w)+1)}},
\]
and 
\begin{align*}
 & \E\left(\rho(w(\mathsf{G}_{1},\ldots,\mathsf{G}_{r}))\right)\\
= & \int_{\text{\ensuremath{w(\mathsf{G}_{1},\ldots,\mathsf{G}_{r})} is not \ensuremath{(\beta,\epsilon)}-spread}}\rho(w(\sG_{1},\ldots,\sG_{r}))\mu_{\U_{n}^{r}}+\int_{\text{\ensuremath{w(\mathsf{G}_{1},\ldots,\mathsf{G}_{r})} is \ensuremath{(\beta,\epsilon)}-spread}}\rho(w(\sG_{1},\ldots,\sG_{r}))\mu_{\U_{n}^{r}}\\
< & \mathbb{P}\left(\text{\ensuremath{w(\mathsf{G}_{1},\ldots,\mathsf{G}_{r})} is not \ensuremath{(\beta,\epsilon)}-spread}\right)\cdot\rho(1)+\rho(1)^{\frac{29}{32}}\leq\frac{1}{2}\rho(1)^{1-\frac{1}{512(\ell(w)+1)}}+\rho(1)^{\frac{29}{32}}.\qedhere
\end{align*}
\end{proof}

\section{\label{sec:Geometry-and-singularities}Geometry and singularities
of word maps }

Throughout this section, $K$ is a field of characteristic $0$. For
a morphism $\varphi:X\rightarrow Y$ of finite type $K$-schemes,
we denote by $X_{y,\varphi}$ the scheme theoretic fiber of $\varphi$
over $y\in Y$. We write $X^{\mathrm{sm}}$ (resp.~$X^{\mathrm{sing}}$)
for the smooth (resp.~singular) locus of $X$. Similarly, we write
$X^{\varphi,\mathrm{sm}}$ (resp.~$X^{\varphi,\mathrm{sing}}$) for
the smooth (resp.~singular) locus of $\varphi$. We denote by $\mathbb{A}_{R}^{N}$
the affine space over a base ring $R$.

We recall the following definitions from \cite{AA16} and \cite{GH21}. 
\begin{defn}
\label{def:(FRS)}A morphism $\varphi:X\rightarrow Y$ between smooth
$K$-varieties is called \textit{(FRS)} if it is flat and if every
$x\in X(\overline{K})$, the fiber $X_{\varphi(x),\varphi}$ is reduced
and has rational singularities. 
\end{defn}

\begin{defn}
\label{def:convolution}Let $\varphi:X\rightarrow\underline{G}$ and
$\psi:Y\rightarrow\underline{G}$ be morphisms from algebraic $K$-varieties
$X,Y$ to an algebraic $K$-group $\underline{G}$. We define their
\emph{convolution} by
\[
\varphi*\psi:X\times Y\rightarrow\underline{G},\text{ }(x,y)\mapsto\varphi(x)\cdot_{\underline{G}}\psi(y).
\]
We denote by $\varphi^{*k}:X^{k}\rightarrow\underline{G}$ the \emph{$k$-th
self convolution} of $\varphi$.
\end{defn}

In \cite{GH19,GH21,GHb}, the second author and Yotam Hendel studied
the behavior of morphisms $\varphi:X\rightarrow\underline{G}$ under
the convolution operation, where it was observed that singularities
of $\varphi^{*t}:X^{t}\rightarrow\underline{G}$ improve as one takes
higher and higher convolution powers. In \cite{GHb}, the focus shifted
towards the special case of word maps on simple algebraic groups and
simple Lie algebras. The goal of this section is to prove the following: 
\begin{thm}
\label{thm:word maps on SLn are (FRS) after enough convolutions}For
every $1\neq w\in F_{r}$ and $n\geq2$, the map $w_{\mathrm{SL}_{n}}^{*t}:\mathrm{SL}_{n}^{rt}\rightarrow\mathrm{SL}_{n}$
is (FRS), for all $t\geq80\ell(w)n^{2}$. 
\end{thm}

As a consequence of Theorem \ref{thm:word maps on SLn are (FRS) after enough convolutions},
we will prove the following corollary, which implies Item (2) of Theorem
\ref{thm: A, bounded density}. 
\begin{cor}
\label{cor:bounded density for small rank}For every $1\neq w\in F_{r}$
and $n\geq2$, we have $\tau_{w,\SU_{n}}^{*t}\in L^{\infty}(\SU_{n})$
for $t\geq80\ell(w)n^{2}$. 
\end{cor}

In order to prove Theorem \ref{thm:word maps on SLn are (FRS) after enough convolutions}
we introduce the notion of jet-schemes.

If $n\in\mathbb{N}$, let $S_{\infty,n}$ be the ring of polynomials
over $K$ on the variables 
\[
x_{1},\ldots,x_{n},x_{1}^{(1)},\ldots,x_{n}^{(1)},x_{1}^{(2)},\ldots,x_{n}^{(2)}\ldots.
\]
$S_{\infty,n}$ has a unique $K$-derivation satisfying $D(x_{i}^{(j)})=x_{i}^{(j+1)}$
(e.g., $D(x_{1}^{2}x_{2})=2x_{1}x_{1}^{(1)}x_{2}+x_{1}^{2}x_{2}^{(1)}$).
If $X\subseteq\mathbb{A}_{K}^{n}$ is an affine $K$-scheme whose
coordinate ring is 
\[
K[x_{1},\dots,x_{n}]/(f_{1},\dots,f_{k}),
\]
we define the \emph{$m$-th jet scheme} $J_{m}(X)$ of $X$ to be
the affine scheme with the following coordinate ring: 
\[
K[x_{1},\dots,x_{n},x_{1}^{(1)},\dots,x_{n}^{(1)},\dots,x_{1}^{(m)},\dots,x_{n}^{(m)}]/(\{D^{u}f_{j}\}_{j=1,u=0}^{k,m}).
\]
We write $\pi_{m,X}:J_{m}(X)\rightarrow X$ for the natural projection.

If $\varphi:X\rightarrow Y$ is a morphism of affine $K$-schemes,
we define $J_{m}(\varphi):J_{m}(X)\rightarrow J_{m}(Y)$ by $J_{m}(\varphi)=(\varphi,D\varphi,\dots,D^{m}\varphi)$.
The definitions of $J_{m}(X),J_{m}(\varphi)$ and $\pi_{m,X}$ can
be glued to arbitrary $K$-schemes and $K$-morphisms (see \cite[Chapter 3]{CLNS18}
and \cite{EM09} for more details). 
\begin{lem}[{\cite[Theorem 29.2.11]{Vak} and \cite[III, Theorem 10.2]{Har77}}]
\label{Lemma:flat and smooth}Let $\varphi:X\rightarrow Y$ be a
morphism of smooth, geometrically irreducible $K$-varieties. Then: 
\begin{enumerate}
\item $\varphi$ is flat\emph{ }if and only if for every $x\in X(\overline{K})$
we have $\dim X_{\varphi(x),\varphi}=\dim X-\dim Y$. 
\item If $\varphi$ is flat, then $\varphi$ is smooth at $x\in X(\overline{K})$
if and only if $x\in X_{\varphi(x),\varphi}^{\mathrm{sm}}$. 
\end{enumerate}
\end{lem}

The following lemma follows from a jet-scheme characterization of
local complete intersection schemes with rational singularities, due
to Musta\c{t}\u{a} \cite{Mus01}. 
\begin{lem}
\label{lemma: jet characterization of (FRS)}Let $\varphi:X\rightarrow Y$
be a morphism of smooth, geometrically irreducible $K$-varieties.
Then the following are equivalent: 
\begin{enumerate}
\item $\varphi$ is (FRS) 
\item $J_{m}(\varphi)$ is flat with locally integral fibers, for every
$m\geq0$. 
\item $\varphi$ is flat, and each fiber $Z:=X_{\varphi(x),\varphi}$ of
$\varphi$ , where $x\in X(\overline{K})$, satisfies $\dim\pi_{m,Z}^{-1}(Z^{\mathrm{sing}})<(m+1)\dim Z$,
for every $m\geq0$. 
\end{enumerate}
\end{lem}

\begin{proof}
The condition $(1)\Leftrightarrow(2)$ is proved in \cite[Lemma 6.5]{GHS}.
Note that in both $(1)$ and $(3)$, $\varphi$ is flat, and thus
by Lemma \ref{Lemma:flat and smooth}, each of its fibers $X_{\varphi(x),x}$
where $x\in X(\overline{K})$, is an equidimensional, local complete
intersection scheme, and in particular Cohen-Macaulay. 

If $(1)$ holds, then each fiber $X_{\varphi(x),\varphi}$ is a local
complete intersection variety with rational singularities, and $(3)$
follows by applying \cite[Theorem 0.1 and Proposition 1.4]{Mus01}.
Suppose on the other hand that $(3)$ holds, and let $Z=X_{\varphi(x),\varphi}$.
Then $\dim\pi_{0,Z}^{-1}(Z^{\mathrm{sing}})=\dim Z^{\mathrm{sing}}<\dim Z$.
Since $Z$ is also Cohen-Macaulay, Serre's criterion for reducedness
\cite[Proposition 5.8.5]{Gro65} implies that $Z$ is reduced. $(1)$
now follows from $(3)$ by \cite[Theorem 0.1 and Proposition 1.4]{Mus01}. 
\end{proof}
\begin{lem}
\label{lem:auxilary lemma}Let $\varphi:X\rightarrow Y$ and $\psi:X'\rightarrow X$
be morphisms of smooth $K$-varieties. Then 
\begin{enumerate}
\item If $\varphi\circ\psi:X'\rightarrow Y$ is flat and $\psi:X'(\overline{K})\rightarrow X(\overline{K})$
is surjective, then $\varphi:X\rightarrow Y$ is flat. 
\item If $\varphi\circ\psi:X'\rightarrow Y$ is (FRS) and $J_{m}(\psi):J_{m}X'(\overline{K})\rightarrow J_{m}X(\overline{K})$
is surjective for all $m\in\N$, then $\varphi:X\rightarrow Y$ is
(FRS). 
\end{enumerate}
\end{lem}

\begin{proof}
For every $x\in X(\overline{K})$, the fiber $X'_{x,\psi}$ of $\psi$
is of dimension at least $\dim X'-\dim X$. Hence, for each $x'\in X'(\overline{K})$,
the codimension of $(X')_{\varphi\circ\psi(x'),\varphi\circ\psi}$
in $X'$ is less than or equal to the codimension of $X_{\varphi\circ\psi(x'),\varphi}$
in $X$. In particular, by Lemma \ref{Lemma:flat and smooth}(1),
Item (1) follows.

Suppose that $\varphi\circ\psi:X'\rightarrow X$ is (FRS). By Lemma
\ref{lemma: jet characterization of (FRS)}, $J_{m}(\varphi\circ\psi)=J_{m}(\varphi)\circ J_{m}(\psi)$
is flat for all $m\in\N$. Since $J_{m}(\psi):J_{m}X'(\overline{K})\rightarrow J_{m}X(\overline{K})$
is surjective, by Item (1), $J_{m}(\varphi)$ is flat.

Now fix $m\in\N$, and let $x\in X(\overline{K})$. Set $Z:=X_{\varphi(x),\varphi}$
and $Z'=X'_{\varphi(x),\varphi\circ\psi}$. Assume, in order to find
a contradiction, that 
\begin{equation}
\dim\pi_{m,Z}^{-1}(Z^{\mathrm{sing}})=(m+1)\dim Z=\dim(J_{m}(Z)).\label{eq:towards contradiction}
\end{equation}
Since $J_{m}(\psi|_{Z'}):J_{m}(Z')\rightarrow J_{m}(Z)$ is surjective,
and since $J_{m}(Z')$ and $J_{m}(Z)$ are equidimensional (as $J_{m}(\varphi\circ\psi)$
and $J_{m}(\varphi)$ are flat), (\ref{eq:towards contradiction})
implies that $\dim J_{m}(\psi|_{Z'})^{-1}\left(\pi_{m,Z}^{-1}(Z^{\mathrm{sing}})\right)=\dim J_{m}(Z')$,
and in particular: 
\begin{equation}
\dim\pi_{m,Z'}^{-1}(\psi^{-1}(Z^{\mathrm{sing}}))=\dim J_{m}(\psi|_{Z'})^{-1}\left(\pi_{m,Z}^{-1}(Z^{\mathrm{sing}})\right)=\dim J_{m}(Z').\label{eq:5.2}
\end{equation}
Note that if $\varphi$ is singular at $\widetilde{x}\in X(\overline{K})$,
then also $\varphi\circ\psi$ is singular at $\widetilde{x}'$ for
every $\widetilde{x}'\in\psi^{-1}(\widetilde{x})$. Together with
Lemma \ref{Lemma:flat and smooth}(2), and since $\varphi$ is flat,
this implies: 
\begin{equation}
\psi^{-1}(Z^{\mathrm{sing}})=\psi^{-1}(X^{\varphi,\mathrm{sing}}\cap Z)\subseteq X'^{\varphi\circ\psi,\mathrm{sing}}\cap Z'=Z'^{\mathrm{sing}}.\label{eq:compare singular locus}
\end{equation}
Finally, by (\ref{eq:5.2}) and (\ref{eq:compare singular locus})
we deduce that $\dim\pi_{m,Z'}^{-1}(Z'^{\mathrm{sing}})=\dim J_{m}(Z')$,
which, by Lemma \ref{lemma: jet characterization of (FRS)}, contradicts
that fact that $\varphi\circ\psi:X'\rightarrow X$ is (FRS). 
\end{proof}
To prove Theorem \ref{thm:word maps on SLn are (FRS) after enough convolutions},
we need the following lemma. A stronger version of this lemma was
proved in \cite[Lemma 7.10 and Proposition 7.12]{GHb} for $n$ even. 
\begin{lem}
\label{lem:reduction to affine}For every $n\geq2$, there exists
a polynomial map $\psi_{n}:\mathbb{A}_{\Q}^{N}\rightarrow\mathrm{SL}_{n}$,
such that $J_{m}(\psi_{n}):\overline{\Q}^{N(m+1)}\rightarrow\left(J_{m}\mathrm{SL}_{n}\right)(\overline{\Q})$
is surjective for all $m\in\N$, and such that $\psi_{w,n}:=w_{\mathrm{SL}_{n}}\circ(\psi_{n},\ldots,\psi_{n}):\mathbb{A}_{\Q}^{Nr}\rightarrow\mathrm{SL}_{n}$
is a polynomial map of degree $\leq36\ell(w)$. 
\end{lem}

\begin{proof}
Consider the following subgroups of $\mathrm{SL}_{n}$, 
\[
\mathrm{Mat}_{p}^{+}:=\left\{ \left(\begin{array}{cc}
I_{p} & A\\
0 & I_{n-p}
\end{array}\right):A\in\mathrm{Mat}_{p\times n-p}\right\} \text{ and }\,\,\mathrm{Mat}_{p}^{-}:=\left\{ \left(\begin{array}{cc}
I_{p} & 0\\
B & I_{n-p}
\end{array}\right):B\in\mathrm{Mat}_{n-p\times p}\right\} ,
\]
for $1\leq p<n$. By the proof of \cite[Lemma 2.1]{Kas07b}, any matrix
in $\mathrm{SL}_{n}(\overline{\Q})$ can be written as a product $A_{1}\cdots A_{M}$
of $M\leq18$ matrices, where $A_{i}\in\mathrm{Mat}_{p_{i}}^{\epsilon_{i}}(\overline{\Q})$
for $1\leq p_{i}<n$ and $\epsilon_{i}\in\{+,-\}$. In particular,
the map $\widetilde{\psi_{n}}:\prod_{i=1}^{M}\mathbb{A}_{\Q}^{(n-p_{i})p_{i}}\rightarrow\mathrm{SL}_{n}$
is surjective on $\mathrm{SL}_{n}(\overline{\Q})$. Since $\widetilde{\psi_{n}}$
is dominant, $J_{m}(\widetilde{\psi_{n}})$ is dominant, and thus
$J_{m}(\widetilde{\psi_{n}}*\widetilde{\psi_{n}})=J_{m}(\widetilde{\psi_{n}})*J_{m}(\widetilde{\psi_{n}})$
is surjective on $\left(J_{m}\mathrm{SL}_{n}\right)(\overline{\Q})$,
for each $m\in\N$. Set $\psi_{n}:=\widetilde{\psi_{n}}*\widetilde{\psi_{n}}$.
Finally, since the inverse of $A\in\mathrm{Mat}_{p}^{+}$ is $-A$,
it follows that $\psi_{w,n}:=w_{\mathrm{SL}_{n}}\circ(\psi_{n},\ldots,\psi_{n}):\left(\prod_{i=1}^{M}\mathbb{A}_{\Q}^{(n-p_{i})p_{i}}\right)^{2r}\rightarrow\mathrm{SL}_{n}$
is a polynomial map of degree $2M\ell(w)\leq36\ell(w)$.
\end{proof}
Note that by Lemmas \ref{lem:auxilary lemma} and \ref{lem:reduction to affine},
in order to prove Theorem \ref{thm:word maps on SLn are (FRS) after enough convolutions}
it is enough to show: 
\begin{prop}
\label{prop:reduction of (FRS) to affine case}For every $1\neq w\in F_{r}$
and $n\geq2$, the map $\psi_{w,n}^{*t}:\mathbb{A}_{\Q}^{Nrt}\rightarrow\mathrm{SL}_{n}$
is (FRS), for all $t\geq80\ell(w)n^{2}$. 
\end{prop}

For the proof of Proposition \ref{prop:reduction of (FRS) to affine case},
we introduce an analytic characterization of the (FRS) property. 
\begin{defn}
\label{def:epsilon of a map}Let $F$ be a local field, and $q\geq1$. 
\begin{enumerate}
\item If $X$ is a smooth $F$-manifold, we denote by $\mathcal{M}^{\infty}(X)$
(respectively $\mathcal{M}_{c}^{\infty}(X)$) the space of smooth
(respectively smooth and compactly supported) measures on $X$. 
\item An $F$-analytic map $\varphi:X\rightarrow Y$ between smooth $F$-analytic
manifolds is called an \emph{$L^{q}$-map} if, for every $\mu\in\mathcal{M}_{c}^{\infty}(X)$,
the pushforward measure $\varphi_{*}\mu$ is in $L^{q}(Y,\nu)$, with
respect to every nowhere-vanishing $\nu\in\mathcal{M}^{\infty}(Y)$. 
\end{enumerate}
\end{defn}

Let $\varphi:X\rightarrow Y$ be a morphism of smooth algebraic $\Q$-varieties.
In \cite[Theorem 3.4]{AA16}, Aizenbud and the first author showed
that $\varphi$ is (FRS) if and only if $\varphi_{F}:X(F)\rightarrow Y(F)$
is an $L^{\infty}$-map for every non-Archimedean local field $F$
of characteristic zero. This was extended to the Archimedean case
in \cite{Rei,GHS}: 
\begin{thm}[{\cite[Corollary 6.2]{GHS}, \cite{Rei}}]
\label{thm:analytic criterion for (FRS)}Let $\varphi:X\rightarrow Y$
be a morphism between smooth algebraic $\Q$-varieties. Then the following
are equivalent: 
\begin{enumerate}
\item $\varphi$ is (FRS). 
\item $\varphi_{\C}:X(\C)\rightarrow Y(\C)$ is an $L^{\infty}$-map. 
\item For every $F\in\{\R,\C\}$, $\varphi_{F}:X(F)\rightarrow Y(F)$ is
an $L^{\infty}$-map. 
\end{enumerate}
\end{thm}

\begin{prop}
\label{prop:reduction to Lq maps}For every $1\neq w\in F_{r}$ and
$n\geq2$, the map $(\psi_{w,n})_{\C}:\C^{Nr}\rightarrow\mathrm{SL}_{n}(\C)$
is an $L^{1+\epsilon}$-map, for every $\epsilon\leq\frac{1}{40\ell(w)n^{2}}$. 
\end{prop}

\begin{lem}
Proposition \ref{prop:reduction of (FRS) to affine case} follows
from Proposition \ref{prop:reduction to Lq maps}. 
\end{lem}

\begin{proof}
The convolution operation from Definition \ref{def:convolution} commutes
with the classical convolution from analysis, so that for every $\mu\in\mathcal{M}_{c}^{\infty}(\C^{Nr})$,
we have $((\psi_{w,n}^{*t})_{\C})_{*}\mu^{t}=\left(((\psi_{w,n})_{\C})_{*}\mu\right)^{*t}$.
Therefore, by Young's convolution inequality for unimodular groups
(see e.g.~\cite[Lemma 1.4]{BCD11}), if $(\psi_{w,n})_{\C}:\C^{Nr}\rightarrow\mathrm{SL}_{n}(\C)$
is an $L^{1+\epsilon}$-map, then $(\psi_{w,n}^{*t})_{\C}:\C^{Nrt}\rightarrow\mathrm{SL}_{n}(\C)$
is an $L^{\infty}$-map for $t\geq\frac{1+\epsilon}{\epsilon}$, so
in particular for $t\geq\max\left\{ 2,\frac{2}{\epsilon}\right\} $
(see \cite[Section 1.1, end of p.3]{GHS}). 
\end{proof}
We now turn to the proof of Proposition \ref{prop:reduction to Lq maps},
from which Theorem \ref{thm:word maps on SLn are (FRS) after enough convolutions}
follows. For the proof we need the notion of log-canonical threshold. 
\begin{defn}
\label{def:log canonical threshold}If $\mathfrak{a}=\langle f_{1},\ldots,f_{m}\rangle\subseteq\C[x_{1},\ldots,x_{n}]$
is an ideal, and $x\in\C^{n}$, define 
\begin{equation}
\operatorname{lct}(\mathfrak{a};x):=\sup\left\{ s>0:\exists U\ni x\text{ \,s.t.\, }\forall\mu\in\mathcal{M}_{c}^{\infty}(U),\,\int_{\C^{n}}\min_{1\leq i\leq m}\Big[\left|f_{i}(z)\right|^{-2s}\Big]\mu(z)<\infty\right\} ,\label{eq:lct-analytic definition}
\end{equation}
where $U\subseteq\C^{n}$ runs over all open neighborhoods of $x$.
We define $\operatorname{lct}(\mathfrak{a}):=\min\limits _{x\in X}\operatorname{lct}(\mathfrak{a};x)$.
This definition does not depend on the choice of the generators of
$\mathfrak{a}$ (see e.g.~\cite[Theorem 1.2]{Mus12}). 
\end{defn}

Note that it is immediate from Definition \ref{def:log canonical threshold}
that 
\begin{equation}
\mathfrak{a}\subseteq\mathfrak{b}\Longrightarrow\lct(\mathfrak{a};x)\leq\lct(\mathfrak{b};x).\label{eq:lct and containment of ideals}
\end{equation}

\begin{proof}[Proof of Proposition \ref{prop:reduction to Lq maps}]
Embed $\mathrm{SL}_{n}(\C)\subseteq\C^{n^{2}}$ and let $\pi_{k,l}:\mathrm{SL}_{n}(\C)\rightarrow\C^{n^{2}-1}$
be the restriction of the projection $(A_{ij})_{1\leq i,j\leq n}\rightarrow(A_{ij})_{1\leq i,j\leq n,(i,j)\neq(k,l)}$.
Then by Lemma \ref{lem:reduction to affine}, the map $\Psi:=\pi_{k,l}\circ(\psi_{w,n})_{\C}:\C^{Nr}\rightarrow\C^{n^{2}-1}$
is a dominant polynomial map of degree $\leq36\ell(w)$. Let $D_{x}\Psi:\C^{Nr}\rightarrow\C^{n^{2}-1}$
denotes the differential of $\Psi$ at $x\in\C^{Nr}$ and, for each
subset $I\subseteq[Nr]$ of size $n^{2}-1$, let $M_{I}$ be the corresponding
$I$-th minor of $D_{x}\Psi$. Denote by $\mathcal{J}_{\Psi}:=\langle\left\{ M_{I}\right\} _{I}\rangle$
the \emph{Jacobian ideal} of $\Psi$, generated by the collection
of those minors. By \cite[Theorem 1.1]{GHS} it follows that $\Psi$
is an $L^{1+\epsilon}$-map, for every $\epsilon<\lct(\mathcal{J}_{\Psi})$.
Note that $\lct(\mathcal{J}_{\Psi})\geq\lct(\langle M_{I_{0}}\rangle)$
for any fixed $I_{0}$, by (\ref{eq:lct and containment of ideals}).
By e.g.~\cite[Eq. 20.1]{Kol}, taking $M_{I_{0}}$ not identically
zero, and since $\deg(M_{I_{0}})\leq(n^{2}-1)\cdot\deg(\Psi)\leq36\ell(w)n^{2}$,
we get that $\lct(\langle M_{I_{0}}\rangle)\geq\frac{1}{36\ell(w)n^{2}}$.
In particular, $\Psi$ is an $L^{1+\epsilon}$-map for $\epsilon=\frac{1}{40\ell(w)n^{2}}$.
Since for any $g\in\mathrm{SL}_{n}(\C)$, one can find $1\leq k,l\leq n$
such that $\pi_{k,l}$ is a local diffeomorphism at $g$, it follows
that $(\psi_{w,n})_{\C}$ is an $L^{1+\epsilon}$-map for $\epsilon=\frac{1}{40\ell(w)n^{2}}$,
as required. 
\end{proof}
We can now use Theorems \ref{thm:word maps on SLn are (FRS) after enough convolutions}
and \ref{thm:analytic criterion for (FRS)} to deduce Corollary \ref{cor:bounded density for small rank}. 
\begin{proof}[Proof of Corollary \ref{cor:bounded density for small rank}]
Let $(\mathrm{SL}_{n})_{\C/\R}$ be the restriction of scalars of
$\mathrm{SL}_{n}$. Explicitly, $(\mathrm{SL}_{n})_{\C/\R}$ is defined
as the collection of pairs $(A,B)\in\mathrm{Mat}_{n\times n}$ such
that $\det(A+Bi)=1$. Note that $(\mathrm{SL}_{n})_{\C/\R}(\R)\simeq\mathrm{SL}_{n}(\C)$.
Let 
\[
\underline{\mathrm{SU}}_{n}:=\left\{ (A,B)\in(\mathrm{SL}_{n})_{\C/\R}:((A-Bi)^{t})(A+Bi)=I_{n}\right\} .
\]
Then $\underline{\mathrm{SU}}_{n}(\R)\simeq\SU_{n}$, and $\underline{\mathrm{SU}}_{n}(\C)\simeq\mathrm{SL}_{n}(\C)$.
By Theorem \ref{thm:word maps on SLn are (FRS) after enough convolutions},
the base change to $\C$ of the morphism $w_{\underline{\mathrm{SU}}_{n}}^{*t}:\underline{\mathrm{SU}}_{n}^{rt}\rightarrow\underline{\mathrm{SU}}_{n}$
is (FRS) for $t\geq80\ell(w)n^{2}$, and therefore $w_{\underline{\mathrm{SU}}_{n}}^{*t}:\underline{\mathrm{SU}}_{n}^{rt}\rightarrow\underline{\mathrm{SU}}_{n}$
is (FRS) as well, for $t\geq80\ell(w)n^{2}$. By Theorem \ref{thm:analytic criterion for (FRS)},
$w_{\SU_{n}}^{*t}:\SU_{n}^{rt}\rightarrow\SU_{n}$ is an $L^{\infty}$-map,
which implies the proposition. 
\end{proof}

\section{\label{sec:small ball}Boundedness of the density of word measures
and small ball estimates }

In this Section we prove Theorems \ref{thm: A, bounded density} and
\ref{thm:small ball estimate}. We start with a short discussion on
Fourier coefficients of word measures and their convolutions. 
\begin{defn}
\label{def:Concatenation of words}Given $w_{1}\in F_{r_{1}}$ and
$w_{2}\in F_{r_{2}}$, we denote by $w_{1}*w_{2}\in F_{r_{1}+r_{2}}$
their \emph{concatenation} with disjoint sets of letters\emph{. }For
example, if $w=[x,y]$, then $w*w=[x,y]\cdot[z,w]$. We denote by
$w^{*t}$ the $t$-th self concatenation of $w\in F_{r}$. 
\end{defn}

Given $w\in F_{r}$ and a compact group $G$, we denote by $\tau_{w,G}$
the distribution of $w_{G}(\mathsf{G}_{1},\ldots,\mathsf{G}_{r})$,
where $\mathsf{G}_{1},\ldots,\mathsf{G}_{r}$ are independent, Haar-distributed
random elements in $G$. For $\rho\in\Irr(G)$, let
\begin{equation}
a_{w,G,\rho}:=\int_{G}\rho(g)\tau_{w,G}=\E\left(\rho(w(\mathsf{G}_{1},\ldots,\mathsf{G}_{r}))\right),\label{eq:Fourier coefficients}
\end{equation}
be the Fourier coefficient of the measure $\tau_{w,G}$ corresponding
to $\rho$. Since $\tau_{w,G}$ is conjugate invariant, we have the
following equality of distributions $\tau_{w,G}=\sum_{\rho\in\Irr(G)}\overline{a_{w,G,\rho}}\cdot\rho$.
By Definition \ref{def:Concatenation of words}, we see that $\tau_{w_{1}*w_{2},G}=\tau_{w_{1},G}*\tau_{w_{2},G}$
for every $w_{1}\in F_{r_{1}}$ and $w_{2}\in F_{r_{2}}$. Since $\rho_{1}*\rho_{2}=\frac{\delta_{\rho_{1},\rho_{2}}}{\rho_{1}(1)}\cdot\rho_{1}$
for every $\rho_{1},\rho_{2}\in\Irr(G)$, we have:
\begin{equation}
a_{w_{1}*w_{2},G,\rho}=\int_{G}\rho(g)\tau_{w_{1}*w_{2},G}(g)=\int_{G}\rho(g)\tau_{w_{1},G}*\tau_{w_{2},G}(g)=\frac{a_{w_{1},G,\rho}\cdot a_{w_{2},G,\rho}}{\rho(1)}.\label{eq:Fourier coefficients of convolution}
\end{equation}

\subsection{\label{subsec:Proof of Thm A}Proof of Theorem \ref{thm: A, bounded density}}
\begin{proof}[Proof of Theorem \ref{thm: A, bounded density}]
We first prove Item (2). Let $t>3\cdot2^{10}\cdot\ell(w)$. By (\ref{eq:Fourier coefficients of convolution})
and by Theorem \ref{thm:Detalied Theorem B}, 
\[
\left|a_{w^{*t},\SU_{n},\rho}\right|=\left|\frac{a_{w,\SU_{n},\rho}^{t}}{\rho(1)^{t-1}}\right|<\rho(1)^{-2},
\]
whenever $\rho(1)>2^{512(\ell(w)+1)n^{2}\log n}$. By \cite[Theorem 5.1]{LL08},
the sum 
\[
\sum_{\rho\in\Irr(\SU_{n})}\left|\overline{a_{w^{*t},\SU_{n},\rho}}\right|\rho(1)\leq\sum_{\rho(1)\leq2^{512(\ell(w)+1)n^{2}\log n}}\rho(1)^{2}+\sum_{\rho(1)>2^{512(\ell(w)+1)n^{2}\log n}}\rho(1)^{-1}
\]
converges, so $\tau_{w,\SU_{n}}^{*t}=\tau_{w^{*t},\SU_{n}}\in L^{\infty}(\SU_{n})$.

Item (1) follows from Item (2) for $n>8\ell(w)$ and from Corollary
\ref{cor:bounded density for small rank} for $n\leq8\ell(w)$. 
\end{proof}

\subsection{\label{subsec:Small-ball-estimates}Small ball estimates}

Let $\langle A,B\rangle_{\mathrm{HS}}:=\tr(AB^{*})$ be the Hilbert\textendash Schmidt
inner product on $\Mat_{n}(\C)$ and let $\|A\|_{\mathrm{HS}}:=\langle A,A\rangle_{\mathrm{HS}}^{1/2}$
be the Hilbert\textendash Schmidt norm. We remind the reader that
we denote by $\gamma$ the bi-invariant Riemannian metric on $\SU_{n}$
induced from the inner product $\langle A,B\rangle_{\widetilde{\mathrm{HS}}}=\beta_{n}^{-1}\cdot\langle A,B\rangle_{\mathrm{HS}}$
on the Lie algebra $\mathfrak{su}_{n}$, with $\beta_{n}:=\frac{4\pi^{2}\lfloor n/2\rfloor\lceil n/2\rceil}{n}$.

Note that $\langle A,B\rangle_{\widetilde{\mathrm{HS}}}=-(2n\beta_{n})^{-1}\kappa(A,B)$,
where $\kappa$ is the Killing form on $\mathfrak{su}_{n}$. It follows
that 
\[
\Vol_{\gamma}(\SU_{n})=\frac{1}{(2n\beta_{n})^{\frac{n^{2}-1}{2}}}\Vol_{\kappa}(\SU_{n}).
\]
By a computation of Macdonald \cite{Mac80}, which was simplified
in \cite[Eq. (7)]{MV}, 
\[
\Vol_{\kappa}(\SU_{n})=\frac{2^{\frac{n^{2}-1}{2}}\cdot n^{\frac{n^{2}}{2}}\cdot(2\pi)^{\frac{n^{2}+n-2}{2}}}{\prod_{k=1}^{n-1}k!},
\]
so 
\begin{equation}
\mu_{\SU_{n}}=\alpha_{n}\Vol_{\gamma},\quad\text{ where }\quad\alpha_{n}=\frac{\beta_{n}^{\frac{n^{2}-1}{2}}\prod_{k=1}^{n-1}k!}{(2\pi)^{\frac{n^{2}+n-2}{2}}\sqrt{n}}.\label{eq:SUn.volume}
\end{equation}
We denote by $B(g,\delta)$ the $\gamma$-ball of radius $\delta$
around $g\in\SU_{n}$ and denote the norm on $\Mat_{n}(\C)$ induced
by $\langle-,-\rangle_{\widetilde{\mathrm{HS}}}$ by $\|\cdot\|_{\widetilde{\mathrm{HS}}}$. 
\begin{lem}
\label{lem:compare to hilbert Schmidt norm}For every $h_{1},h_{2}\in\SU_{n}$, 
\begin{enumerate}
\item 
\[
\left\Vert h_{1}-h_{2}\right\Vert _{\widetilde{\mathrm{HS}}}\leq\gamma(h_{1},h_{2})\leq\frac{\pi}{2}\left\Vert h_{1}-h_{2}\right\Vert _{\widetilde{\mathrm{HS}}}.
\]
\item 
\[
B(h_{1},\delta_{1})\cdot B(h_{2},\delta_{2})\subseteq B(h_{1}h_{2},\delta_{1}+\delta_{2}).
\]
\end{enumerate}
\end{lem}

\begin{proof}
Item (1) follows e.g.~from \cite[Lemma 1.3]{Mec19}. Item (2) holds
for any bi-invariant metric on a group. 
\end{proof}
\begin{lem}[{Bishop\textendash Gromov Volume Comparison Theorem, \cite[Section 11.10, Corollary 4]{BC64}}]
\label{lem:Bishop=002013Gromov}Let $(M,g)$ be a complete $d$-dimensional
Riemannian manifold with a non-negative Ricci curvature. Denote the
volume induced from $g$ by $\Vol_{g}$. Let $p\in M$ and let $B_{M}(p,\delta)$
be the $g$-ball of radius $\delta$ around $p$. Then the function
$\delta\mapsto\frac{\Vol_{g}(B_{M}(p,\delta))}{\delta^{d}}$ is non-increasing,
and $\Vol_{g}(B_{M}(p,\delta))\leq\Vol_{\mathrm{Euc}}(B_{\R^{d}}(0,\delta))$. 
\end{lem}

\begin{cor}
\label{cor:Estimates on volume of metric balls}The following hold: 
\begin{enumerate}
\item For every $\delta>0$, every $t\geq1$, and every $h\in\SU_{n}$,
we have $\mu_{\SU_{n}}\left(B(h,t\delta)\right)\leq t^{n^{2}-1}\mu_{\SU_{n}}\left(B(h,\delta)\right)$. 
\item For every $0<\delta<1$ and every $h\in\SU_{n}$, we have $\delta^{n^{2}-1}\leq\mu_{\SU_{n}}\left(B(h,\delta)\right)\leq6^{n^{2}}\delta^{n^{2}-1}$. 
\end{enumerate}
\end{cor}

\begin{proof}
By \cite[Theorem 2.2]{Mil76}, the metric $\gamma$ has positive Ricci
curvature. Equation (\ref{eq:SUn.volume}) and Lemma \ref{lem:Bishop=002013Gromov}
imply that 
\[
\frac{\mu_{\SU_{n}}\left(B(h,\delta t)\right)}{(\delta t)^{n^{2}-1}}=\alpha_{n}\frac{\Vol_{\gamma}B(h,\delta t)}{(\delta t)^{n^{2}-1}}\leq\alpha_{n}\frac{\Vol_{\gamma}B(h,\delta)}{\delta^{n^{2}-1}}=\frac{\mu_{\SU_{n}}\left(B(h,\delta)\right)}{\delta^{n^{2}-1}},
\]
proving Item (1).

The left inequality of Item (2) follows from $1=\mu_{\SU_{n}}\left(B(h,1)\right)\leq\mu_{\SU_{n}}\left(B(h,\delta)\right)\delta^{-n^{2}+1}$.
For the right inequality of Item (2), recall that the volume of an
Euclidean $(n^{2}-1)$-ball is 
\begin{equation}
\Vol(B_{\R^{n^{2}-1}}(h,\delta))=\frac{\pi^{\frac{n^{2}-1}{2}}}{\Gamma(\frac{n^{2}+1}{2})}\delta^{n^{2}-1}.\label{eq:volume of n-ball}
\end{equation}
Thus, by Lemma \ref{lem:Bishop=002013Gromov}, and Equation (\ref{eq:volume of n-ball}),
\[
\mu_{\SU_{n}}\left(B(h,\delta)\right)=\alpha_{n}\Vol_{\gamma}B(h,\delta)\leq\alpha_{n}\Vol B_{\R^{n^{2}-1}}(h,\delta)=\frac{\prod_{k=1}^{n-1}k!\beta_{n}^{\frac{n^{2}-1}{2}}}{(2\pi)^{\frac{n^{2}+n-2}{2}}\sqrt{n}}\frac{\pi^{\frac{n^{2}-1}{2}}\delta^{n^{2}-1}}{\Gamma(\frac{n^{2}+1}{2})}.
\]
Since $\prod_{k=1}^{n-1}k!\leq n^{n-1}\cdot n^{n-2}\cdots n^{1}=n^{\frac{n^{2}-n}{2}}$,
$\beta_{n}\leq\pi^{2}n$, and $\Gamma\left(\frac{n^{2}+1}{2}\right)\geq\Gamma\left(\frac{n^{2}}{2}\right)\geq\frac{1}{n}\left(\frac{n^{2}}{2e}\right)^{\frac{n^{2}}{2}}$,
we get 
\[
\mu_{\SU_{n}}\left(B(h,\delta)\right)\leq\frac{n^{\frac{n^{2}-n}{2}}(\pi^{2}n)^{\frac{n^{2}-1}{2}}}{(2\pi)^{\frac{n^{2}+n-2}{2}}\sqrt{n}}\frac{\pi^{\frac{n^{2}-1}{2}}\delta^{n^{2}-1}}{\frac{1}{n}\left(\frac{n^{2}}{2e}\right)^{n^{2}/2}}=n^{-\frac{n}{2}}\cdot\pi^{n^{2}-\frac{n}{2}-\frac{1}{2}}\cdot e^{\frac{n^{2}}{2}}\cdot2^{1-\frac{n}{2}}\cdot\delta^{n^{2}-1}\leq(\pi\sqrt{e})^{n^{2}}\delta^{n^{2}-1}\leq6^{n^{2}}\delta^{n^{2}-1}.\qedhere
\]
\end{proof}
\begin{proof}[Proof of Theorem \ref{thm:small ball estimate}]
We first prove the claim for $n\leq8\ell(w)$. Setting $t=5120\ell(w)^{3}\geq80\ell(w)n^{2}$,
Corollary \ref{cor:bounded density for small rank} implies that there
is a constant $C(n)\geq1$ such that the density of the measure $\tau_{w,\SU_{n}}^{*t}$
is bounded by $C(n)$. Thus, by Lemma \ref{lem:compare to hilbert Schmidt norm}
and Corollary \ref{cor:Estimates on volume of metric balls}(1), for
every $g\in\SU_{n}$, 
\[
\Big(\tau_{w,\SU_{n}}\left(B(g,\delta)\right)\Big)^{t}\leq\tau_{w,\SU_{n}}^{*t}\left(B(g^{t},t\delta)\right)\leq C(n)\mu_{\SU_{n}}\left(B(I_{n},t\delta)\right)\leq C(n)t^{n^{2}-1}\mu_{\SU_{n}}\left(B(I_{n},\delta)\right).
\]
If $\delta<\frac{1}{64\cdot t^{2}\cdot C(n)}$ then $C(n)t^{n^{2}-1}\mu_{\SU_{n}}\left(B(I_{n},\delta)\right)^{\frac{1}{2}}\leq1$
and we get
\[
\mathbb{P}\Big(w(\mathsf{G}_{1},\ldots,\mathsf{G}_{r})\in B(g,\delta)\Big)=\tau_{w,\SU_{n}}\left(B(g,\delta)\right)\leq\left(\mu_{\SU_{n}}\left(B(I_{n},\delta)\right)\right)^{\frac{1}{2t}}.
\]
Next, we prove the theorem for $g=I_{n}$ and $n>8\ell(w)$. We will
show that (\ref{eq:small ball estimate}) holds for $\epsilon(w)=\frac{1}{32(\ell(w)+1)}$
and $\delta(w)=2^{-128(\ell(w)+1)}$. Note that if $h\in\SU_{n}$
satisfies $\left\Vert h-I_{n}\right\Vert _{\mathrm{HS}}<\pi\sqrt{n}\delta$
then $h$ must have at least $n/2$ eigenvalues whose Euclidean distance
from $1$ is $\leq\sqrt{2}\pi\delta<\frac{16}{\pi}\delta$, and hence
it is not $(\frac{1}{2},8\delta)$-spread. By Lemma \ref{lem:compare to hilbert Schmidt norm},
\begin{align}
\mu_{\SU_{n}^{r}}\left(w^{-1}(B(I_{n},\delta))\right) & \leq\mu_{\SU_{n}^{r}}\left\{ (g_{1},\ldots,g_{r})\in\SU_{n}^{r}:\left\Vert w(g_{1},\ldots,g_{r})-I_{n}\right\Vert _{\mathrm{\widetilde{\mathrm{HS}}}}<\delta\right\} \nonumber \\
 & \leq\mu_{\SU_{n}^{r}}\left\{ (g_{1},\ldots,g_{r})\in\SU_{n}^{r}:\left\Vert w(g_{1},\ldots,g_{r})-I_{n}\right\Vert _{\mathrm{HS}}<\pi\sqrt{n}\delta\right\} \nonumber \\
 & \leq\mu_{\SU_{n}^{r}}\left\{ (g_{1},\ldots,g_{r})\in\SU_{n}^{r}:\text{\ensuremath{w(g_{1},\ldots,g_{r})} is not \ensuremath{\left(\frac{1}{2},8\delta\right)}-spread}\right\} .\label{eq:tranlating to Hilbert-Schmidt}
\end{align}
For every $a_{1},\ldots,a_{r}\in\U_{1}$ and every $g_{1},\ldots,g_{r}\in\SU_{n}$,
the element $w(g_{1},\ldots,g_{r})$ is $\left(\frac{1}{2},8\delta\right)$-spread
if and only if the element $w(a_{1}g_{1},\ldots,a_{r}g_{r})$ is $\left(\frac{1}{2},8\delta\right)$-spread.
Thus, 
\[
\mu_{\SU_{n}^{r}}\left(w^{-1}(B(I_{n},\delta))\right)\leq\mu_{\U_{n}^{r}}\left\{ (g_{1},\ldots,g_{r})\in\U_{n}^{r}:\text{\ensuremath{w(g_{1},\ldots,g_{r})} is not \ensuremath{\left(\frac{1}{2},8\delta\right)}-spread}\right\} ,
\]
so, by Theorem \ref{thm:probabilistic result} and Corollary \ref{cor:Estimates on volume of metric balls}(2),
\[
\mu_{\SU_{n}^{r}}\left(w^{-1}(B(I_{n},\delta))\right)\leq2^{3n^{2}}\left(8\delta\right)^{\frac{n^{2}}{16(\ell(w)+1)}}<2^{4n^{2}}\delta^{\frac{n^{2}}{16(\ell(w)+1)}}<\delta^{\frac{n^{2}}{32(\ell(w)+1)}}<\left(\mu_{\SU_{n}}(B(I_{n},\delta))\right)^{\frac{1}{32(\ell(w)+1)}}.
\]
Finally, we prove the claim for an arbitrary element $g\in\U_{n}$.
Given $w\in F_{r}$, let $\widetilde{w}\in F_{2r}$ be the word $\widetilde{w}(x_{1},\ldots,x_{r},y_{1},\ldots,y_{r})=w(x_{1},\ldots,x_{r})\cdot w(y_{1},\ldots,y_{r})^{-1}$.
For each $g\in\SU_{n}$ ,we have: 
\[
\mu_{\SU_{n}^{r}}\left(w^{-1}\left(B(g,\delta)\right)\right)^{2}\leq\mu_{\SU_{n}^{2r}}\left(\widetilde{w}^{-1}\left(B(I_{n},2\delta)\right)\right)\leq\left(\mu_{\SU_{n}}\left(B(I_{n},2\delta)\right)\right)^{\epsilon(\widetilde{w})}\leq\left(2^{n^{2}-1}\mu_{\SU_{n}}\left(B(I_{n},\delta)\right)\right)^{\epsilon(\widetilde{w})}.
\]
If $\delta<2^{-9}$, then, by Corollary \ref{cor:Estimates on volume of metric balls},
$2^{n^{2}-1}<\mu_{\SU_{n}}(B(I_{n},\delta))^{-\frac{1}{2}}$, so 
\[
\mu_{\SU_{n}^{r}}\left(w^{-1}\left(B(g,\delta)\right)\right)^{2}\leq\mu_{\SU_{n}}\left(B(I_{n},\delta)\right)^{\frac{\epsilon(\widetilde{w})}{2}}.
\]
Since $\epsilon(\widetilde{w})=\frac{1}{32(2\ell(w)+1)}>\frac{1}{64(\ell(w)+1)}$,
the proof is finished. 
\end{proof}

\section{\label{sec:Fourier-coefficients-of power word}Fourier coefficients
of the power word}

In this section we provide estimates for the Fourier coefficients
of the power word $w=x^{\ell}$, thus proving Theorem \ref{thm D:Fourier coefficients of power words}.
We will show in Proposition \ref{prop:reduction to multiplicities}
that for each $\rho\in\Irr(\U_{n})$ the Fourier coefficient $\E_{X\in\U_{n}}(\rho(X^{\ell}))$
is precisely the dimension of the subspace of $H$-invariants in the
representation with character $\rho$, where $H\leqslant\U_{n}$ is
a certain product of smaller unitary groups. In order to calculate
$\dim\rho^{H}$, we need a lemma on the multiplicities of the irreducible
characters appearing in the restriction of a representation $\rho\in\Irr(\U_{n})$
to $\U_{n-k}\times\U_{k}$.

Throughout the section we use the following notation. For a Lie group
$G$, a representation is a pair $(\pi,V_{\pi})$, with $\pi:G\rightarrow\mathrm{GL}(V_{\pi})$
a continuous homomorphism. We denote the dual of $(\pi,V_{\pi})$
by $(\pi^{\vee},V_{\pi}^{\vee})$. Given representations $(\pi,V_{\pi})$
and $(\pi',V_{\pi'})$ of $G$, we denote by $\langle\pi,\pi'\rangle_{G}:=\dim\Hom_{G}(\pi,\pi')$
the dimension of the space of $G$-homomorphisms. If $H\leq G$ is
a closed subgroup of a compact group, and $(\tau,V_{\tau})$ is a
finite dimensional representation of $H$, we denote by $(\mathrm{Ind}_{H}^{G}\tau,\mathrm{Ind}_{H}^{G}V_{\tau})$
the induced representation, where 
\[
\mathrm{Ind}_{H}^{G}V_{\tau}:=\left\{ f:G\rightarrow V_{\tau}:f(gh^{-1})=\tau(h)f(g)\text{ and }f\text{ is continuous}\right\} .
\]

\begin{lem}
\label{lem:multiplicities in spherical pairs}Let $G$ be a complex
connected reductive group, and $H\leq G$ be a spherical subgroup\footnote{A subgroup $H\leq G$ is \emph{spherical} if there exists an open
dense $B$-orbit in $G/H$, where $B$ is a Borel subgroup of $G$}. Then, for every finite dimensional irreducible representations $(\pi,V_{\pi})$
of $G$ and $(\tau,V_{\tau})$ of $H$, 
\[
\langle\pi|_{H},\tau\rangle_{H}\leq\dim\tau.
\]
\end{lem}

\begin{proof}
Let $T\subseteq G$ be a maximal torus and let $T\subseteq B$ be
a Borel subgroup. Denote the highest weight of $\pi$ by $\lambda$
and let $v_{\pi}\in V_{\pi}$ be a highest weight vector. Extending
$\lambda$ to $B$ via the projection $B\rightarrow T$, the vector
$v_{\pi}$ is $(B,\lambda)$-equivariant. By assumption, there is
$g_{0}\in G$ such that $Hg_{0}B$ is dense in $G$. We claim that
every $H$-equivariant linear map $\varphi:V_{\pi}\rightarrow V_{\tau}$
is determined by $\varphi(g_{0}v_{\pi})$, which implies the lemma.
Indeed, if $h\in H$ and $b\in B$, then $\varphi(hg_{0}b.v_{\pi})=\lambda(b)\cdot h.\left(\varphi(g_{0}v_{\pi})\right)$.
Since $Hg_{0}B$ is Zariski dense and $Gv_{\pi}$ spans $V_{\pi}$,
$\varphi$ is determined. 
\end{proof}
\begin{prop}
\label{prop:bounds}Let $\rho\in\mathrm{Irr}(\U_{n})$ and let $\tau\in\mathrm{Irr}(\U_{n-k}\times\U_{k})$,
then: 
\begin{enumerate}
\item We have 
\[
\langle\rho|_{\U_{n-k}\times\U_{k}},\tau\rangle_{\U_{n-k}\times\U_{k}}\leq\dim\tau.
\]
\item Fix $0<\alpha\leq\frac{1}{2}$. Then for any $\delta>0$, there exists
$N(\delta)>0$ such that for every $n>N(\delta)$ and every $\alpha n\leq k\leq\frac{n}{2}$,
the number of characters $\sigma\boxtimes\sigma'\in\mathrm{Irr}(\U_{n-k}\times\U_{k})$,
such that $\left\langle \sigma\boxtimes\sigma',\rho|_{\U_{n-k}\times\U_{k}}\right\rangle _{\U_{n-k}\times\U_{k}}\neq0$
is at most $\rho(1)^{\delta}$. 
\end{enumerate}
\end{prop}

\begin{proof}
1) Since the pair $\left(\mathrm{GL}_{n}(\C),\mathrm{GL}_{n-k}(\C)\times\mathrm{GL}_{k}(\C)\right)$
is the complexification of the pair $\left(\U_{n},\U_{n-k}\times\U_{k}\right)$,
each $\rho\in\mathrm{Irr}(\U_{n})$ (resp. $\tau\in\mathrm{Irr}(\U_{n-k}\times\U_{k})$)
can be extended to the character $\widetilde{\rho}$ of a representation
of $\mathrm{GL}_{n}(\C)$ (resp. $\widetilde{\tau}$ of $\mathrm{GL}_{n-k}(\C)\times\mathrm{GL}_{k}(\C)$)
(see e.g.~\cite[Lemma 5.30]{Kna01}). In particular, we have 
\[
\langle\rho|_{\U_{n-k}\times\U_{k}},\tau\rangle_{\U_{n-k}\times\U_{k}}=\langle\widetilde{\rho}|_{\mathrm{GL}_{n-k}(\C)\times\mathrm{GL}_{k}(\C)},\widetilde{\tau}\rangle_{\mathrm{GL}_{n-k}(\C)\times\mathrm{GL}_{k}(\C)}.
\]
Since $\left(\mathrm{GL}_{n}(\C),\mathrm{GL}_{n-k}(\C)\times\mathrm{GL}_{k}(\C)\right)$
is a symmetric pair, it is also a spherical pair (see e.g.~\cite[Theorem 26.14 and Table 26.3]{Tim11}),
and hence Item (1) now follows from Lemma \ref{lem:multiplicities in spherical pairs}.

2) Given $\delta$, \cite[Corollary 3]{GLM12} implies that there
is $\widetilde{N}(\delta)$ such that if $G$ is a compact connected
simple Lie group of rank greater than $\widetilde{N}(\delta)$ and
$R_{d}(G)$ is the number of irreducible representations of $G$ of
degree $\leq d$, then $R_{d}(G)\leq d^{\delta/2}$. Hence, if $n>N_{1}(\delta):=\frac{1}{\alpha}\widetilde{N}(\delta)$,
then, for every $d$, $R_{d}(\SU_{n-k}\times\SU_{k})\leq d^{\delta}$.
In particular, $R_{\rho(1)}(\SU_{n-k}\times\SU_{k})\leq\rho(1)^{\delta}$,
proving Item (2) for $\SU_{n}$.

We now prove Item (2) for $\U_{n}$. Write $\rho=\rho_{\lambda,n}$,
where $\lambda\in\Lambda_{n}$. After twisting by $\det$, we can
assume that $\lambda_{n}=0$, i.e., $\lambda$ is a partition into
at most $n-1$ parts. Denote $m=\left|\lambda\right|$. Let $\mu\in\Lambda_{k},\nu\in\Lambda_{n-k}$
be such that $\rho_{\mu,k}\boxtimes\rho_{\nu,n-k}\hookrightarrow\rho$.
By assumption, $\rho$ is polynomial, so the representations $\rho_{\mu,k},\rho_{\nu,n-k}$
are also polynomial. Thus, $\mu$ and $\nu$ are partitions. By looking
at the restrictions of $\rho$ and $\rho_{\mu,k}\boxtimes\rho_{\nu,n-k}$
to the center of $\U_{n}$, we get that $\left|\mu\right|+\left|\nu\right|=m$
and, in particular, $0\leq\mu_{k}\leq\frac{m}{\alpha n}$ and $0\leq\nu_{n-k}\leq\frac{m}{(1-\alpha)n}$.
Denoting 
\[
S_{\mu}:=\mu+\mathbb{Z}\cdot(1^{k})=\left\{ \mu'\in\Lambda_{k}\mid\rho_{\mu,k}|_{\SU_{k}}=\rho_{\mu',k}|_{\SU_{k}}\right\} 
\]
and, similarly, $S_{\nu}:=\nu+\mathbb{Z}\cdot(1^{n-k})$, we get that
\[
\left|\left\{ (\mu',\nu')\in S_{\mu}\times S_{\nu}\mid\rho_{\mu',k}\boxtimes\rho_{\nu',n-k}\hookrightarrow\rho\right\} \right|\leq\frac{m^{2}}{\alpha(1-\alpha)n^{2}}.
\]
It follows that 
\begin{align}
S_{k,\rho} & :=\left|\left\{ \tau\in\Irr(\U_{k}\times\U_{n-k})\mid\tau\hookrightarrow\rho|_{\U_{k}\times\U_{n-k}}\right\} \right|\label{eq:bound on the number of reps}\\
 & \leq\frac{m^{2}}{\alpha(1-\alpha)n^{2}}\left|\left\{ \tau'\in\Irr(\SU_{k}\times\SU_{n-k})\mid\tau'\hookrightarrow\rho|_{\SU_{k}\times\SU_{n-k}}\right\} \right|.\nonumber 
\end{align}
Since we proved Item (2) for $G=\SU_{n}$, for every $n>N_{1}(\delta/3)$,
we have: 
\begin{equation}
\Big|\left\{ \tau'\in\Irr(\SU_{k}\times\SU_{n-k})\mid\tau'\hookrightarrow\rho|_{\SU_{k}\times\SU_{n-k}}\right\} \Big|\leq\rho(1)^{\delta/3}.\label{eq:bound on reps for SUn}
\end{equation}
In the notation of Definition \ref{def:first and last fundamental weights},
since $\lambda=\sum_{i=1}^{n}a_{i}\varpi_{i}$ is a partition (with
$a_{n}=0$), then $\lambda_{+}^{\vee,n},\lambda_{-}$ are both partitions,
with $\ell(\lambda_{-}),\ell(\lambda_{+}^{\vee,n})\leq\frac{n}{2}$.
Moreover, $\left|\lambda_{-}\right|+\left|\lambda_{+}\right|=m$,
and by the construction of $\lambda_{+}^{\vee,n}$ we have $\left|\lambda_{-}\right|+\left|\lambda_{+}^{\vee,n}\right|\geq(\lambda_{-})_{1}+(\lambda_{+}^{\vee,n})_{1}=\lambda_{1}\geq\frac{m}{n}$.
By Remark \ref{rem:useful fact} and Lemma \ref{lem:dimension of small representations}(2),
we have 
\begin{equation}
\rho(1)=\rho_{\lambda,n}(1)\geq\max\left\{ \rho_{\lambda_{-},n}(1),\rho_{\lambda_{+},n}(1)\right\} =\max\left\{ \rho_{\lambda_{-},n}(1),\rho_{\lambda_{+}^{\vee,n},n}(1)\right\} \geq\min\left\{ 2^{\frac{m}{8n}},2^{\frac{n}{2}}\right\} .\label{eq:lower bound on rho(1)}
\end{equation}
If $m\geq4n^{2}$ then $(\frac{m}{n})^{2}\leq2^{\frac{\delta m}{24n}}\leq\rho(1)^{\delta/3}$
for $n\gg_{\delta}1$. If $m<4n^{2}$ then $(\frac{m}{n})^{2}<16n^{2}\leq2^{\frac{\delta n}{6}}\leq\rho(1)^{\delta/3}$
for $n\gg_{\delta}1$. Hence, by combining (\ref{eq:bound on the number of reps})
and (\ref{eq:bound on reps for SUn}), for $n\gg_{\alpha,\delta}1$,
we have 
\[
S_{k,\rho}\leq\frac{m^{2}}{\alpha(1-\alpha)n^{2}}\rho(1)^{\delta/3}\leq(\frac{m}{n})^{2}\rho(1)^{2\delta/3}\leq\rho(1)^{\delta}.
\]
as required. 
\end{proof}

\subsection{\label{subsec:Fourier-coefficients-of power word}Fourier coefficients
of the power word}

We now turn to prove Theorem \ref{thm D:Fourier coefficients of power words}.
Let $\ell\geq2$ and let $w=x^{\ell}$ be a power word. Let $H_{n,\ell}:=\U_{\left\lfloor n/\ell\right\rfloor +1}^{j}\times\U_{\left\lfloor n/\ell\right\rfloor }^{\ell-j}$,
where $j:=n\mod\ell$. Denote by $\widetilde{\ell}:=\min(n,\ell)$. 
\begin{prop}
\label{prop:reduction to multiplicities}For every $n\geq2$, every
$\ell\geq1$ and every $\rho\in\Irr(\U_{n})$, we have: 
\[
\E_{X\in\U_{n}}(\rho(X^{\ell}))=\dim\rho^{H_{n,\ell}}.
\]
\end{prop}

\begin{proof}
This follows from the works of Rains \cite{Rai97,Rai03}. More precisely,
in \cite[Theorem 1.3]{Rai03} Rains showed that the eigenvalues of
$X^{\ell}$ for a random $X\in\U_{n}$ have the same distribution
as the union of the spectra of $\widetilde{\ell}$ independent Haar
random matrices in $\U_{\left\lfloor n/\ell\right\rfloor },\ldots,\U_{\left\lfloor n/\ell\right\rfloor },\U_{\left\lfloor n/\ell\right\rfloor +1},\ldots,\U_{\left\lfloor n/\ell\right\rfloor +1}$.
Since the distribution of $X^{\ell}$ is conjugate invariant, it is
determined by the distribution of eigenvalues. We therefore have:
\begin{equation}
\E_{X\in\U_{n}}(\rho(X^{\ell}))=\E_{X\in H_{n,\ell}}(\rho(X))=\dim\rho^{H_{n,\ell}}.\qedhere\label{eq:Fourier coefficients of power words}
\end{equation}
\end{proof}
\begin{prop}
\label{prop:Fourier coefficient of power words}Let $\ell\in\N$.
For every $\delta>0$ there exists $N(\delta)\in\N$ such that, for
every $n>N(\delta)$ and every $\rho\in\mathrm{Irr}(\U_{n})$, 
\[
\left|\mathbb{E}_{X\in\U_{n}}(\rho(X^{\ell}))\right|=\dim\rho^{H_{n,\ell}}\leq\rho(1)^{1-\frac{1}{\ell-1}+\delta}.
\]
\end{prop}

\begin{proof}
We may assume $\ell=\widetilde{\ell}$. We prove the claim by induction
on $\ell$. For the base of the induction, $\ell=2$, we have $\dim\rho^{H_{n,2}}\leq1$
since $(\U_{n},\U_{\lceil n/2\rceil}\times\U_{\left\lfloor n/2\right\rfloor })$
is a Gelfand pair (because $(\mathrm{GL}_{n}(\C),\mathrm{GL}_{\lceil n/2\rceil}(\C)\times\mathrm{GL}_{\left\lfloor n/2\right\rfloor }(\C))$
is a spherical pair). For $\ell=1$ we have $\mathbb{E}_{X\in\U_{n}}(\rho(X))=0=\rho(1)^{-\infty}$
for every non-trivial character. Now suppose that $\ell\geq3$ and
that the statement holds for all $\ell'<\ell$. Write $\ell=k+(\ell-k)$
for some $k<\ell$. Set $\widetilde{H}_{n,a,b}:=\U_{\left\lfloor n/\ell\right\rfloor +1}^{a}\times\U_{\left\lfloor n/\ell\right\rfloor }^{b}$,
and $\widetilde{L}_{k}:=\widetilde{H}_{n,k,0}$ if $k\leq j$ and
$\widetilde{L}_{k}:=\widetilde{H}_{n,j,k-j}$ otherwise. Similarly,
set $\widetilde{R}_{k}:=\widetilde{H}_{n,j-k,\ell-j}$ if $k\leq j$
and $\widetilde{R}_{k}:=\widetilde{H}_{n,0,\ell-k}$ otherwise. Set
$n_{L}:=k\left\lfloor n/\ell\right\rfloor +\min(k,j)$ and $n_{R}=(\ell-k)\left\lfloor n/\ell\right\rfloor +\max(j-k,0)$.

For $\sigma\in\mathrm{Irr}(\U_{n_{L}})$ and $\sigma'\in\mathrm{Irr}(\U_{n_{R}})$,
we denote: 
\[
m_{\sigma}:=\dim\sigma^{\widetilde{L}_{k}},\text{ \,\,\,and\,\,\,\,}m_{\sigma'}:=\dim\sigma'^{\widetilde{R}_{k}}.
\]
By the induction hypothesis, with $(n,\ell)$ replaced by $(n_{L},k)$
(resp. $(n_{R},\ell-k)$), we have $m_{\sigma}\leq(\dim\sigma)^{1-\min(1,\frac{1}{k-1})+\delta_{L}}$
and $m_{\sigma'}\leq(\dim\sigma')^{1-\min(1,\frac{1}{\ell-k-1})+\delta_{R}}$,
for $\delta_{L},\delta_{R}>0$ arbitrary small. Using Frobenius reciprocity,
we obtain:
\begin{equation}
\begin{split}\langle\rho,\C\rangle_{H_{n,\ell}} & =\langle\rho,\mathrm{Ind}_{\U_{n_{L}}\times\U_{n_{R}}}^{\U_{n}}(\mathrm{Ind}_{\widetilde{L}_{k}}^{\U_{n_{L}}}(\C)\boxtimes\mathrm{Ind}_{\widetilde{R}_{k}}^{\U_{n_{R}}}(\C))\rangle_{\U_{n}}\\
 & =\sum_{\sigma\hookrightarrow\mathrm{Ind}_{\widetilde{L}_{k}}^{\U_{n_{L}}}(\C),\sigma'\hookrightarrow\mathrm{Ind}_{\widetilde{R}_{k}}^{\U_{n_{R}}}(\C)}m_{\sigma}m_{\sigma'}\langle\rho,\mathrm{Ind}_{\U_{n_{L}}\times\U_{n_{R}}}^{\U_{n}}(\sigma\boxtimes\sigma')\rangle_{\U_{n}}\\
 & \leq\sum_{\sigma\hookrightarrow\mathrm{Ind}_{\widetilde{L}_{k}}^{\U_{n_{L}}}(\C),\sigma'\hookrightarrow\mathrm{Ind}_{\widetilde{R}_{k}}^{\U_{n_{R}}}(\C)}(\dim\sigma)^{1-\min(1,\frac{1}{k-1})+\delta_{L}}(\dim\sigma')^{1-\min(1,\frac{1}{\ell-k-1})+\delta_{R}}m_{\rho,\sigma\boxtimes\sigma'},
\end{split}
\label{eq:wasstar}
\end{equation}
where $m_{\rho,\sigma\boxtimes\sigma'}:=\langle\rho,\mathrm{Ind}_{\U_{n_{L}}\times\U_{n_{R}}}^{\U_{n}}(\sigma\boxtimes\sigma')\rangle_{\U_{n}}$.
Note that by Proposition \ref{prop:bounds}(2), the number $N_{k}$
of non-zero summands in (\ref{eq:wasstar}) is at most $\rho(1)^{\delta/2}$
for $n\gg_{\delta}1$. Let $2>\beta(\ell)>1$ be a constant to be
determined later. Then using Jensen's inequality:
\begin{equation}
\langle\rho,\C\rangle_{H_{n,\ell}}^{\beta(\ell)}\leq N_{k}^{\beta(\ell)-1}\sum_{\substack{\sigma\hookrightarrow\mathrm{Ind}_{\widetilde{L}_{k}}^{\U_{n_{L}}}(\C)\\
\sigma'\hookrightarrow\mathrm{Ind}_{\widetilde{R}_{k}}^{\U_{n_{R}}}(\C)
}
}(\dim\sigma)^{\left(1-\min(1,\frac{1}{k-1})+\delta_{L}\right)\beta(\ell)}(\dim\sigma')^{\left(1-\min(1,\frac{1}{\ell-k-1})+\delta_{R}\right)\beta(\ell)}m_{\rho,\sigma\boxtimes\sigma'}^{\beta(\ell)}.\label{eq:formula for multiplicity}
\end{equation}
We take 
\begin{equation}
\beta(\ell):=\min\left\{ \frac{2}{\left(2-\min(1,\frac{1}{k-1})+\delta_{L}\right)},\frac{2}{\left(2-\min(1,\frac{1}{\ell-k-1})+\delta_{R}\right)}\right\} ,\label{eq:beta_l}
\end{equation}
so that $\left(1-\min(1,\frac{1}{k-1})+\delta_{L}\right)\beta(\ell),\left(1-\min(1,\frac{1}{\ell-k-1})+\delta_{R}\right)\beta(\ell)\leq2-\beta(\ell)$.
Combining (\ref{eq:formula for multiplicity}) and (\ref{eq:beta_l})
with the fact that $\left(m_{\rho,\sigma\boxtimes\sigma'}\right)^{\beta(\ell)-1}\leq(\dim\sigma)^{\beta(\ell)-1}(\dim\sigma')^{\beta(\ell)-1},$which
holds by Proposition \ref{prop:bounds}, we deduce: 
\[
\langle\rho,\C\rangle_{H_{n,\ell}}^{\beta(\ell)}\leq N_{k}^{\beta(\ell)-1}\sum_{\sigma\hookrightarrow\mathrm{Ind}_{\widetilde{L}_{k}}^{\U_{n_{L}}}(\C),\sigma'\hookrightarrow\mathrm{Ind}_{\widetilde{R}_{k}}^{\U_{n_{R}}}(\C)}(\dim\sigma)(\dim\sigma')m_{\rho,\sigma\boxtimes\sigma'}\leq N_{k}^{\beta(\ell)-1}\rho(1)\leq\rho(1)^{1+\delta/2}.
\]
In particular, we get, 
\begin{align*}
\langle\rho,\C\rangle_{H_{n,\ell}} & \leq\rho(1)^{\frac{1}{\beta(\ell)}+\frac{\delta}{2\beta(\ell)}}\leq\rho(1)^{\max\left\{ 1-\min(\frac{1}{2},\frac{1}{2(k-1)})+\frac{\delta_{L}}{2},1-\min(\frac{1}{2},\frac{1}{2(\ell-k-1)})+\frac{\delta_{R}}{2}\right\} +\frac{\delta}{2}}\\
 & \leq\rho(1)^{1-\min\left\{ \frac{1}{2},\frac{1}{2(k-1)},\frac{1}{2(\ell-k-1)}\right\} +\max\left\{ \frac{\delta_{L}}{2},\frac{\delta_{R}}{2}\right\} +\frac{\delta}{2}}.
\end{align*}
Taking $k=\left\lfloor \ell/2\right\rfloor \geq\frac{\ell}{2}-\frac{1}{2}$
we have $\min\left\{ \frac{1}{2},\frac{1}{2(k-1)},\frac{1}{2(\ell-k-1)}\right\} \geq\frac{1}{\ell-1}$,
and hence 
\[
\langle\rho,\C\rangle_{H_{n,\ell}}\leq\rho(1)^{1-\frac{1}{\ell-1}+\max\left\{ \frac{\delta_{L}}{2},\frac{\delta_{R}}{2}\right\} +\frac{\delta}{2}}.
\]
Taking $\delta_{R},\delta_{L}<\delta$, we deduce the proposition. 
\end{proof}

\subsection{\label{subsec:Optimality-of-the}Optimality of the bounds in Theorem
\ref{thm: A, bounded density} }

The following proposition shows that the dependence of $t(w)$ on
$\ell(w)$ in Item (2) of Theorem \ref{thm: A, bounded density} is
at least linear. 
\begin{prop}
\label{prop:Theorem A is optimal}Let $\ell\geq2$ and let $w_{\ell}=x^{\ell}$.
Then for every $n\gg_{\ell}1$, we have $\tau_{w_{\ell},\SU_{n}}^{*\ell-1}\notin L^{\infty}(\SU_{n})$.
\end{prop}

\begin{proof}
It is enough to show the claim for $\U_{n}$, namely, that $\tau_{w_{\ell},\U_{n}}^{*\ell-1}\notin L^{\infty}(\U_{n})$.
Indeed, let $\mathrm{m}:\SU_{n}\times\U_{1}\rightarrow\U_{n}$ be
the multiplication map and note that $\mu_{\U_{n}}=\mathrm{m}_{*}(\mu_{\SU_{n}}\times\mu_{\U_{1}})$.
Since $\mathrm{m}\circ((w_{\ell}^{*\ell-1})_{\SU_{n}},(w_{\ell}^{*\ell-1})_{\U_{1}})=(w_{\ell}^{*\ell-1})_{\U_{n}}\circ\mathrm{m}^{\ell-1}$
and since $\mathrm{m}$, $\mathrm{m}^{\ell-1}$, and $(w_{\ell}^{*\ell-1})_{\U_{1}}$
are all local diffeomorphisms, $\tau_{w_{\ell},\SU_{n}}^{*\ell-1}\in L^{\infty}(\SU_{n})$
implies $\tau_{w_{\ell},\U_{n}}^{*\ell-1}\in L^{\infty}(\U_{n})$.

For every $\delta>0$ and every closed subgroup $H\leq\U_{n}$, and
$g\in\U_{n}$, let 
\[
B_{H}(g,\delta):=\left\{ h\in H:\left\Vert h-g\right\Vert _{\mathrm{HS}}<\delta\right\} ,
\]
and let $\chi_{H,\delta}:H\rightarrow\R$ be the characteristic function
of $B_{H}(I_{n},\delta)$. Since $\chi_{\U_{n},\delta}$ is a conjugate
invariant function, it can be written as $\chi_{\U_{n},\delta}=\sum_{\rho\in\Irr(\U_{n})}b_{\rho}\rho$
for $b_{\rho}\in\C$, and therefore by (\ref{eq:Fourier coefficients of power words}),
\begin{align*}
\tau_{w_{\ell},\U_{n}}(B_{\U_{n}}(I_{n},\delta)) & =\int_{\U_{n}}\chi_{\U_{n},\delta}(g^{\ell})\mu_{\U_{n}}=\sum_{\rho\in\Irr(\U_{n})}b_{\rho}\int_{\U_{n}}\rho(g^{\ell})\mu_{\U_{n}}=\sum_{\rho\in\Irr(\U_{n})}b_{\rho}\int_{H_{n,\ell}}\rho(h)\mu_{H_{n,\ell}}\\
 & =\int_{H_{n,\ell}}\chi_{\U_{n},\delta}(h)\mu_{H_{n,\ell}}=\int_{H_{n,\ell}}\chi_{H_{n,\ell},\delta}(h)\mu_{H_{n,\ell}}=\mu_{H_{n,\ell}}(B_{H_{n,\ell}}(I_{n},\delta)).
\end{align*}
For every $a,b\geq0$, $B_{\U_{n}}(I_{n},a)\cdot B_{\U_{n}}(I_{n},b)\subseteq B_{\U_{n}}(I_{n},a+b)$,
so 
\[
\underbrace{B_{\U_{n}}\left(I_{n},\frac{\delta}{\ell-1}\right)\cdots B_{\U_{n}}\left(I_{n},\frac{\delta}{\ell-1}\right)}_{\ell-1\text{ times}}\subseteq B_{\U_{n}}\left(I_{n},\delta\right).
\]
Thus, for every $\delta>0$,
\[
\tau_{w_{\ell},\U_{n}}^{*(\ell-1)}\left(B_{\U_{n}}(I_{n},\delta)\right)\geq\left(\tau_{w_{\ell},\U_{n}}\left(B_{\U_{n}}\left(I_{n},\frac{\delta}{\ell-1}\right)\right)\right)^{\ell-1}=\Big(\mu_{H_{n,\ell}}\left(B_{H_{n,\ell}}(I_{n},\frac{\delta}{\ell-1})\right)\Big)^{\ell-1}
\]
For every closed subgroup $H\leq\U_{n}$ there are constants $c_{H},C_{H}>0$
such that $c_{H}\leq\frac{\mu_{H}(B_{H}(I_{n},\delta))}{\delta^{\dim H}}\leq C_{H}$,
for every $0<\delta\leq1$. Thus, there is a constant $C_{n,\ell}$
such that
\[
\frac{\tau_{w_{\ell},\U_{n}}^{*(\ell-1)}\left(B_{\U_{n}}(I_{n},\delta)\right)}{\mu_{\U_{n}}(B_{\U_{n}}(I_{n},\delta))}\geq\frac{\left(\mu_{H_{n,\ell}}\left(B_{H_{n,\ell}}(I_{n},\frac{\delta}{\ell-1})\right)\right)^{\ell-1}}{\mu_{\U_{n}}(B_{\U_{n}}(I_{n},\delta))}\geq C_{n,\ell}\delta^{(\ell-1)\dim H_{n,\ell}-\dim\U_{n}}.
\]
Taking $\delta\rightarrow0$, since $(\ell-1)\dim H_{n,\ell}<\dim\U_{n}$
whenever $n\gg_{\ell}1$, the density of $\tau_{w_{\ell},\U_{n}}^{*(\ell-1)}$
is unbounded. 
\end{proof}

\section{\label{sec:Fourier-estimates-for small reps}Fourier estimates for
small representations}

In this section we prove Theorem \ref{thm C:small Fourier coefficients}.
Our main tool is the \emph{Weingarten calculus}, which we describe
next. We refer to \cite{CMN22} for a beautiful survey on the topic.

Given a permutation $\sigma\in S_{m}$, let $\cyc(\sigma)$ be the
number of cycles of $\sigma$. Given two natural numbers $m$ and
$m\leq n$, the group $S_{m}$ acts on $\left(\mathbb{C}^{n}\right)^{\otimes m}$,
by $\sigma.\left(v_{1}\otimes\cdots\otimes v_{m}\right)=v_{\sigma(1)}\otimes\cdots\otimes v_{\sigma(m)}$,
and the function $\sigma\mapsto n^{\cyc(\sigma)}$ is the character
of this representation. By the Schur\textendash Weyl duality \cite{Wey39},
the irreducible representation of $S_{m}$ with character $\chi_{\lambda}$
appears in $\left(\mathbb{C}^{n}\right)^{\otimes m}$ with (non-zero)
multiplicity $\rho_{\lambda,n}(1)$, and, therefore, the function
$\sigma\mapsto n^{\cyc(\sigma)}$ is an invertible element of the
group algebra $\mathbb{C}[S_{m}]$ (with respect to the convolution
operation). This gives rise to the following definition: 
\begin{defn}[{{{{{\cite[Proposition 2.3]{CS06}}}}}}]
\label{def:The-Weingarten-function}Let $m\leq n$. The \emph{Weingarten
function} $\Wg_{m,n}:S_{m}\rightarrow\mathbb{C}$ is the inverse of
the function $\sigma\mapsto n^{\cyc(\sigma)}$ in $\mathbb{C}[S_{m}]$.
Its Fourier expansion is given by 
\begin{equation}
\Wg_{m,n}(\sigma)=\frac{1}{m!^{2}}\sum_{\lambda\vdash m}\frac{\chi_{\lambda}(1)^{2}}{\rho_{\lambda}(1)}\chi_{\lambda}(\sigma).\label{eq:Weingarten function}
\end{equation}
\end{defn}

\begin{thm}[{{{{{\cite[Theorem 4.4]{CMN22}}}}}}]
\label{thm:Weingarten} Let $m\leq n$ be natural numbers, let $F_{1},F_{2},H_{1},H_{2}:[m]\rightarrow[n]$
and let $\sX$ be a $\U_{n}$-valued random variable distributed according
to the Haar measure. Then, 
\[
\mathbb{E}\left(\prod_{i\in[m]}\sX_{F_{1}(i),F_{2}(i)}\cdot\prod_{i\in[m]}\overline{\sX_{H_{1}(i),H_{2}(i)}}\right)=\sum_{\substack{\sigma,\tau\in S_{m}\\
F_{1}\circ\sigma=H_{1}\\
F_{2}\circ\tau=H_{2}
}
}\Wg_{m,n}(\sigma\tau^{-1}).
\]
\end{thm}

The following lemma follows from \cite[Theorem 1.1]{CM17}: 
\begin{lem}
\label{lem:Weingarten.bounds}If $(8m)^{7/4}\leq n$, then, for every
$\sigma\in S_{m}$, $\left|\Wg_{m,n}(\sigma)\right|\leq\Wg_{m}(1)\leq\frac{2}{n^{m}}$. 
\end{lem}

\begin{lem}
\label{lem:moments.trace}Let $w\in F_{r}$ be a non-trivial word
of length $\ell$, let $\sX_{1},\ldots,\sX_{r}$ be independent $\U_{n}$-valued
random variables, each distributed according to the Haar measure,
and let $(8M\ell)^{7/4}\leq n$. Then 
\[
\mathbb{E}\left|\tr w(\sX_{1},\ldots,\sX_{r})\right|^{2M}\leq2^{\ell}(M\ell)!^{2}.
\]
\end{lem}

\begin{proof}
Write $w$ as a cyclically reduced word $w=w_{1}\cdots w_{\ell}$,
where, for each $i$, $w_{i}$ is either $x_{j}$ or $x_{j}^{-1}$,
for some $j$. Denote the cyclic group of order $\ell$ by $C_{\ell}$
and, for each natural number $i$, let $[i]=\left\{ 1,\ldots,i\right\} $.
Finally, for $i\in[r]$, let $A_{i}=\left\{ t\in C_{\ell}\mid w_{t}=x_{i}\right\} $
and $B_{i}=\left\{ t\in C_{\ell}\mid w_{t}=x_{i}^{-1}\right\} $.
Let $A:=\bigcup_{i=1}^{r}A_{i}$ and $B:=\bigcup_{i=1}^{r}B_{i}$.
Denoting 
\[
\sW_{i}=\begin{cases}
\sX_{j} & w_{i}=x_{j}\\
\sX_{j}^{-1} & w_{i}=x_{j}^{-1}
\end{cases},
\]
we have $w(\sX_{1},\ldots,\sX_{r})=\sW_{1}\cdots\sW_{\ell}$, so 
\[
\tr w(\sX_{1},\ldots,\sX_{r})=\sum_{i_{1},\ldots,i_{\ell}\in[n]}(\sW_{1})_{i_{1},i_{2}}\cdots(\sW_{\ell})_{i_{\ell},i_{1}}=\sum_{\substack{F:C_{\ell}\times[2]\rightarrow[n]\\
F(j,2)=F(j+1,1)
}
}\prod_{j\in C_{\ell}}(\sW_{j})_{F(j,1),F(j,2)}.
\]
Therefore 
\[
\left|\tr w(\sX_{1},\ldots,\sX_{r})\right|^{2}=\sum_{\substack{F:C_{\ell}\times[4]\rightarrow[n]\\
F(j,2)=F(j+1,1)\\
F(j,4)=F(j+1,3)
}
}\prod_{j\in C_{\ell}}(\sW_{j})_{F(j,1),F(j,2)}\overline{(\sW_{j})_{F(j,3),F(j,4)}},
\]
and 
\[
\left|\tr w(\sX_{1},\ldots,\sX_{r})\right|^{2M}=\sum_{\substack{F:C_{\ell}\times[M]\times[4]\rightarrow[n]\\
F(j,t,2)=F(j+1,t,1)\\
F(j,t,4)=F(j+1,t,3)
}
}\prod_{(j,t)\in C_{\ell}\times[M]}(\sW_{j})_{F(j,t,1),F(j,t,2)}\overline{(\sW_{j})_{F(j,t,3),F(j,t,4)}}.
\]
Let $T:C_{\ell}\times[M]\times[4]\rightarrow C_{\ell}\times[M]\times[4]$
be the involution that switches between $(j+1,t,1)$ and $(j,t,2)$
and between $(j+1,t,3)$ and $(j,t,4)$, for all $j,t$. We have 
\begin{align*}
\left|\tr w(\sX_{1},\ldots,\sX_{r})\right|^{2M} & =\sum_{\substack{F:C_{\ell}\times[M]\times[4]\rightarrow[n]\\
F\circ T=F
}
}\prod_{(j,t)\in C_{\ell}\times[M]}(\sW_{j})_{F(j,t,1),F(j,t,2)}\overline{(\sW_{j})_{F(j,t,3),F(j,t,4)}}\\
 & =\sum_{\substack{F:C_{\ell}\times[M]\times[4]\rightarrow[n]\\
F\circ T=F
}
}\prod_{i=1}^{r}\left(\prod_{(j,t)\in A_{i}\times[M]}(\sX_{i})_{F(j,t,1),F(j,t,2)}\cdot\prod_{(j,t)\in B_{i}\times[M]}(\sX_{i})_{F(j,t,4),F(j,t,3)}\cdot\right.\\
 & \left.\cdot\prod_{(j,t)\in A_{i}\times[M]}\overline{(\sX_{i})_{F(j,t,3),F(j,t,4)}}\cdot\prod_{(j,t)\in B_{i}\times[M]}\overline{(\sX_{i})_{F(j,t,2),F(j,t,1)}}\right).
\end{align*}
For $k,k'\in[4]$, let $p_{k,k'}:(A\times[M]\times\left\{ k\right\} )\cup(B\times[M]\times\left\{ k'\right\} )\rightarrow C_{\ell}\times[M]$
be the map induced from the projection $C_{\ell}\times[M]\times[4]\rightarrow C_{\ell}\times[M]$.
Then 
\begin{align*}
 & \left|\tr w(\sX_{1},\ldots,\sX_{r})\right|^{2M}\\
= & \sum_{\substack{F:C_{\ell}\times[M]\times[4]\rightarrow[n]\\
F\circ T=F
}
}\prod_{i=1}^{r}\left(\prod_{(j,t)\in(A_{i}\cup B_{i})\times[M]}(\sX_{i})_{F\circ p_{1,4}^{-1}(j,t),F\circ p_{2,3}^{-1}(j,t)}\cdot\prod_{(j,t)\in(A_{i}\cup B_{i})\times[M]}\overline{(\sX_{i})_{F\circ p_{3,2}^{-1}(j,t),F\circ p_{4,1}^{-1}(j,t)}}\right).
\end{align*}
Applying Theorem \ref{thm:Weingarten} for each $i\in[r]$, and Lemma
\ref{lem:Weingarten.bounds}, 
\begin{align}
\mathbb{E}\left|\tr w(\sX_{1},\ldots,\sX_{r})\right|^{2M} & =\sum_{\substack{F:C_{\ell}\times[M]\times[4]\rightarrow[n]\\
F\circ T=F
}
}\prod_{i=1}^{r}\Biggl(\sum_{\substack{\sigma_{i},\tau_{i}\in\Sym\left((A_{i}\cup B_{i})\times[M]\right)\\
F\circ p_{1,4}^{-1}\circ\sigma_{i}=F\circ p_{3,2}^{-1}\\
F\circ p_{2,3}^{-1}\circ\tau_{i}=F\circ p_{4,1}^{-1}
}
}\Wg_{|A_{i}|+|B_{i}|,n}(\sigma_{i}\circ\tau_{i}^{-1})\Biggl)\nonumber \\
 & \leq\frac{2^{r}}{n^{M\ell}}\sum_{\substack{F:C_{\ell}\times[M]\times[4]\rightarrow[n]\\
F\circ T=F
}
}\prod_{i=1}^{r}\Biggl(\sum_{\substack{\sigma_{i},\tau_{i}\in\Sym\left((A_{i}\cup B_{i})\times[M]\right)\\
F\circ p_{1,4}^{-1}\circ\sigma_{i}=F\circ p_{3,2}^{-1}\\
F\circ p_{2,3}^{-1}\circ\tau_{i}=F\circ p_{4,1}^{-1}
}
}1\Biggl)\nonumber \\
 & =\frac{2^{r}}{n^{M\ell}}\sum_{\substack{\sigma_{1},\tau_{1},\ldots,\sigma_{r},\tau_{r}\\
\sigma_{i},\tau_{i}\in\Sym\left((A_{i}\cup B_{i})\times[M]\right)
}
}\left|\left\{ F:C_{\ell}\times[M]\times[4]\rightarrow[n]\mid\substack{F\circ T=F\\
F\circ p_{1,4}^{-1}\circ\sigma_{i}=F\circ p_{3,2}^{-1}\\
F\circ p_{2,3}^{-1}\circ\tau_{i}=F\circ p_{4,1}^{-1}
}
\right\} \right|.\label{eq:Weingarten.final.estimate}
\end{align}
Since $\sigma_{i}$ have disjoint supports, the number of tuples $\sigma_{1},\ldots,\sigma_{r}$
is at most $(M\ell)!$. Similarly, the number of tuples $\tau_{1},\ldots,\tau_{r}$
is at most $(M\ell)!$.

Fix $\sigma_{1},\ldots,\sigma_{r},\tau_{1},\ldots,\tau_{r}$. Denote
$A=\cup A_{i},B=\cup B_{i}$. The union of the graphs of $p_{1,4}^{-1}\circ\sigma_{i}\circ p_{3,2}$,
$i\in[r]$ is a perfect matching between 
\[
(A\times[M])\times\left\{ 1\right\} \cup(B\times[M])\times\left\{ 4\right\} \text{ and }(A\times[M])\times\left\{ 3\right\} \cup(B\times[M])\times\left\{ 2\right\} .
\]
Similarly, the union of the graphs of $p_{2,3}^{-1}\circ\sigma_{i}\circ p_{4,1}$,
$i\in[k]$ is a perfect matching between 
\[
(A\times[M])\times\left\{ 2\right\} \cup(B\times[M])\times\left\{ 3\right\} \text{ and }(A\times[M])\times\left\{ 4\right\} \cup(B\times[M])\times\left\{ 1\right\} .
\]
The union of these matchings is an involution $S:C_{\ell}\times[M]\times[4]\rightarrow C_{\ell}\times[M]\times[4]$.

Neither $S$ nor $T$ has fixed points. We claim that $S\circ T$
does not have a fixed point either. For example, if $j\in A_{i}$
and $t\in[M]$, then 
\[
S(j,t,1)\in A_{i}\times[M]\times\left\{ 3\right\} \cup B_{i}\times[M]\times\left\{ 2\right\} ,
\]
since $T(j,t,1)=(j-1,t,2)$, the only possibility that $T(j,t,1)=S(j,t,1)$
is when $j-1\in B_{i}$. But this will mean that $w_{j-1}=x_{i}^{-1}$
whereas the assumption that $j\in A_{i}$ means that $w_{j}=x_{i}$,
contradicting the assumption that $w_{1}\cdots w_{\ell}$ is cyclically
reduced. A similar argument holds if $j\in B_{i}$.

The set 
\[
\left\{ F:C_{\ell}\times[M]\times[4]\rightarrow[n]\mid\substack{F\circ T=F\\
F\circ p_{1,4}^{-1}\circ\sigma_{i}=F\circ p_{3,2}^{-1}\\
F\circ p_{2,3}^{-1}\circ\tau_{i}=F\circ p_{4,1}^{-1}
}
\right\} 
\]
consists of all functions $F:C_{\ell}\times[M]\times[4]\rightarrow[n]$
that are both $T$ and $S$ invariant. By the claims above, each orbit
of $\langle T,S\rangle$ has size at least 4, so the size of this
set is bounded by $n^{M\ell}$. The result now follows from (\ref{eq:Weingarten.final.estimate}). 
\end{proof}
We now prove Theorem \ref{thm C:small Fourier coefficients}. 
\begin{proof}[Proof of Theorem \ref{thm C:small Fourier coefficients}]
Note that by (\ref{eq:Fourier coefficients of convolution}), for
any $w\in F_{r}$, the Fourier coefficients $a_{w*w^{-1},\U_{n},\rho}$
of $w*w^{-1}$ for $\rho\in\mathrm{Irr}(\U_{n})$ are positive, and
$a_{w*w^{-1},\U_{n},\rho}=\frac{\left|a_{w,\U_{n},\rho}\right|^{2}}{\rho(1)}$.
Hence, by replacing $w$ with $w*w^{-1}$, we may assume that $w$
has positive Fourier coefficients.

Let $\lambda\in\Lambda_{n}$, and let $\rho_{\lambda,n}\in\Irr(\U_{n})$.
By Definition \ref{def:first and last fundamental weights}, there
exists partitions $\mu\vdash m_{1}$ and $\nu\vdash m_{2}$, such
that $\mu:=\lambda_{-}$ and $\nu:=\lambda_{+}^{\vee,n}$, and such
that $\rho_{\lambda,n}\hookrightarrow\rho_{\mu,n}\otimes\left(\rho_{\nu,n}^{\vee}\cdot(\det)^{t}\right)$
for some $t\in\Z$. Moreover, by Remark \ref{rem:useful fact}, we
have $\rho_{\mu,n}(1),\rho_{\nu,n}(1)\leq\rho_{\lambda,n}(1)$. By
the positivity of Fourier coefficients of $w$, we have: 
\begin{align}
\left(\E(\rho_{\lambda,n}w(\sX_{1},\ldots,\sX_{r}))\right)^{2} & \le\left(\E\left(\left(\rho_{\mu,n}\otimes\left(\rho_{\nu,n}^{\vee}\cdot(\det)^{t}\right)\right)w(\sX_{1},\ldots,\sX_{r})\right)\right)^{2}\nonumber \\
 & \leq\E\left(\left|\rho_{\mu,n}w(\sX_{1},\ldots,\sX_{r})\right|^{2}\right)\E\left(\left|\rho_{\nu,n}w(\sX_{1},\ldots,\sX_{r})\right|^{2}\right).\label{eq:redution to partitions}
\end{align}
Hence, in order to prove Theorem \ref{thm C:small Fourier coefficients},
it is enough to show there exist $C,\alpha>0$ such that: 
\[
\E\left(\left|\rho_{\mu',n}w(\sX_{1},\ldots,\sX_{r})\right|^{2}\right)<\rho_{\mu',n}(1)^{1-\epsilon},
\]
for $\mu'\vdash m$, such that $\rho_{\mu',n}(1)<2^{Cn^{\alpha}}$
and $\ell(\mu')\leq n/2$. Let us take $\alpha=\frac{1}{4\ell+1}$
and $C=\frac{1}{8\ell}$. Then by Lemma \ref{lem:dimension of small representations}(2),
$\rho_{\mu',n}(1)\geq2^{\frac{m}{4}}$ and hence, $m\leq4Cn^{\alpha}\leq\frac{1}{2\ell}n^{\frac{1}{4\ell+1}}$.
In particular, $(2m\ell)^{4\ell}\leq\frac{n}{2m\ell}\leq\frac{n}{m}$,
and therefore, by Lemma \ref{lem:dimension of small representations}(1):
\begin{equation}
2^{\ell}(m\ell)!^{2}\leq(2m\ell)^{2m\ell}\leq\left(\frac{n}{m}\right)^{\frac{m}{2}}<\rho_{\mu',n}(1)^{\frac{1}{2}}.\label{eq:silly.inequality}
\end{equation}
Using the positivity of the Fourier coefficients of $w$, since $\rho_{\mu',n}\hookrightarrow\left(\C^{n}\right)^{\otimes m}$,
we have: 
\begin{equation}
\E\left(\left|\rho_{\mu',n}w(\sX_{1},\ldots,\sX_{r})\right|^{2}\right)\leq\mathbb{E}\left(\left|\tr w(\sX_{1},\ldots,\sX_{r})\right|^{2m}\right).\label{eq:reduction to powers of trace}
\end{equation}
Since the condition $m\leq\frac{1}{2\ell}n^{\frac{1}{4\ell+1}}$ implies
that $(8m\ell)^{7/4}\leq n$, we get by Lemma \ref{lem:moments.trace}
and (\ref{eq:silly.inequality}), 
\begin{equation}
\E\left(\left|\rho_{\mu',n}w(\sX_{1},\ldots,\sX_{r})\right|^{2}\right)<\rho_{\mu',n}(1)^{\frac{1}{2}}.\label{eq:required bounds}
\end{equation}
Applying (\ref{eq:required bounds}) to $\mu'=\mu,\nu$, and by (\ref{eq:redution to partitions}),
we can now finish the proof: 
\begin{align*}
\E(\rho_{\lambda,n}w(\sX_{1},\ldots,\sX_{r})) & \leq\E\left(\left|\rho_{\mu,n}w(\sX_{1},\ldots,\sX_{r})\right|^{2}\right)^{\frac{1}{2}}\E\left(\left|\rho_{\nu,n}w(\sX_{1},\ldots,\sX_{r})\right|^{2}\right)^{\frac{1}{2}}\\
 & \leq\rho_{\mu,n}(1)^{\frac{1}{4}}\rho_{\nu,n}(1)^{\frac{1}{4}}\leq\rho_{\lambda,n}(1)^{\frac{1}{2}}.\qedhere
\end{align*}
\end{proof}

\section{\label{sec:Fourier-estimates-for symmetric powers}Fourier estimates
for the symmetric powers}

In this section we prove Theorem \ref{thm:main thm symmetric powers}.
Our first step in the proof will be to use the character bounds of
Proposition \ref{prop:character bound for symmetric} to obtain Fourier
estimates for $h_{m,n}$ for $m\geq An$ for some $A>1$. 
\begin{prop}
\label{prop:proof for large symmetric powers}For every $w\in F_{r}$,
there exist $\epsilon,A>0$, such that for every $n\geq2$, and $m\geq An$:
\[
\mathbb{E}\left(\left|h_{m,n}(w(\sX_{1},...,\sX_{r}))\right|^{2}\right)\leq\binom{n+m-1}{m}^{2(1-\epsilon)}.
\]
\end{prop}

\begin{proof}
For every fixed $n\geq2$ the proposition follows from \cite[Proposition 7.2]{AG},
so we may assume $n\gg1$. Let $\ell\geq2$ be the length of $w$.
Taking $A$ large enough, and using (\ref{eq:Striling on binomial coefficients}),
we may assume $2^{\frac{200\ell}{c}n}<A^{n-1}<h_{m,n}(1)$ for any
fixed $c>0$. By Theorem \ref{thm:probabilistic result}, taking $\epsilon'=h_{m,n}(1)^{-\frac{c}{n}}$,
and $\beta=\frac{1}{2}$, we have for $n\gg1$, 
\begin{equation}
\mathbb{P}\left(w(\sX_{1},\ldots,\sX_{r})\text{ fails to be }(\beta,\epsilon')\text{-spread}\right)<2^{3n^{2}}\epsilon'^{\frac{n^{2}}{32\ell}}<h_{m,n}(1)^{\frac{nc}{64\ell}-\frac{nc}{32\ell}}=h_{m,n}(1)^{-\frac{nc}{64\ell}}.\label{eq:probabilistic bounds}
\end{equation}
Moreover, by Proposition \ref{prop:character bound for symmetric},
there exist $A>0$ such that if $n\ge2$, $m\ge An$: 
\begin{equation}
\left|h_{m,n}(g)\right|<h_{m,n}(1)^{1-\frac{1}{8}},\label{eq:character bounds}
\end{equation}
for every $(\beta,\epsilon')$-spread element $g\in\U_{n}$ for suitable
$c>0$. Splitting $\U_{n}^{r}$ into two parts, corresponding to the
event in (\ref{eq:probabilistic bounds}) and its complement, and
using (\ref{eq:probabilistic bounds}) and (\ref{eq:character bounds}),
we get: 
\[
\mathbb{E}\left(\left|h_{m,n}(w(\sX_{1},\ldots,\sX_{r}))\right|^{2}\right)\leq h_{m,n}(1)^{-\frac{nc}{64\ell}}h_{m,n}(1)^{2}+h_{m,n}(1)^{2-\frac{1}{4}}.
\]
Taking $n\gg1$ implies the proposition. 
\end{proof}
To deal with the case when $m\leq An$ we need the following lemma: 
\begin{lem}
\label{lem:Technical lemma}Let $1<a<b$. Then for every $n\geq1$
such that $an,bn$ are integers, we have: 
\[
\binom{an}{n}<\binom{bn}{n}^{1-\frac{\log_{b}(\frac{b}{a})}{1+\log_{b}(e)}}.
\]
\end{lem}

\begin{proof}
Note that 
\[
\bigl(\frac{b}{a}\bigr)^{n}<\frac{(bn)\cdots((b-1)n+1)}{(an)\cdots((a-1)n+1)}=\binom{bn}{n}\binom{an}{n}^{-1}.
\]
By (\ref{eq:Striling on binomial coefficients}), ${bn \choose n}\leq(be)^{n}=b^{(1+\log_{b}(e))n}$.
Hence, 
\[
\binom{an}{n}<\bigl(\frac{a}{b}\bigr)^{n}\binom{bn}{n}=b^{-\log_{b}(\frac{b}{a})n}\binom{bn}{n}=\left(b^{-(1+\log_{b}(e))n}\right)^{\frac{\log_{b}(\frac{b}{a})}{1+\log_{b}(e)}}\binom{bn}{n}\leq\binom{bn}{n}^{1-\frac{\log_{b}(\frac{b}{a})}{1+\log_{b}(e)}}.\qedhere
\]
\end{proof}
We now restate and prove Theorem \ref{thm:main thm symmetric powers}. 
\begin{thm}
\label{thm:proof for symmetric powers}For every $w\in F_{r}$, there
exists $\epsilon(w)>0$, such that for every $n\geq2$ and every $m\geq1$:
\[
\mathbb{E}\left(\left|h_{m,n}(w(\sX_{1},\ldots,\sX_{r}))\right|^{2}\right)\leq\binom{n+m-1}{m}^{2(1-\epsilon(w))}.
\]
\end{thm}

\begin{proof}
By Proposition \ref{prop:proof for large symmetric powers}, it is
enough to prove the claim for $m\leq An$ for any fixed $A>0$. The
case when $m\leq\delta(w)n$, for some $\delta(w)>0$ follows from
\cite[Theorem 1.4]{AG}. To treat the intermediate case $\delta(w)n\leq m\leq An$
we use a similar argument to the one given for the exterior powers
in \cite[Proof of Theorem 7.3, pp.\ 22--23]{AG}. Replacing $w$ by
$w*w^{-1}$, we may assume that all Fourier coefficients of $w$ are
non-negative real numbers. By \cite[Theorem 1.4]{AG}, we can choose
$0<\delta(w)<1$ so that for every $c$ less than the integer $\delta(w)n$,
\[
\mathbb{E}\left(\left|h_{c,n}(w(\sX_{1},\ldots,\sX_{r}))\right|^{2}\right)\leq\binom{n+c-1}{c}.
\]
We claim that there is an isomorphism of $U_{n}$-representations
\begin{equation}
\Sym^{m}(\C^{n})\otimes\left(\Sym^{m}(\C^{n})\right)^{\vee}\cong\bigoplus_{c=0}^{m}V_{c},\label{equation:decomp}
\end{equation}
where $V_{c}$ is the irreducible representation with character $\rho_{c}:=\rho_{(c,0,\ldots,0,-c),n}$.
This follows from the Littlewood-Richardson rule (Proposition~\ref{prop:Littlewood-Richardson-unitary}),
since all strict $(m,m,\ldots,m,0)$-expansions of $(m,0,\ldots,0)$
are obtained by adding, for some $c\in[0,m]$ and for all $i\in[1,m]$,
$c$ boxes labeled '$i-1$' followed by $m-c$ boxes labeled '$i$'
to the $i$-th row, except that we omit $0$'s and $n$'s. This gives
$\rho_{(m+c,m,\ldots,m,m-c),n}$, which is $\det^{m}\rho_{c}$. The
largest irreducible factor in the decomposition (\ref{equation:decomp})
has character $\rho_{m}$, which implies 
\begin{equation}
\rho_{m}(1)\geq\frac{1}{m+1}\binom{n+m-1}{m}^{2}\geq\binom{n+m-1}{m}^{3/2}\label{eq:rho_k ineq}
\end{equation}
for $n\ge4$. The decomposition (\ref{equation:decomp}) gives the
equality 
\begin{equation}
\left|h_{m,n}\right|^{2}=\sum_{c=0}^{m}\rho_{c}.\label{eq:h as sum of rho c}
\end{equation}
This implies $\rho_{m}=\left|h_{m,n}\right|^{2}-\left|h_{m-1,n}\right|^{2}$,
and from this and (\ref{eq:rho_k ineq}), it follows that 
\begin{equation}
\mathbb{E}(\rho_{c}(w(\sX_{1},\ldots,\sX_{r})))\leq\mathbb{E}\left(\left|h_{c,n}(w(\sX_{1},\ldots,\sX_{r}))\right|^{2}\right)\leq\binom{n+c-1}{c}\leq\rho_{c}(1)^{\frac{2}{3}}.\label{eq:bound on rho_c}
\end{equation}
In particular, by (\ref{eq:Fourier coefficients of convolution}),
we have, 
\begin{equation}
\mathbb{E}(\rho_{c}(w^{*9}(\sX_{1},...,\sX_{9r})))=\frac{\mathbb{E}(\rho_{c}(w(\sX_{1},...,\sX_{r})))^{9}}{\rho(1)^{8}}\leq\rho_{c}(1)^{-2}.\label{eq:bounds on convolutions}
\end{equation}
By (\ref{eq:bounds on convolutions}) we have, 
\begin{equation}
\mathbb{E}\left(\left|h_{\delta(w),n}(w^{*9}(\sX_{1},...,\sX_{9r}))\right|^{2}\right)=\sum_{c=0}^{\delta(w)n}\mathbb{E}\left(\rho_{c}\circ w^{*9}\right)\leq\sum_{c=0}^{\delta(w)n}\rho_{c}(1)^{-2}\leq\sum_{j=1}^{\infty}\frac{1}{j^{2}}<2.\label{eq:bound.on.small.wedge.power}
\end{equation}

For any $\delta(w)n\leq m\leq An$, we have an embedding $\Sym^{m}(\C^{n})\hookrightarrow\Sym^{m-\delta(w)n}(\C^{n})\otimes\Sym^{\delta(w)n}(\C^{n})$.
Moreover, since $\delta(w)(n-1)\leq m<An\leq2A(n-1)$, we get: 
\[
\frac{\log_{\frac{n+m-1}{n-1}}(\frac{n+m-1}{n+m-\delta(w)n-1})}{1+\log_{\frac{n+m-1}{n-1}}(e)}\geq\frac{\log_{2A+1}(\frac{A+1}{A+1-\delta(w)})}{1+\log_{1+\delta(w)}(e)}.
\]
Hence, by (\ref{eq:bound.on.small.wedge.power}) and by Lemma \ref{lem:Technical lemma},
taking $\widetilde{\delta}:=\frac{1}{2}\frac{\log_{2A+1}(\frac{A+1}{A+1-\delta(w)})}{1+\log_{1+\delta(w)}(e)}$,
we have for $n\gg_{A,w}1$: 
\begin{equation}
\begin{split}\mathbb{E}\left(\left|h_{m,n}(w^{*9}(\sX_{1},\ldots,\sX_{9r}))\right|^{2}\right) & \leq\mathbb{E}\left(\left|h_{\delta(w)n,n}(w^{*9}(\sX_{1},\ldots,\sX_{9r}))\right|^{2}\left|h_{m-\delta(w)n,n}(w^{*9}(\sX_{1},\ldots,\sX_{9r}))\right|^{2}\right)\\
 & \leq\mathbb{E}\left(\left|h_{\delta(w)n,n}(w^{*9}(\sX_{1},\ldots,\sX_{9r}))\right|^{2}\right)\cdot\binom{n+m-\delta(w)n-1}{n-1}^{2}\\
 & \leq2\binom{n+m-\delta(w)n-1}{n-1}^{2}\\
 & \leq2\binom{n+m-1}{n-1}^{2\biggl(1-\frac{\log_{\frac{n+m-1}{n-1}}\frac{n+m-1}{n+m-1-\delta(w)n}}{1+\log_{\frac{n+m-1}{n-1}}e}\biggr)}\\
 & \leq2h_{m,n}(1)^{2(1-2\widetilde{\delta})}\leq h_{m,n}(1)^{2(1-\widetilde{\delta})}.
\end{split}
\label{eq:delta tilde savings}
\end{equation}
For $n\gg_{A,w}1$, therefore, by (\ref{eq:h as sum of rho c}), H\"older's
inequality, (\ref{eq:bounds on convolutions}), (\ref{eq:bound on rho_c}),
(\ref{eq:delta tilde savings}), (\ref{eq:rho_k ineq}), the subadditivity
of $x^{1/(9-\widetilde{\delta})}$, and the assumptions that $m\ge\delta(w)n$
and $n$ is sufficiently large, 
\begin{align*}
\mathbb{E}\left(\left|h_{m,n}(w(\sX_{1},\ldots,\sX_{r}))\right|^{2}\right) & =\mathbb{E}\left(\sum_{c=0}^{m}\rho_{c}\circ w\right)\\
 & \leq(m+1)^{\frac{8}{9}}\left(\sum_{c=0}^{m}\mathbb{E}\left(\rho_{c}\circ w\right)^{9}\right)^{\frac{1}{9}}=(m+1)^{\frac{8}{9}}\left(\sum_{c=0}^{m}\mathbb{E}\left(\rho_{c}\circ w^{*9}\right)\rho_{c}(1)^{8}\right)^{\frac{1}{9}}\\
 & \leq(m+1)^{\frac{8}{9}}\left(\sum_{c=0}^{m}\mathbb{E}\left(\left|h_{c,n}(w^{*9})\right|^{2}\right)\rho_{c}(1)^{8}\right)^{\frac{1}{9}}\\
 & \leq(m+1)^{\frac{8}{9}}\left(\sum_{c=0}^{m}h_{c,n}(1)^{2(1-\widetilde{\delta})}\rho_{c}(1)^{8}\right)^{\frac{1}{9}}\\
 & \leq(m+1)\left(\sum_{c=0}^{m}\rho_{c}(1)^{9-\widetilde{\delta}}\right)^{\frac{1}{9}}\\
 & \leq(m+1)\left(\sum_{c=0}^{m}\rho_{c}(1)\right)^{1-\frac{\widetilde{\delta}}{9}}<h_{m,n}(1)^{2(1-\frac{\widetilde{\delta}}{10})}.
\end{align*}
All that remains is to deal with finitely many values of $n$. The
theorem therefore follows from \cite[Proposition 7.2]{AG}. 
\end{proof}
\bibliographystyle{alpha}
\bibliography{bibfile}

\end{document}